%% file: _PaperOdeLower.tex
\documentclass{article}
\usepackage[
	backend=biber,
	datamodel=mydatamodel,
	style=alphabetic,
	natbib=true,
	url=false,
	doi=true,
	eprint=true,
	giveninits=true,
	isbn=false,
	maxbibnames=10,
]{biblatex}
\DeclareFieldFormat{mrnumber}{%
	MR\space
	\ifhyperref
	{\href{http://www.ams.org/mathscinet-getitem?mr=#1}{\nolinkurl{#1}}}
	{\nolinkurl{#1}}}
\DeclareFieldFormat{zbl}{%
	Zbl\space
	\ifhyperref
	{\href{https://zbmath.org/?q=an:#1}{\nolinkurl{#1}}}
	{\nolinkurl{#1}}}
\renewbibmacro*{doi+eprint+url}{%
	\iftoggle{bbx:doi}
	{\printfield{doi}}
	{}%
	\newunit\newblock
	\printfield{mrnumber}%
	\newunit\newblock
	\printfield{zbl}%
	\newunit\newblock
	\iftoggle{bbx:eprint}
	{\usebibmacro{eprint}}
	{}%
	\newunit\newblock
	\iftoggle{bbx:url}
	{\usebibmacro{url+urldate}}
	{}}
\usepackage[a4paper, margin=3cm]{geometry}
\usepackage{hyperref}
\usepackage{authblk}
\input{packages}
\usepackage[capitalize,nameinlink,noabbrev]{cleveref}
\input{macros_general}
\input{macros_specific}
\newbox{\myorcidthanksbox}
\sbox{\myorcidthanksbox}{\large\includegraphics[height=1.8ex]{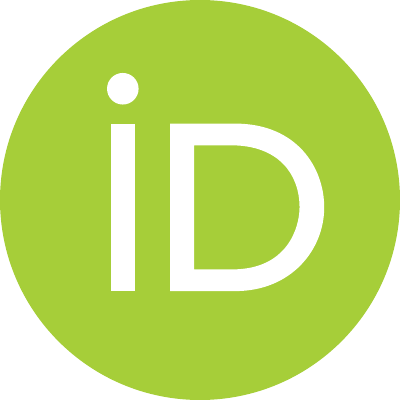}}
\newcommand{\orcidthanks}[1]{%
    \href{https://orcid.org/#1}{\raisebox{-0.5ex}{\usebox{\myorcidthanksbox}}\,#1}}
\usepackage{thmtools}

\input{macros_theorem}
\title{Lower Bounds for Nonparametric Estimation of Ordinary Differential Equations}
\date{}
\author[1,2]{Christof Schötz\thanks{math@christof-schoetz.de, \orcidthanks{0000-0003-3528-4544}}}
\author[3]{Maximilian Siebel\thanks{siebel@math.uni-heidelberg.de, \orcidthanks{0009-0008-7843-1712}}}
\affil[1]{Potsdam Institute for Climate Impact Research}
\affil[2]{Technical University of Munich}
\affil[3]{Heidelberg University}
\begin{document}
\maketitle
\begin{abstract}
	\input{abstract}
\end{abstract}
\tableofcontents
\input{sec_intro.tex}
\input{sec_model.tex}
\input{sec_stubble.tex}
\input{sec_snake.tex}
\begin{appendix}
\input{sec_app_kernel.tex}

	\input{sec_app_master.tex}\input{sec_app_stubble.tex}\input{sec_app_snake.tex}
    \input{sec_app_run.tex}
\end{appendix}
\input{sec_acknowledgments}
\printbibliography
\end{document}

%% file: packages.tex
\usepackage{amsthm,amsmath,amsfonts,amssymb}
\usepackage{mathtools}
\usepackage{stmaryrd}
\usepackage{enumitem} 
\usepackage{graphicx} 
\usepackage{dsfont} 
\usepackage{tikz} 
\usetikzlibrary{calc}
\usepackage{colortbl}
\usepackage{parskip}
\usepackage{graphicx}
\usepackage{caption}
\usepackage{subcaption}
\usepackage{float}

%% file: macros_general.tex
\newcommand*{\mc}[1]{\mathcal{#1}}
\newcommand*{\mb}[1]{\mathbb{#1}}

\newcommand*{\ms}[1]{\mathsf{#1}}
\newcommand*{\mo}[1]{\mathbf{#1}}

\newcommand{\N}{\mathbb{N}}
\newcommand{\Nn}{\mathbb{N}_0}
\newcommand{\Z}{\mathbb{Z}}
\newcommand{\Q}{\mathbb{Q}}
\newcommand{\R}{\mathbb{R}}
\newcommand{\Rp}{\R_{\geq0}}
\newcommand{\Rpp}{\R_{>0}}
\newcommand{\Rpo}{\R_{\geq1}}
\newcommand{\Rgeq}[1]{\R_{\geq#1}}
\newcommand{\Rppo}{\R_{>1}}
\newcommand{\ball}{\mathrm{B}}
\newcommand{\nset}[2]{{\left\llbracket #1,#2\right\rrbracket}}
\newcommand{\nnset}[1]{{\left\llbracket #1\right\rrbracket}}
\newcommand{\nnzset}[1]{{\left\llbracket 0, #1\right\rrbracket}}
\newcommand{\lcb}{\left\lbrace} 
\newcommand{\rcb}{\right\rbrace} 
\newcommand{\cb}[1]{\lcb #1 \rcb} 
\newcommand{\cbOf}[1]{\mathopen{}\lcb #1 \rcb\mathclose{}} 
\newcommand{\lab}{\left[} 
\newcommand{\rab}{\right]} 
\newcommand{\ab}[1]{\lab #1 \rab} 
\newcommand{\abOf}[1]{\!\ab{#1}} 
\newcommand{\lb}{\left(} 
\newcommand{\rb}{\right)} 
\newcommand{\br}[1]{\lb #1 \rb} 
\newcommand{\brOf}[1]{\!\br{#1}} 
\newcommand{\abs}[1]{\left| #1 \right|} 
\newcommand{\normof}[1]{\Vert#1\Vert}

\newcommand{\cardof}[1]{\##1}

\renewcommand{\Pr}{\mathbf{P}} 

\newcommand{\PrOf}[1]{\Pr\brOf{#1}}
\newcommand{\E}{\mathbf{E}} 
\newcommand{\Eof}[1]{\E[#1]}
\newcommand{\Eoff}[2]{\E_{#1}[#2]}

\newcommand{\EOff}[2]{\E_{#1}\abOf{#2}}

\newcommand{\LpNormOf}[3]{\left\Vert#3\right\Vert_{L^{#1}({#2})}}
\newcommand{\supNormof}[1]{\vert#1\vert_\infty}
\newcommand{\supNormOf}[1]{\abs{#1}_\infty}
\newcommand{\opNormof}[1]{\Vert#1\Vert_{\mathsf{op}}}
\newcommand{\opNormOf}[1]{\left\Vert#1\right\Vert_{\mathsf{op}}}
\newcommand{\euclof}[1]{\Vert#1\Vert_2}
\newcommand{\euclOf}[1]{\left\Vert#1\right\Vert_2}

\newcommand{\sizedMid}[2]{#1 \, \kern-\nulldelimiterspace\mathopen{}\left| \vphantom{#1}\,#2\right.\mathclose{}\kern-\nulldelimiterspace}
\newcommand{\setByEle}[2]{\cb{\sizedMid{#1}{#2}}}
\newcommand{\setByEleInText}[2]{\{#1 \mid #2\}}
\newcommand{\tr}{^{\!\top}\!} 
\newcommand{\pr}{^\prime}
\newcommand{\prr}{^{\prime\prime}}

\newcommand{\ind}{\mathds{1}}
\newcommand{\indOfOf}[2]{\ind_{\!#1}\!\brOf{#2}}%
\newcommand{\smoothC}{C^\infty}
\newcommand{\eqcm}{\,,} 
\newcommand{\eqfs}{\,.}
\newcommand{\dl}{\mathrm{d}} 
\newcommand{\supp}{\operatorname{supp}} 
\renewcommand{\subset}{\subseteq}
\newcommand{\euler}{\mathrm{e}} 
\newcommand{\kullback}{\mathrm{K\!L}} 
\newcommand{\lebesgue}{\mathbb{L}} 
\DeclareMathOperator*{\argmin}{arg\,min}

\DeclareMathOperator*{\esssup}{ess\,sup}

\newcommand{\const}[1]{C_{\mathsf{#1}}}
\newcommand{\assuRef}[1]{\texorpdfstring{\protect\hyperlink{assu#1}{\textsc{#1}}}{}}
\newcommand{\newAssuRef}[1]{\hypertarget{assu#1}{\textsc{#1}}}
\newcommand{\indset}[3]{#1_{\nset{#2}{#3}}}

%% file: macros_specific.tex
\newcommand{\noise}{\varepsilon}
\newcommand{\ftrue}{f^\star}
\newcommand{\festi}{\hat f}

\newcommand{\xany}[1]{x_{#1}}
\newcommand{\xdeltaSymb}{x^\delta}
\newcommand{\xdelta}[1]{\xdeltaSymb_{#1}}

\newcommand{\stepsize}{\Delta\!t}

\newcommand{\Tmax}{T_{\mathsf{max}}}
\newcommand{\Tsum}{T_{\mathsf{sum}}}

\newcommand{\nmax}{n_{\mathsf{max}}}
\newcommand{\hypercube}{[0,1]^d}
\newcommand{\xdomain}{\mathcal{X}}

\newcommand{\setIdx}{\mathcal{I}}
\newcommand{\ptrue}{\theta^\star}

\newcommand{\ParamSnakeBase}[2]{\Theta_{#1}^{#2}}
\newcommand{\ParamSnakeSymb}{\Theta^{\mathsf{snake}}}
\newcommand{\ParamSnake}{\ParamSnakeSymb_{d,\beta}}
\newcommand{\FSmooth}{\mathcal{F}_{d,\beta}}
\newcommand{\dm}[1]{d_{\mathsf{#1}}} 

\newcommand{\Kper}{K_{\mathsf{per}}}
\newcommand{\tube}{\mathcal{T}}

\newcommand{\htext}[1]{h^{\ms{#1}}}
\newcommand{\hbump}[2]{\htext{bump}_{#1,#2}}
\newcommand{\hpulse}[2]{\htext{pulse}_{#1,#2}}
\newcommand{\stdKernel}{K^*}
\newcommand{\rmax}{r_{\ms{max}}}
\newcommand{\Lbeta}{L_{\beta}} 

%% file: macros_theorem.tex
\def\postBoxSkip{1.0ex}
\def\postBoxSkipCmd{\vskip\postBoxSkip}
\def\preBoxSkip{1.0ex}
\def\preBoxSkipCmd{\vskip\preBoxSkip}
\declaretheoremstyle[
	bodyfont=\normalfont,
	postfoothook={\postBoxSkipCmd},
	preheadhook={\preBoxSkipCmd},
	mdframed={
		backgroundcolor = black!2,
		startcode={},
}]{ruledBoxStyle}
\declaretheoremstyle[
	bodyfont=\normalfont,
	postfoothook={\postBoxSkipCmd},
	preheadhook={\preBoxSkipCmd},
	mdframed={
		backgroundcolor=white,
}]{ruledBoxStyleWhite}
\declaretheoremstyle[
	bodyfont=\normalfont,
	postfoothook={\postBoxSkipCmd},
	preheadhook={\preBoxSkipCmd},
	mdframed={
		backgroundcolor=black!2,
		linecolor = black!2,
		tikzsetting = {
			draw = black,
			line width = 2pt,%
			dashed,%
			dash pattern = on 10pt off 3pt
		},
}]{dashedBoxStyle}
\declaretheoremstyle[
	bodyfont=\normalfont,
	postfoothook={\postBoxSkipCmd},
	preheadhook={\preBoxSkipCmd},
	mdframed={
		linecolor = white,
		startcode={},
		tikzsetting = {
			draw = black,
			line width = 1pt,%
			loosely dotted,
		},
	}
]{dashedStyle}
\declaretheoremstyle[
	bodyfont=\normalfont,
	postfoothook={\postBoxSkipCmd},
	preheadhook={\preBoxSkipCmd},
	mdframed={
		linecolor = black,
		innerlinewidth=1pt,outerlinewidth=1pt,
		middlelinewidth=1pt,
		linecolor=black,middlelinecolor=white,
		startcode={},
	}
]{doubleStyle}
\declaretheoremstyle[
	bodyfont=\normalfont,
	postfoothook={\postBoxSkipCmd},
	preheadhook={\preBoxSkipCmd},
	mdframed={
		backgroundcolor = black!4,
		linecolor = black!4,
		startcode={},
}]{boxStyle}
\declaretheoremstyle[
	headfont=\normalfont\itshape,
	notefont=\normalfont\itshape,
	notebraces={}{},
	bodyfont=\normalfont,
	qed=\qedsymbol,
	numbered=no,
	headindent=0pt,
	postheadspace=1ex,
	name={Proof},
	postheadhook={},
	mdframed={
		hidealllines = true,
		innerrightmargin = 0pt,
		innerleftmargin = 0pt,
		innertopmargin = 0pt,
		innerbottommargin = 0pt,
		leftmargin = 0pt,
		rightmargin = 0pt,
	}
]{proofStyle}
\declaretheoremstyle[
	bodyfont=\normalfont,
	postfoothook={\postBoxSkipCmd},
	preheadhook={\preBoxSkipCmd},
	mdframed={
		backgroundcolor = white,
		linecolor = black,
		startcode={},
		leftline = false,
		rightline = false,
}]{tobBottomStyle}
\declaretheoremstyle[
bodyfont=\normalfont,
]{standardStyle}
\declaretheorem[style=ruledBoxStyle,name=Definition, numberwithin=section]{definition}
\declaretheorem[style=ruledBoxStyle,name=Lemma,numberwithin=section,numberlike=definition]{lemma}
\declaretheorem[style=ruledBoxStyle,name=Proposition,numberwithin=section,numberlike=definition]{proposition}
\declaretheorem[style=ruledBoxStyle,name=Theorem,numberwithin=section,numberlike=definition]{theorem}
\declaretheorem[style=ruledBoxStyle,name=Corollary,numberwithin=section,numberlike=definition]{corollary}
\declaretheorem[style=boxStyle,name=Remark,numberwithin=section,numberlike=definition]{remark}
\declaretheorem[style=boxStyle,name=Notation,numberwithin=section,numberlike=definition]{notation}
\declaretheorem[style=dashedStyle,name=Example,numberwithin=section,numberlike=definition]{example}
\declaretheorem[style=dashedStyle,name=Assumption,numberwithin=section,numberlike=definition]{assumption}

%% file: abstract.tex
We noisily observe solutions of an ordinary differential equation $\dot u = f(u)$ at given times, where $u$ lives in a $d$-dimensional state space. The model function $f$ is unknown and belongs to a Hölder-type smoothness class with parameter $\beta$. For the nonparametric problem of estimating $f$, we provide lower bounds on the error in two complementary model specifications: the snake model with few, long observed solutions and the stubble model with many short ones. The lower bounds are minimax optimal in some settings. They depend on various parameters, which in the optimal asymptotic regime leads to the same rate for the squared error in both models: it is characterized by the exponent $-2\beta/(2(\beta+1)+d)$ for the total number of observations $n$. To derive these results, we establish a master theorem for lower bounds in general nonparametric regression problems, which makes the proofs more comparable and seems to be a useful tool for future use.

%% file: sec_intro.tex
\section{Introduction}\label{sec:intro}
\phantom{x}\\
\textbf{Motivation.}
In numerous scientific domains spanning from physics, engineering, and chemistry to biology, aspects of the real world are modeled as dynamical systems. See \cite{Ott1993, Strogatz2024} for a comprehensive overview. If the dynamics of such systems are not fully known, we may want to learn them from observational data and statistics comes into play. Reviews of current methods for statistical inference and data analysis for dynamical systems are given in \cite{McGoff15, Dattner2020}.

\textbf{Model.}
Here we consider $d$-dimensional, autonomous, first order ordinary differential equations (ODEs) as simple, yet general instances of a dynamical systems.  They have the form
\begin{equation}\label{eq:Intro1}
	\dot{u}(t) = f(u(t))
	\eqcm
\end{equation}
where $t\in\R$ is called \textit{time}, $u\colon\R\to\R^d$ describes the \textit{state} of the system at given time with its time-derivative denoted as $\dot u := \frac{\dl u}{\dl t}$, and $f\colon\R^d\to\R^d$ is the so-called \textit{model function} describing the dynamics of the system. The ODE is \textit{autonomous} as $f$ does not explicitly depend on time. It is of \textit{first order} as only the first derivative of $u$ is part of the equation. We denote by $t\mapsto U(f, x, t)$ the solution to the \textit{initial value problem} that is given by \eqref{eq:Intro1} and \textit{initial conditions} $x\in\R^d$, i.e., $\frac{\dl}{\dl t} U(f, x, t) = f(U(f, x, t))$ and $U(f, x, 0) = x$.
Classical textbook results provide existences and uniqueness of solutions to such initial value problems if $f$ is globally Lipschitz continuous, see for instance \cite{Magnus2023}. In our statistical model, we assume $f$ to be unknown, but to belong to some (infinite-dimensional) smoothness class, which is a subset of the Lipschitz continuous functions $\R^d \to \R^d$. For known initial conditions $x_1, \dots, x_m$, we assume to measure $t\mapsto U(f,x_j,t)$ at known time points $t_{j,1},\dots t_{j,n_j}$ to collect a total of $n = \sum_j n_j$ observations. The measurement error is modeled as centered and independent $d$-dimensional random variables $\noise_{j,i}$. Thus, our observations $Y_{j,i}$ follow the model equation
\begin{equation}\label{eq:RegODEIntro}
    Y_{j,i} = U(f,x_j,t_{j,i}) + \noise_{j,i}\eqfs
\end{equation}
Our goal is to investigate lower bounds on the rate of convergence for estimating the model function $f$. See \cref{fig:IntroDataExample} for a first visual illustration of the model.

\begin{figure}
    \begin{subfigure}{0.45\textwidth}
        \centering
        \includegraphics[width=\textwidth]{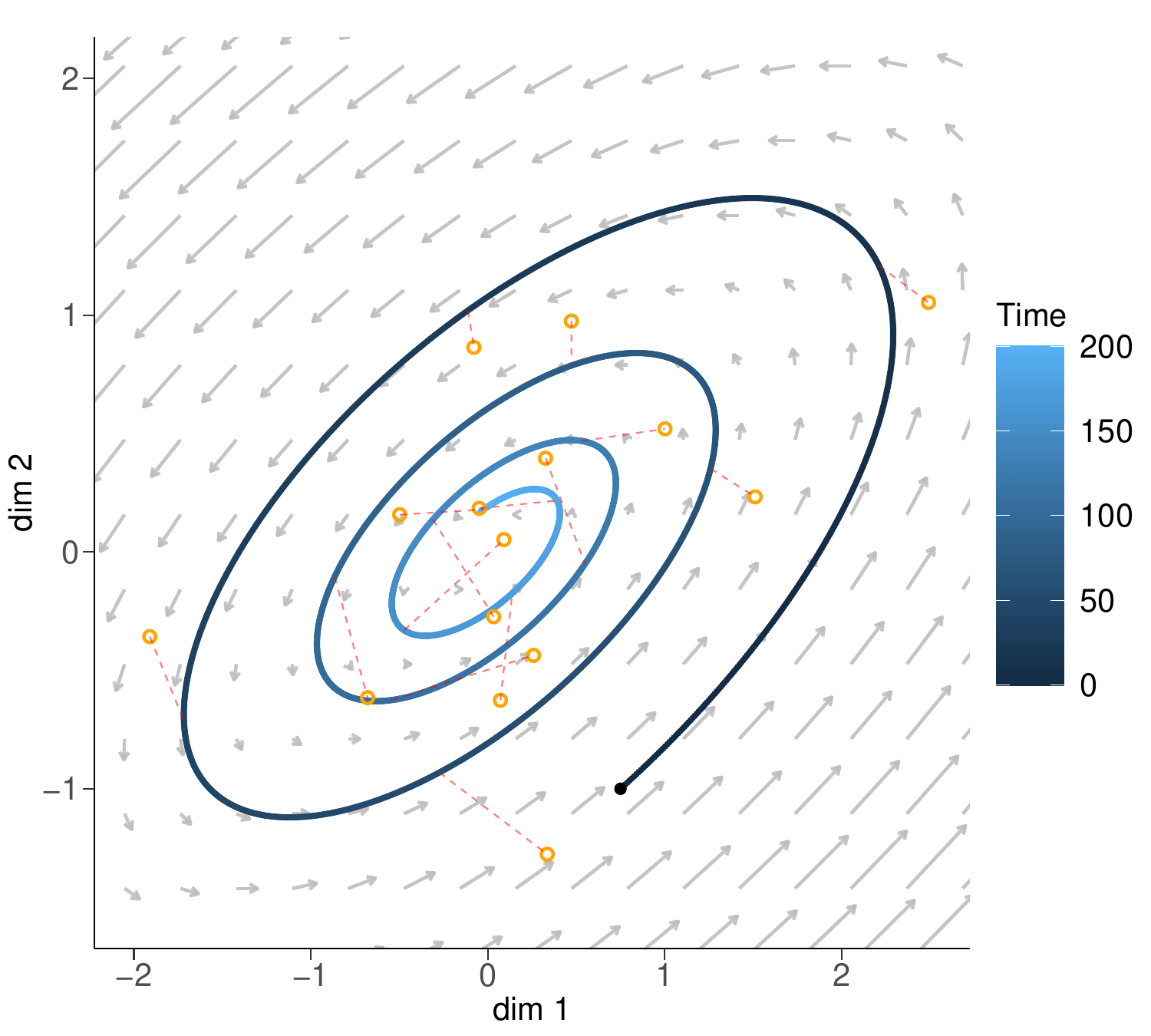}
        \caption{The state space view. The trajectory of $u$ is shown in blue colors. The gray arrows visualize the model function. The observations (depicted with orange circles) are connected with dashed lines to their noise-free state.}
        \label{fig:dataintroexample}
    \end{subfigure}
    \hfill
    \begin{subfigure}{.45\textwidth}
        \centering
        \includegraphics[width=\textwidth]{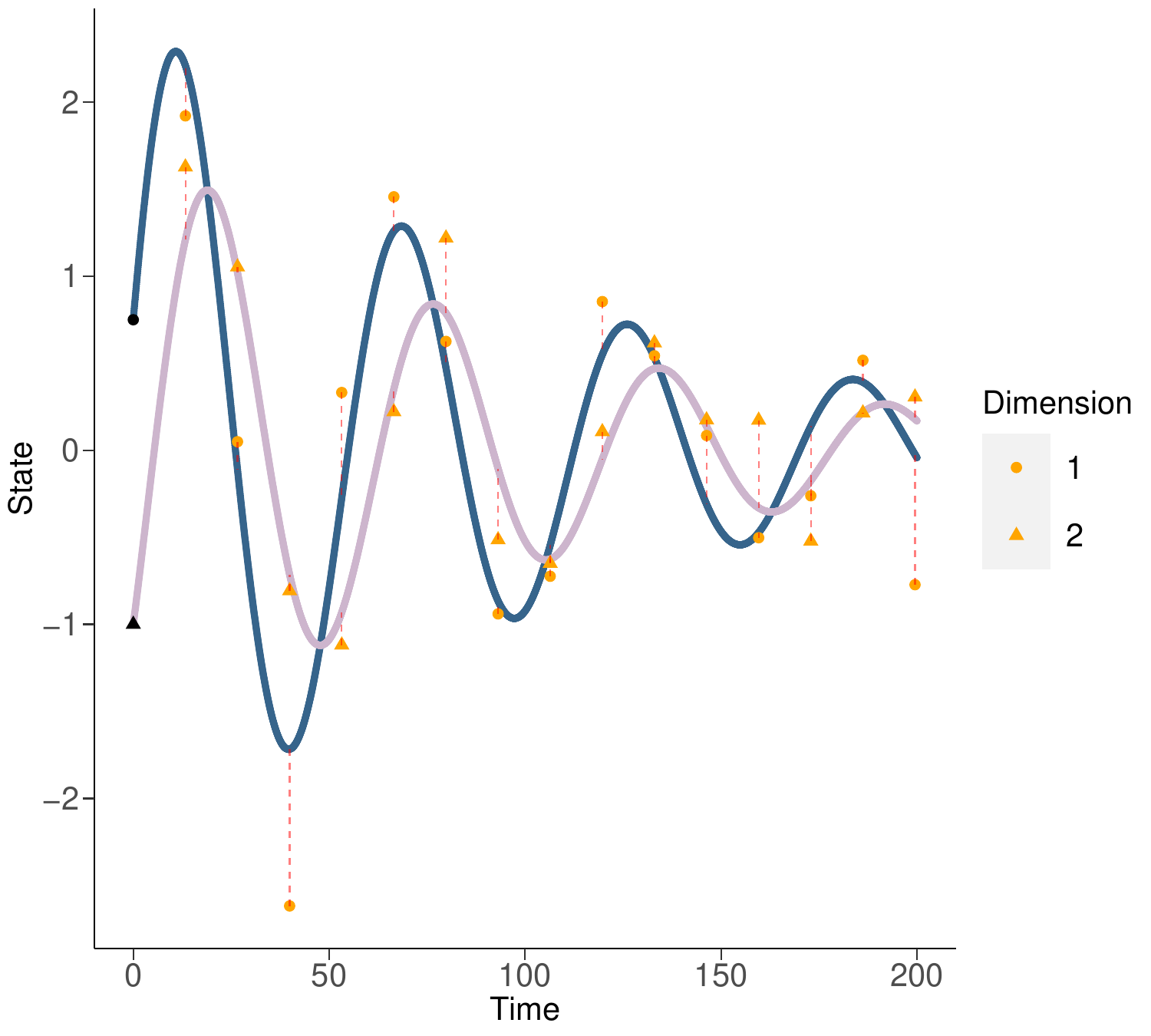}
        \caption{The evolution of the two state dimensions over time. The observations (depicted with orange circles and triangles, respectively) are connected with dashed lines to their noise-free counterpart.}
        \label{fig:data_dimplot_example}
    \end{subfigure}
    \caption{An example of the model with $m=1$ and $d=2$. A solution $u(t) = U(f, x_1, t)$ of an ODE is shown with observations $Y_i$.}
    \label{fig:IntroDataExample}
\end{figure}

\textbf{Inverse Problems.}
    Since the parameter of interest $f$ is observed only indirectly via the state variable $U(f,x,\cdot)$, the problem can be naturally interpreted as a statistical inverse problem. The non-linearity of the solution map (also called \textit{forward map}) $f\mapsto U(f,x,\cdot)$, renders classical results from linear inverse problem theory inapplicable. While (deterministic) non-linear inverse problems have a long mathematical tradition (see \cite{Engl} for instance), more recent work is concerned with nonparametric statistical theory for regression models driven by differential equations with non-linear forward maps. For instance, \cite{Monard2021} and \cite{NicklGeerWang2020} investigate PDE-constrained regression models from a Bayesian perspective and through M-estimation techniques, respectively. This emerging field has shown considerable promise for several PDE models (see \cite{Giordano2020,BohrNickl2023,Kekkonen2022}, and \cite{NicklTiti2024} for instance). These models are more complex than \eqref{eq:Intro1} in that they incorporate second order derivatives and derivatives in more than one variable. But they are also simpler in that they are linear differential equations, i.e., linear combinations of derivatives with coefficient functions to be estimated. In contrast, we allow arbitrary smooth model functions $f$ in \eqref{eq:Intro1}.

\textbf{Related Work.}
For the ODE-model \eqref{eq:RegODEIntro}, the parametric setting, in which $f = f_\theta$ is determined by a finite-dimensional parameter vector $\theta \in \R^p$, is well-studied. In \cite{Brunel_2008, qi10, gugushvili12, dattner15} different methods are explored that achieve the optimal $n^{-\frac12}$-error rate for estimating $\theta$. These estimators typically consist of two steps, in which initially the trajectory itself is estimated nonparametrically, and subsequently, $\theta$ is estimated using a least squares or maximum likelihood approach. See also \cite{Ramsay_2017} for a comprehensive overview.

In a parametric ODE model, each observation typically contains information on the whole, global parameter vector $\theta$. This is different in the nonparametric setting, where $f$ lives in an infinite dimensional space of functions. An observation at one location $z = U(f, x, t)$ is only locally informative, i.e., only contains information about $f(\tilde z)$ with $\tilde z$ close to $z$. As the location $z$ is determined by the unknown model function $f$, the nonparametric setting requires fundamentally different approaches than the parametric one.

Several algorithms have been proposed in the nonparametric setting of \eqref{eq:RegODEIntro} or similar ones, but almost no theoretical results on the estimation error have been proven:
The problem can be approached, e.g., by using Gaussian processes \cite{heinonen18}, neural networks \cite{chen18}, or random feature maps and ensemble Kalman filters \cite{gottwald21}. In \cite{Lahouel2023}, an algorithm for nonparametric estimation in the model \eqref{eq:RegODEIntro} based on reproducing kernel Hilbert spaces is introduced. The authors also prove a $n^{-\frac14}$-convergences rate upper bound for the root mean integrated squared error of estimating the observed trajectory $u$ within the observed time span.

Another novel statistical perspective has been developed in \cite{marzouk2023}, where distribution learning via neural differential equations is studied. Utilizing transport of measure theory, the authors establish nonparametric statistical convergence analysis for distribution learning via ODE models trained through likelihood maximization. 

\textbf{Upper Bounds, Stubble and Snake.} Upper bounds on the error rate for nonparametric estimation of the model function $f$ have recently been studied in \cite{upperbounds}. To make consistent estimation possible, one of two different sets of conditions are assumed, leading to the so-called \textit{snake} and \textit{stubble} models. The observations need to contain sufficient information about the model function $f$ in the subset of its domain where we want to estimate it (the \textit{domain of interest}, e.g., $[0,1]^d$). But the location of observation itself is determined by the dynamics of the system, i.e., by $f$. In the observation scheme of the stubble model, we assume to observe many short trajectories $t\mapsto U(f, x_j, t)$ that start from initial conditions which uniformly cover the domain of interest. In contrast to this, in the snake model, we measure only few, but long trajectories that are required to cover the domain of interest in a suitable way. Both observation schemes are illustrated in \cref{fig:test}. For both settings \cite{upperbounds} provides estimation algorithms with upper bounds on the rate of convergence.

\textbf{Contributions.} We provide lower bounds for the estimation risk in terms of pointwise, $L^p$ and uniform loss for the snake and the stubble model, where the relevant parameter dependencies are tracked rigorously.

The proofs are based on theorems for deriving lower bounds in a very generic regression model framework (\cref{thm:lower:master}, \cref{thm:lower:deterministic:master}). We belief these to be useful tools not restricted to their specific application in this work. We illustrate the usefulness of these general theorems by applying them to the classical nonparametric regression framework (\cref{thm:lower:regression}). Furthermore, using the general results for the proofs in the snake and stubble ODE models as well as in the classical regression model, allows for a direct comparison (\cref{tbl:symbolsMaster}) and potentially deeper understanding of the proofs.

We prove lower bounds for the estimation of the ground truth model function $f = \ftrue\colon \R^d\to\R^d$ in \eqref{eq:RegODEIntro} by an estimator $\festi$ based on $Y_{j,i}$, assuming that $\ftrue$ belongs to a Hölder-type class of $\beta$-smooth functions $\FSmooth$. The results are provided for point-wise, $L^p$, and uniform loss. They are presented in \cref{thm:stubble:probabilistic}, \cref{thm:stubble:deterministic}, and \cref{cor:stubble:nice} for the stubble model and in \cref{thm:snake:probabilistic}, \cref{thm:snake:deterministic}, and \cref{cor:snake:combined} for the snake model. Here, we paraphrase the results for point-wise error at $x_0\in\R^d$ and assume equidistant observation times, $t_{j,i+1}-t_{j,i} = \stepsize$.

\textbf{In the Stubble model:}
    If the initial conditions form a regular grid in the domain of interest and only a bounded amount of observations is made per trajectory, we have
	\begin{equation}\label{eq:intro:stubble:result}
		C \inf_{\festi} \sup_{\ftrue\in\FSmooth} \EOff{\ftrue}{\euclOf{\festi(x_0)-\ftrue(x_0)}^2}
		\geq
		\stepsize^{2\beta} + \br{n \stepsize^2}^{-\frac{2\beta}{2\beta+d}}
	\end{equation}
	for some constant $C>0$.

\textbf{In the Snake model:}
	If the observed trajectories cover the domain of interest when inflated to a radius $\delta$, then
	\begin{equation}\label{eq:intro:snake:result}
	  	C \inf_{\festi} \sup_{\ftrue\in\FSmooth}  \EOff{\ftrue}{\euclOf{\festi(x_0)-\ftrue(x_0)}^2}
	  	\geq
	  	\delta^{2\beta}+ \br{\delta^{d-1} \stepsize}^{\frac{2\beta}{2(\beta+1)+d}}
	\end{equation}
	for some constant $C>0$.

\textbf{In both models:}
	Even though the two models are complementary, we show that the same minimal lower bound holds independently of the asymptotics of $\stepsize$ and $\delta$ (\cref{cor:stubble:nice:onlyn}, \cref{cor:snake:combined:nice}): When balancing the two terms in \eqref{eq:intro:stubble:result} as well as the two terms in \eqref{eq:intro:snake:result}, we obtain
	\begin{equation}\label{eq:intro:nice}
		C \inf_{\festi} \sup_{\ftrue\in\FSmooth} \EOff{\ftrue}{\euclOf{\festi(x_0)-\ftrue(x_0)}^2}
		\geq
		n^{-\frac{2\beta}{2(\beta +1) + d}}
		\eqfs
	\end{equation}
	This rate makes sense intuitively when comparing it to the classical minimax optimal nonparametric rate of convergence for the squared error in regression,
	\begin{equation}\label{eq:intro:npreg}
		n^{-\frac{2(\tilde\beta-s)}{2\tilde\beta + d}}
		\eqcm
	\end{equation}
	see \cref{thm:lower:regression} in \cref{sec:app:lower}, where the $s$-th derivative of a $\tilde\beta$-smooth function $\R^d \to \R$ is estimated. In our ODE-regression model, the observed solutions have smoothness $\tilde\beta = \beta+1$, and we want to estimate their first derivative ($s=1$) in the form  of $f$, which has a $d$-dimensional domain. Thus, the rate in \eqref{eq:intro:nice} fits precisely the classical nonparametric regression rate \eqref{eq:intro:npreg}.

These rates are shown to be minimax optimal in certain settings by comparing them to \cite{upperbounds}.

\textbf{Overview.} In \cref{sec:model}, the general ODE model is discussed. In \cref{sec:stubble} and \cref{sec:snake}, we focus on the stubble model and the snake model, respectively, and present the main results, i.e., the lower bounds and discuss the ideas of their proofs. We show some auxiliary results and review some basics and notation on multivariate derivatives in \cref{sec:analystical}. In \cref{sec:app:lower}, the general lower bounds theorem for regression models is stated, proven and applied in an example. The main proofs are given in \cref{sec:app:stubble} and \cref{sec:app:snake}. Finally, \cref{sec:app:run} comments on the number of trajectories $m$ required to obtain the error lower bound in the snake model.

%% file: sec_model.tex
\section{General Model}\label{sec:model}
\subsection{Model Description}
We consider an autonomous, first order, ordinary differential equation (ODE) of the form $\dot u = f(u)$, where $f$ is an unknown smooth function, which we call the model function. One or several solutions to this ODE are observed with noise at given times. The noise is assumed to be independent.
The task we consider is to estimate the model function $f$.

As we will show in the following, care has to be taken when defining the precise model and estimation task. In a straight-forward setting, no consistent estimation is possible at all.
\subsection{Model Definition}\label{sec:modelDef}
\begin{notation}\mbox{}
	\begin{enumerate}[label = (\roman*)]
		\item
		Let $\mb K \in \cb{\N, \Z, \R}$.
		For $a\in\mb K$, define $\mb K_{\geq a} := \mb K \cap [a,\infty)$ and $\mb K_{>a} := \mb K \cap (a,\infty)$.
		For $n,m\in\Z$, define $\nset{n}{m} := \Z\cap[n,m]$ and $\nnset{n}:=\nset{1}{n}$.
		\item
		For $x\in\R$, define the largest integer strictly smaller than $x$ as $\llfloor x\rrfloor := \max\setByEleInText{m\in\Z}{m < x}$, and the smallest integer strictly larger than $x$ as $\llceil x\rrceil :=\min\setByEleInText{m\in\Z}{m > x}$. The non-strict versions are denote as $\lfloor x\rfloor := \max\setByEleInText{m\in\Z}{m \leq x}$ and $\lceil x\rceil :=\min\setByEleInText{m\in\Z}{m\geq x}$.
		\item
		For $\mathcal{I}\subseteq\Z$ define the shorthand notation $O_\mathcal{I}:=(O_i)_{i\in\mathcal{I}}$ for any family of objects $O_i$, $i\in\mc I$.
		In particular, we will write $\indset{O}{n}{m}$ for $(O_n, \dots, O_m)$, where $n,m\in\N$, $n\leq m$.
		\item
		Let $(\Omega, \mathcal{A}_\Omega, \Pr)$ be a probability space and denote by $\Eof{\cdot}$ the corresponding expectation. We assume that every random variable mentioned in the following is defined on this space without further mentioning it.
		\item
		Let $d\in\N$. For a probability distribution $P$ on $\R^d$ and a vector $v\in\R^d$, denote by $v + P$ the distribution $P$ shifted by $v$, i.e., $(v+P)(A) = P(\setByEleInText{x - v}{x\in A})$ for all measurable sets $A\subset \R^d$.
		\item
		Let $d\in\N$. Denote the Euclidean norm of $v\in\R^d$ by $\euclof{v}$.
	\end{enumerate}
\end{notation}
Let $d\in\N$ be the dimension of the \textit{state space} $\R^d$.
For a (globally) Lipschitz-continuous function $f\colon\R^d\mapsto\R^d$ (the \textit{model function}) and \textit{initial conditions} $x\in\R^d$, let $U(f, x, \cdot)\colon\R\to\R^d,\,t\mapsto U(f, x, t)$ be the solution to the \textit{initial value problem} $\dot u(t) = f(u(t))$ with $u(0) = x$. The global Lipschitz assumption implies global existence and uniqueness of the solutions. We refer to $t$ and the domain of $u$ as \textit{time} and to $U(f,x,t)$ as \textit{state}. The image of $U(f,x,\cdot)$ is a \textit{trajectory}. The map $U(f, \cdot, \cdot)\colon\R^d\times\R\to\R^d$ is called \textit{flow}. It has the \textit{semigroup property}
\begin{equation}
	U(f, U(f, x, s),t) = U(f, x, s + t)  
\end{equation}
for all $x\in\R^d, s, t, \in \R$.

Let $\beta\in\Rpo$ be the \textit{smoothness parameter}. Set $\ell := \llfloor\beta\rrfloor$. Let $L_0, \dots, L_\ell, \Lbeta\in\Rpp$.
Let $\FSmooth := \Sigma^{d\to d}(\beta, \indset{L}{0}{\ell}, \Lbeta)$ denote our \textit{smoothness class} of model functions, which contains exactly those functions $f\colon\R^d\to\R^d$ with the following properties: For $j\in\nnset d$, let $f_j\colon\R^d \to\R$ be the $j$-th component functions of $f$, i.e., $f(x) = (f_1(x), \dots, f_d(x))$ for all $x\in\R^d$. For $k\in\nnzset\ell$, the $k$-th derivative $D^k f_j(x)$ of $f_j$ at any location $x\in\R^d$ is bounded in operator norm by $L_k$, i.e.,
\begin{equation}
	\forall k\in\nnset\ell\colon \forall x\in\R^d\colon \ \opNormof{D^k f_j(x)}\leq L_k
	\eqfs
\end{equation}
Moreover, the $\ell$-th derivative is $(\beta-\ell)$-Hölder continuous with constant $\Lbeta$, i.e,
\begin{equation}
	\forall x,\tilde x\in\R^d\colon\ \opNormOf{D^\ell f_j(x) - D^\ell f_j(\tilde x)} \leq \Lbeta \euclOf{x - \tilde x}^{\beta-\ell}
	\eqfs
\end{equation}
We refer to the appendix section \ref{def:Hoelder} for further details on derivatives and Hölder-smoothness classes.

Let $\ftrue\in\FSmooth$ be the \textit{true model function} or \textit{ground truth}. Let $m \in \N$ be the \textit{number of observed trajectories}. Let $n_1, \dots, n_m \in \N$ be the \textit{number of observations per trajectory}. Denote $\nmax :=\max_{j\in\nnset m}n_j$. Let $\setIdx := \bigcup_{j=1}^m (\{j\}\times\nnset{n_j})$ be the \textit{set of indices} for our observations and $n := \cardof{\setIdx} = \sum_{j=1}^m n_j$ the \textit{total number of observations}.
Let $T_j\in\Rpp$ be the \textit{maximum observation time of trajectory $j\in\nnset m$}. Let the \textit{observation time} of the $i$-th observation on trajectory $j$ be $t_{j,i} \in [0, T_j]$ for $(j,i) \in \setIdx$ with $t_{j,0} = 0$, $t_{j,n_j} = T_j$, and $t_{j,i-1} \leq t_{j,i}$. Denote the (overall) \textit{maximum observation time} as $T_{\ms{max}} := \max_{j\in\nnset{m}} T_j$ and the \textit{total time} as $\Tsum := \sum_{j\in\nnset{m}} T_j$.
Let $\xany1, \dots, \xany m \in \R^d$ be the \textit{initial conditions} for the $m$ trajectories.
Let the \textit{noise} $\noise_{j,i}$, $(j,i) \in \setIdx$ be independent $\R^d$-valued random variables with expectation $\Eof{\noise_{j,i}} = 0$.
Let
\begin{equation}\label{sec2:data}
	Y_{j,i} = U(\ftrue, \xany j, t_{j,i}) + \noise_{j,i}\,,\qquad (j,i)\in\setIdx
\end{equation}
be our observations. In this model, we observe $Y_{j,i}$ and know $t_{j,i}$ as well as $\xany j$, but $\ftrue$ is unknown.

To obtain lower bounds on the error of estimating $\ftrue$, we will require the noise distribution to fulfill the following condition.
\begin{assumption}\mbox{ }
	\begin{itemize}
		\item \newAssuRef{Noise}:
			There is $\const{noise}\in\Rpp$ such that, for all $v_1,v_2\in\R^{d}$, we have
			\begin{equation}
				\kullback(v_1 + P^\noise, v_2 + P^\noise) \leq \const{noise}\euclof{v_1 - v_2}^2,
			\end{equation}
            where $\kullback(\cdot,\cdot)$ denotes the Kullback-Leibler distance for two probability measures.
	\end{itemize}
\end{assumption}
\begin{remark}\mbox{}\label{rem:noise}
    \begin{enumerate}[label=(\roman*)]
        \item
            \textbf{Example.}
            The normal distribution $P^\noise = \mc N(0, A)$, where $A\in\R^{d\times d}$ is symmetric and positive definite, fulfills \assuRef{Noise} with
            \begin{equation}
                \const{noise} = \frac{1}{2\lambda_{\ms{min}}}
                \eqcm
            \end{equation}
            where $\lambda_{\ms{min}}\in\Rpp$ is the smallest eigenvalue of $A$.
        \item
            \textbf{Interpretation.}
            \assuRef{Noise} ensures that the noise is not restricted to too simple distributions. For example, if we would know that the errors have a uniform distribution, a faster rate of estimation may be possible. Intuitively, the mean of a uniform distribution can be estimated with a $1/n$ error rate, whereas the error for estimating the mean of a Gaussian is $1/\sqrt{n}$. This leads to different rates of convergence in regression problems. See \cite[Exercise 2.7]{Tsybakov09Introduction}.
    \end{enumerate}
\end{remark}
\subsection{Target: Estimation of the Model Function}
\begin{notation}\mbox{}
    \begin{enumerate}[label = (\roman*)]
        \item
        For $p\in [1,\infty]$, $d\in\N$, a measurable set $A\subset\R^d$, and a measurable function $g\colon A \to \R$, denote the $L^p$-norm of $g$ on $A$ as
        \begin{equation}
            \LpNormOf{p}{A}{g} =
            \begin{cases}
                \br{\int_A \abs{g}^p \dl x}^\frac1p & \text{ for } p\in[1,\infty)\eqcm \\
                \esssup_{x\in A} \abs{g} & \text{ for } p = \infty\eqfs
            \end{cases}
        \end{equation}
    \end{enumerate}
\end{notation}
Let us consider the estimation of $\ftrue$ via an estimator denoted $\festi$.
A first ingredient for a typical loss function for such an objective is a way to compare $\ftrue(x), \festi(x)\in\R^d$. We will always use the Euclidean norm $\euclof\cdot$ in $\R^d$ for this purpose and base loss functions on
\begin{equation}
	\euclof{\festi-\ftrue} \colon \R^d\to\Rp\eqcm x\mapsto  \euclof{\festi-\ftrue}(x) := \euclof{\festi(x)-\ftrue(x)}
	\eqfs
\end{equation}
For a given state $x_0\in\R^d$, we may want to consider the pointwise loss
\begin{equation}
	\euclof{\festi-\ftrue}(x_0) = \euclOf{\festi(x_0) - \ftrue(x_0)}
	\eqfs
\end{equation}
To obtain a global evaluation of an estimate, we can instead integrate $\euclof{\festi - \ftrue}$ over some measurable subset $A \subset \R^d$. A general form of this loss is
\begin{equation}
	\LpNormOf{p}{A}{\euclOf{\festi - \ftrue}}
	\eqcm
\end{equation}
where $p\in [1,\infty]$.
The set $A$ may be chosen beforehand as a fixed domain of interest such as $A := [0,1]^d$. However, it is not possible to obtain a consistent estimate without making further assumptions. Consider following example. We observe $m=1$ trajectories in $d=1$ dimensions and are interested in $\ftrue$ on $A=[0, 1]$. But the initial conditions $x_1 = 0$ and the true model function $\ftrue$ are so that $\ftrue(x_1) = 0$. Then we will only make observations at $U(\ftrue, x_1, t) = x_1$. Therefore, we will not be able to create a meaningful estimate of, say, $\ftrue(1)$. This example can be easily extended to higher dimensions.

To obtain non-trivial results, we make restrictions on the general model so that consistent estimation is possible. The restrictions ensure a suitable distribution of observations across the domain of interest (see \cref{fig:sub1} for an example where this as condition is not fulfilled). Two kinds of restrictions are discussed: In the stubble model, we assume to cover the domain of interest by a suitable number of initial conditions, each of which starts a new trajectory that is observed (e.g., \cref{fig:sub3}). In the snake model, the number of initial conditions and their location is not directly restricted, but the observed trajectory or trajectories must have forms to suitably cover of the domain of interest (e.g., \cref{fig:sub2}). The precise meaning of \textit{suitably} is discussed in \cref{sec:snake}.

\begin{figure}
	\centering
    \begin{subfigure}{.49\textwidth}
        \centering
         \includegraphics[width=\textwidth]{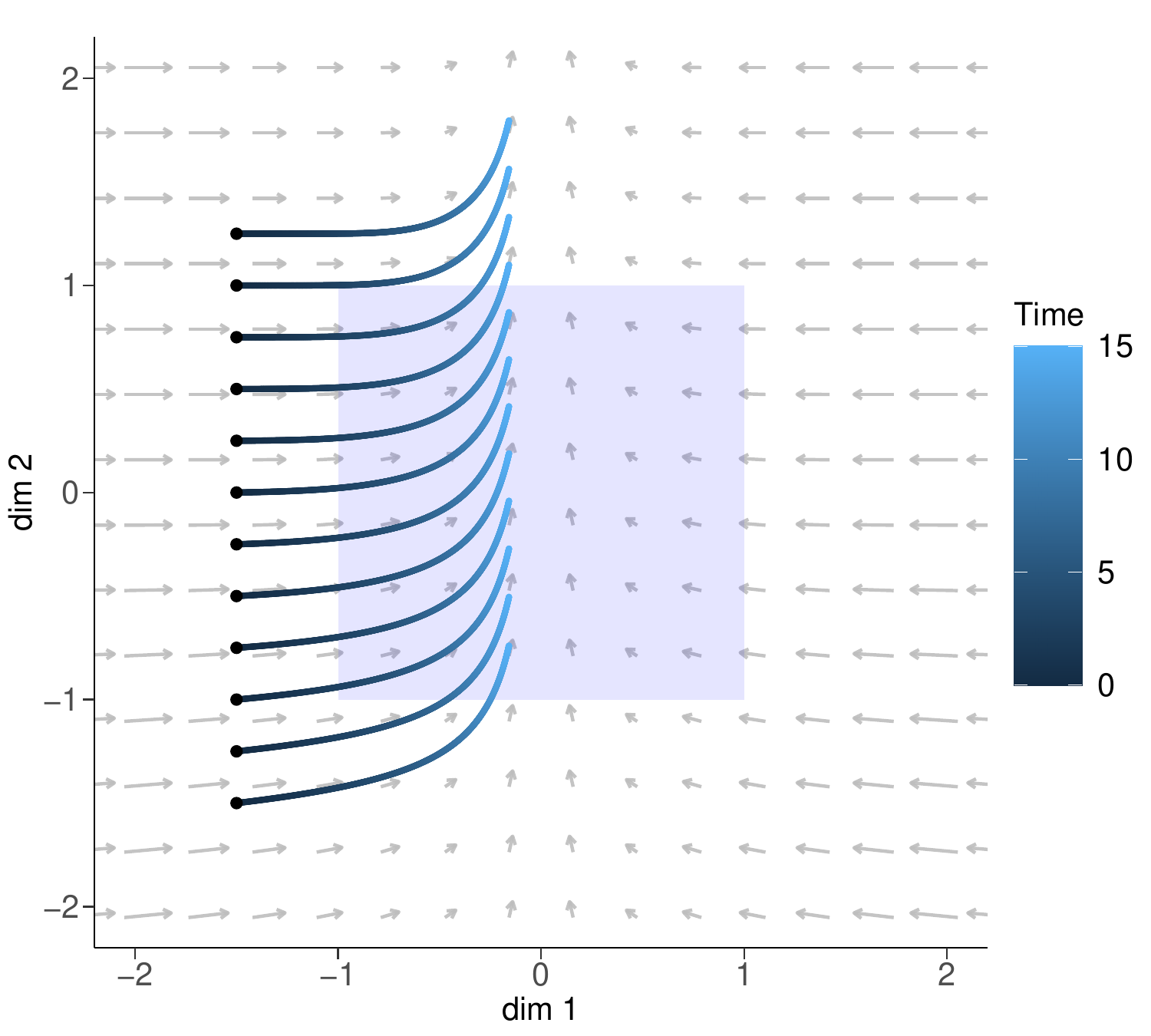}
            \caption{Observations of the given trajectories are insufficient for the estimation of $f$ within the domain of interest.}
        \label{fig:sub1}
    \end{subfigure}\\[1ex]
    \begin{subfigure}{0.49\textwidth}
        \centering
         \includegraphics[width=\textwidth]{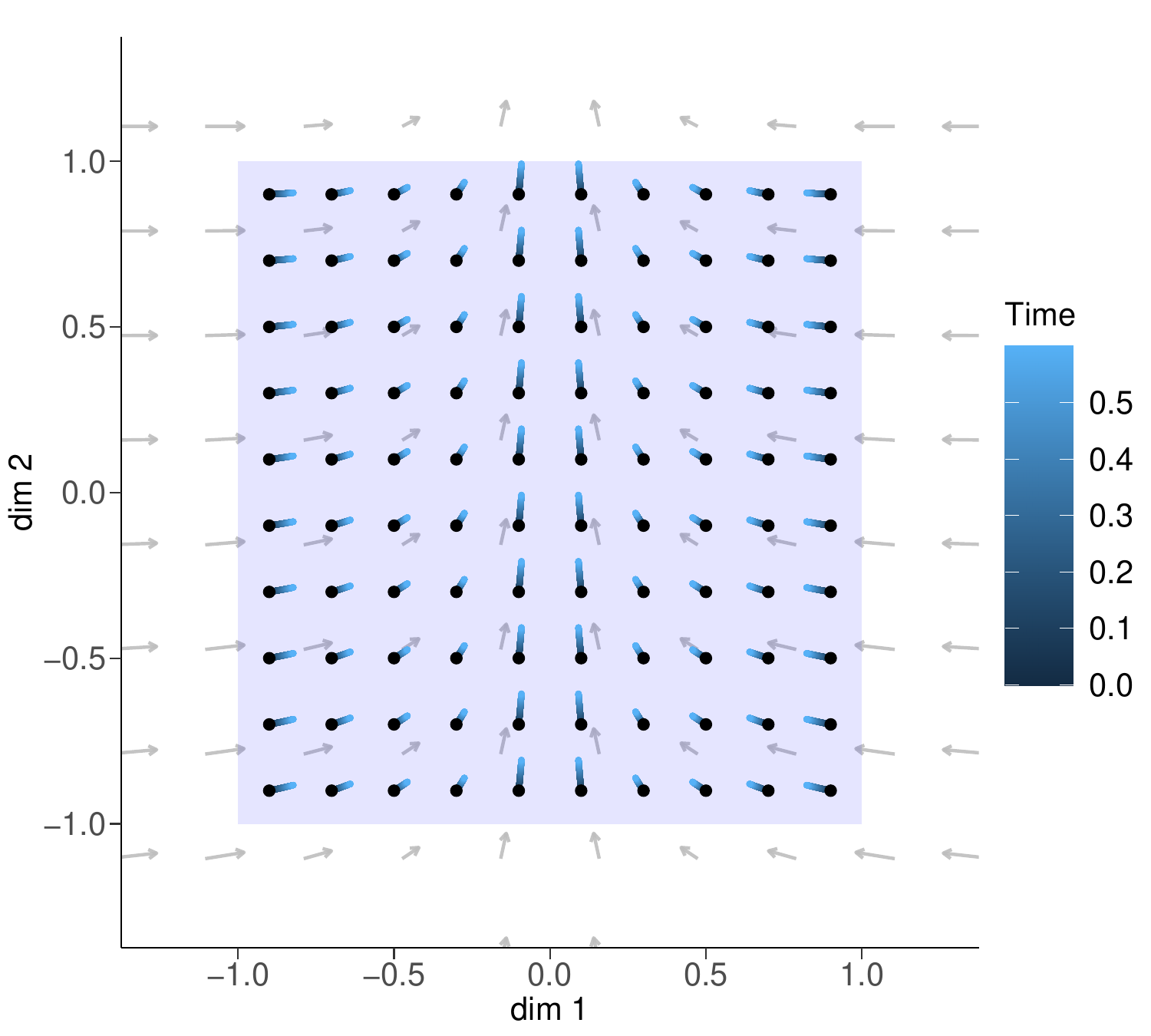}
            \caption{Stubble Model with a uniform grid of initial conditions in the domain of interest.}
        \label{fig:sub3}
    \end{subfigure}
    \hfill
    \begin{subfigure}{.49\textwidth}
        \centering
         \includegraphics[width=\textwidth]{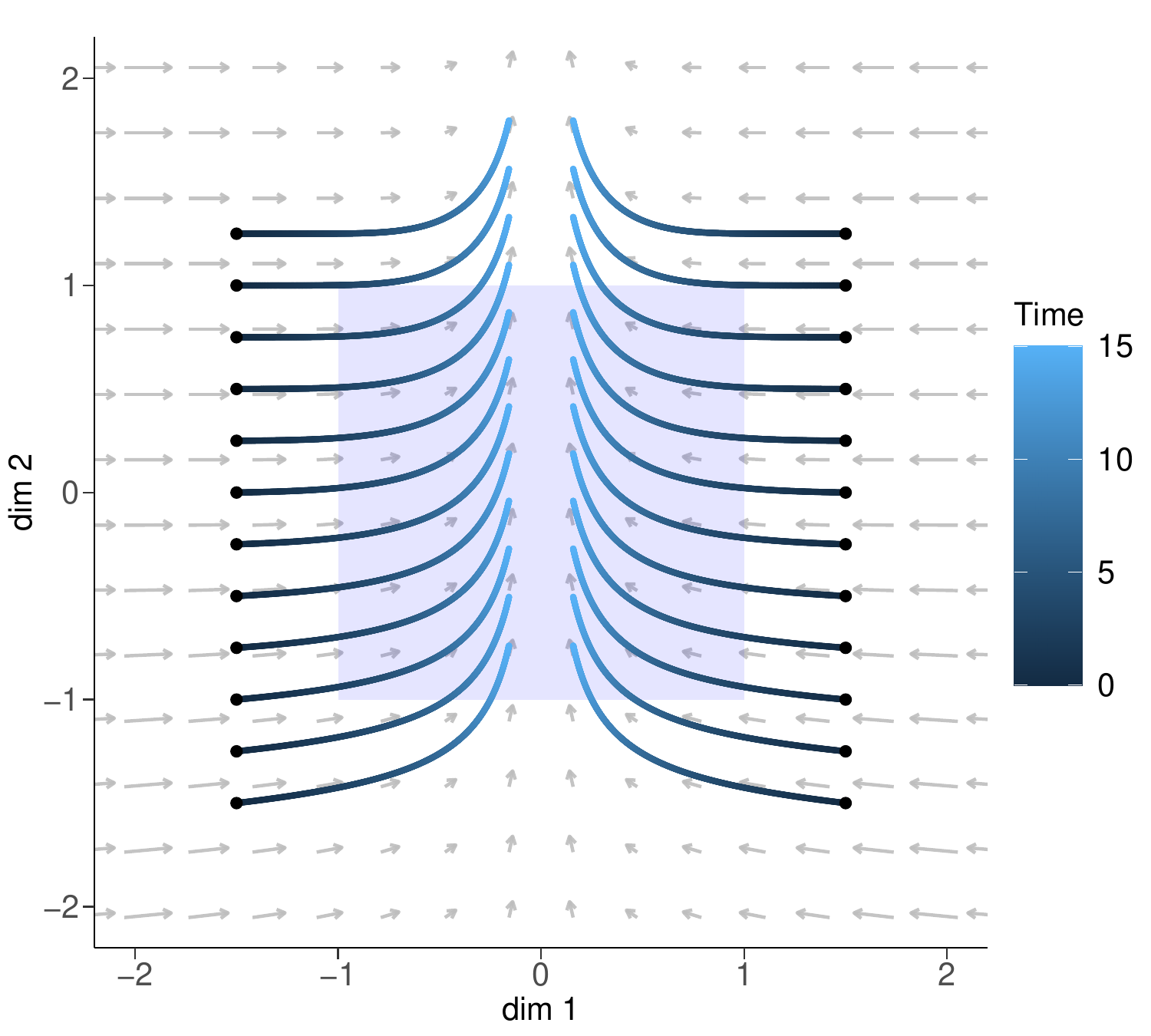}
            \caption{Snake Model with trajectories that cover the domain of interest in a suitable sense.}
        \label{fig:sub2}
    \end{subfigure}\\[1ex]
    \caption{Illustrations of the problem of consistent estimation of $f$ and solution strategies. The underlying model function $f$ has a vanishing value in the first dimension when the first dimension becomes 0. The plots show the state space. The model function is depicted with gray arrows. Trajectories are shown in black to blue color with black points indicating their initial conditions. The color transition visualizes time. The domain of interest is $[-1,1]^2$.}
    \label{fig:test}
\end{figure}

%% file: sec_stubble.tex
\section{Stubble Model}\label{sec:stubble}
In this section, we consider the model described in \cref{sec:modelDef} with the target of estimating $\ftrue$ on the domain of interest $\hypercube$ and denote by $\Pr_{\ftrue}$ and $\EOff{\ftrue}{\cdot}$ the corresponding probability measure and expectation, respectively. We make a restriction on the initial conditions to avoid trivial settings in which consistent estimation is impossible. The resulting model is called the \textit{stubble model}. In this model, we derive two different types of lower bounds -- one \textit{deterministic}, which is independent of the noise and the sample size $n$ and another \textit{probabilistic} depending on these objects. Then the rates are combined for a full picture. The obtained lower bounds on the error rates are shown to be minimax optimal in at least some settings.
\subsection{Model}
\begin{notation}\mbox{}
	\begin{enumerate}[label = (\roman*)]
		\item
		Let $d\in\N$, $z\in\R^d$, $r\in\Rpp$. Denote the open ball in $\R^d$ with radius $r$ and center $z$ as
		\begin{equation}
			\ball^d(z,r) := \setByEle{x\in\R^d}{\euclOf{x-z} < r}
			\eqfs
		\end{equation}
		\item
		Let $A$ be a set. Denote the indicator function of $A$ as $\indOfOf{A}{\cdot}$, i.e.,
		\begin{equation}
			\indOfOf{A}{x} = \begin{cases}
				1, & \text{ if } x\in A\eqcm\\
				0, & \text{ if } x\not\in A\eqfs
				\end{cases}
		\end{equation}
	\end{enumerate}
\end{notation}
In the stubble model, we need following assumption on the initial conditions.
\begin{assumption}\mbox{ }
	\begin{itemize}
		\item \newAssuRef{Cover}:
		There is $\const{cvr}\in\Rpp$ such that, for all $z\in[0,1]^d$, $r\in\Rpp$, we have
		\begin{equation}
			\frac1m\sum_{j=1}^m \indOfOf{\ball^d(z, r)}{x_j} \leq \max\brOf{\frac1m,\, \const{cvr}r^d}
			\eqfs
		\end{equation}
	\end{itemize}
\end{assumption}
\begin{remark}\mbox{}
    \begin{enumerate}[label=(\roman*)]
        \item
            \textbf{Example.}
            The regular grid in $[0, 1]^d$,
            \begin{equation}
                \cb{x_1, \dots, x_m} = \setByEle{\begin{pmatrix}\frac{k_{1}}K & \dots & \frac{k_{d}}{K}\end{pmatrix}\tr}{k_1, \dots, k_d\in\nset0K}
            \end{equation}
            with $K\in\N$ and $m = (K+1)^d$, fulfills \assuRef{Cover} with
            \begin{equation}
                \const{cvr} = 4^d
                \eqcm
            \end{equation}
            which can be seen by bounding the hypersphere $\ball^d(z,r)$ by a hypercube of side length $2r$.
        \item
            \textbf{Interpretation.}
            To make a non-trivial upper bound on the estimation error possible, one could require each part of the domain of interest to contain a sufficient amount of initial conditions and, therefore, of observations. This would ensure that there are no large holes in the state space without observations, where the estimation error does not vanish. As we are interested in a lower bound, we make a complementary assumption: We require that no ball in the domain of interest contains too many initial conditions, which could yield over-proportionally good estimates locally.
    \end{enumerate}
\end{remark}
\subsection{Lower Bound -- Probabilistic}
\begin{theorem}\label{thm:stubble:probabilistic}
	Use the model of \cref{sec:modelDef}.
	Assume \assuRef{Noise} and \assuRef{Cover}.
	Let $p\in\Rpo$.
	Then there is $C\in\Rpp$ large enough, depending only on $\beta, d, L_0, \dots, L_\ell, \Lbeta, \const{noise}, \const{cvr}, p$ with the following property:
	Assume
	\begin{equation}
		m^{\frac{2\beta}d} \geq C \nmax \Tmax^2
		\qquad\text{and}\qquad
		m \nmax \Tmax^2 \geq C
		\eqfs
	\end{equation}
	Then
	\begin{align}
		\forall x_0\in\hypercube\colon\ \inf_{\festi} \sup_{\ftrue\in\FSmooth}
		&\Pr_{\ftrue}\brOf{C \euclOf{\festi - \ftrue}(x_0) \geq \br{m \nmax \Tmax^2}^{-\frac{\beta}{2\beta+d}}}
		\geq
		\frac14
		\eqcm\\
		\inf_{\festi} \sup_{\ftrue\in\FSmooth}
		&\Pr_{\ftrue}\brOf{C \sup_{x\in\hypercube}\euclOf{\festi - \ftrue}(x) \geq \br{\frac{m \nmax \Tmax^2}{\log\brOf{m \nmax \Tmax^2}}}^{-\frac\beta{2\beta+d}}}
		\geq
		\frac14
		\eqcm\\
		\inf_{\festi} \sup_{\ftrue\in\FSmooth}
		&\Pr_{\ftrue}\brOf{C \LpNormOf{p}{\hypercube}{\euclOf{\festi - \ftrue}} \geq \br{m \nmax \Tmax^2}^{-\frac{\beta}{2\beta+d}}}
		\geq
		\frac14
		\eqcm
	\end{align}
    where the infima range over all estimators $\festi$ of $\ftrue$ based on the observations $Y_{\setIdx}$.
\end{theorem}
\begin{remark}\mbox{ }\label{rem:stubble:probabilistic}
	\begin{enumerate}[label=(\roman*)]
		\item
        \textbf{Maximum Values.}
        The bounds are presented in terms of the maximum number of observations for a trajectory $\nmax = \max_j n_j$ and the maximum observation time for a trajectory $\Tmax = \max_j T_j$. We do not expect the bounds to be optimal if the variation in $n_j$ or $T_j$ is extreme. Thus, a primary way of understanding the theorem is for the case of equal numbers of observation $n_j$ and equal observation times $T_j$. In this case, the total number of observations is $n=m\nmax$.
        \item
        \textbf{Relation Assumptions.}
		We are mainly interested in asymptotics where $\nmax$ is constant, $m \xrightarrow{n\to\infty} \infty$, and $\Tmax  \xrightarrow{n\to\infty} 0$ or $\Tmax \leq c$. For this case, the condition $m^{\frac{2\beta}d} \geq C \nmax \Tmax^2$ is not a restriction. The requirement $m \nmax \Tmax^2 \geq C$ requires $\Tmax$ to not decline too quickly. As an example, constant $\nmax$, $m = c n$, $\Tmax = c n^{-\frac{1}{2(\beta +1) + d}}$ fulfill the conditions. These settings are used in \cref{cor:stubble:nice:onlyn} to show a minimal lower bound only depending on $n$.
        \item
        \textbf{Error Tendencies.}
		Generally, we obtain a lower bound that is large if $n \approx m\nmax$ is small and $\Tmax$ is small. The former is typical, the latter can be explained as follows: The larger time gets, the more distance can be attained between trajectories that start from the same initial conditions, but are driven by different model functions. Thus, the signal to noise ratio gets larger. In other words, the longer two stubbles are, the more apparent the differences in their dynamics will be.
	\end{enumerate}
\end{remark}
\cref{thm:stubble:probabilistic} is a direct consequence of \cref{thm:stubble:probabilistic:details}, see appendix \ref{sec:app:stubble:probabilistic}, where the proof is given in detail. Here, we present a sketch of the proof.
\begin{proof}[Proof sketch of \cref{thm:stubble:probabilistic}]
    The proof is based on an application of the master theorem, \cref{thm:lower:master}, which presents a tool for deriving lower bounds for general regression models and is itself based on \cite{Tsybakov09Introduction}.
    Essentially, we need to find two (or more) hypotheses $f\in\FSmooth$ that have a large distance in the error metric of interest while still being hard to distinguish from observations.

    We use the null hypothesis
    \begin{equation}
        f_0(x) = \br{0,\, \dots,\, 0}\tr
        \eqcm
    \end{equation}
    which results in no movement of the states $U(f_0,x,t)$ over time. For the alternatives, we add one or several \textit{bumps}\footnote{Here a \textit{bump} refers to a smooth function $\R\to\R$ with compact support that first increases into positive values, than decreases until it is 0.} to the first dimension of $f_0$. The bumps are scaled and shifted versions of prototype $\hbump{d}{\beta}\colon \R^d \to \R$ that is based on a symmetric kernel with compact support. See \cref{lmm:bumpandpulse} for a definition.
    For a location $z\in\R^d$ and a scale $r\in\Rpp$, we define the alternative model function as
    \begin{equation}
        f_{z,r}(x) := \br{\Lbeta r^\beta \hbump{d}{\beta}\brOf{\frac{x-z}{r}},\, 0,\, \dots,\, 0}\tr
        \eqfs
    \end{equation}
    The solutions of the initial value problem corresponding to such an alternative are constant in all directions, except when they start in the support of a bump. In that case there is some small movement in the first dimension. See \cref{fig:bumpsstubble} for an illustration.
\end{proof}
\begin{figure}[H]
    \begin{subfigure}{0.45\textwidth}
        \centering
         \includegraphics[width=\textwidth]{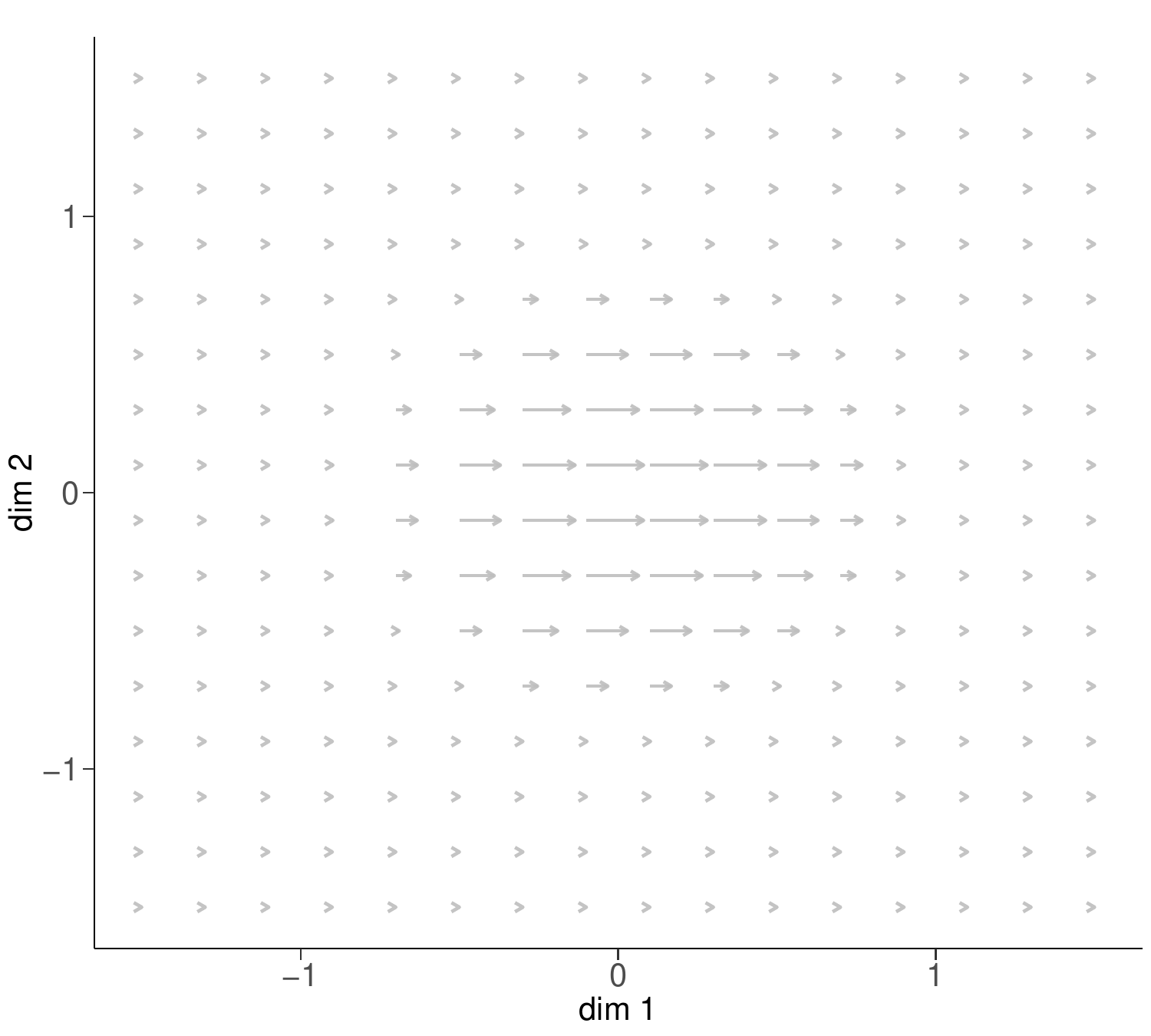}
            \caption{The model function $f_{z,r}$.}
        \label{fig:bumpvfstubble}
    \end{subfigure}
    \hfill
    \begin{subfigure}{.45\textwidth}
        \centering
         \includegraphics[width=\textwidth]{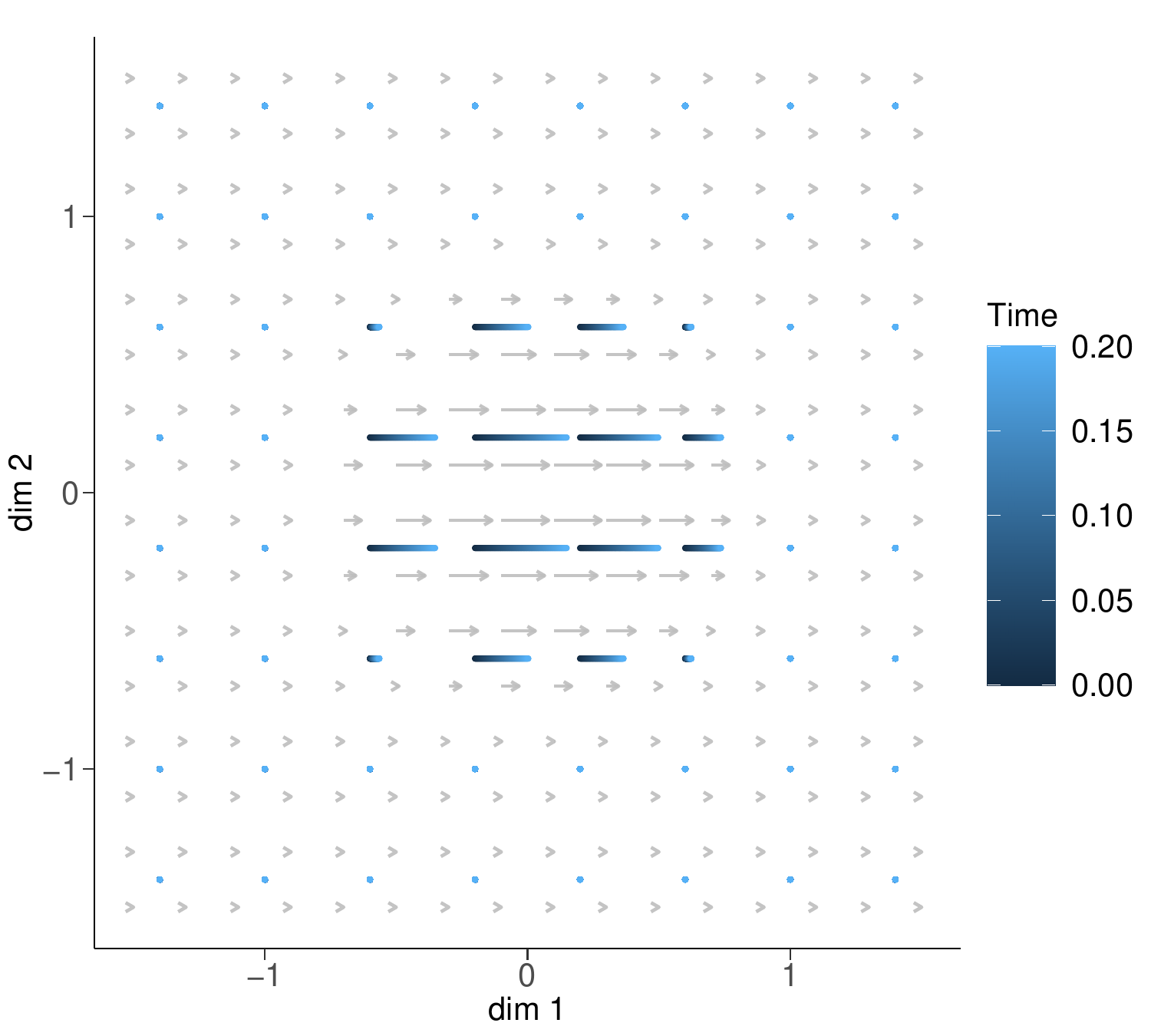}
            \caption{The model function $f_{z,r}$ with a grid of solutions.}
        \label{fig:bumptrajstubble}
    \end{subfigure}
    \caption{The model function used in the proof of error lower bounds in the stubble model. Object description as in \cref{fig:test}.}
    \label{fig:bumpsstubble}
\end{figure}
\subsection{Lower Bound -- Deterministic}
We assume here that we make observations at equally spaced timepoints. In this setting, we always have a non-vanishing error term -- even when observing infinitely many trajectories starting at every point in the domain of interest (and outside), and even when observing without any measurement noise. This is demonstrated with the following theorem. We show the existence of two different model functions that induce trajectories that coincide periodically. If this period is equal to the constant timestep of observation times, it is impossible to distinguish the two functions.
\begin{theorem}\label{thm:stubble:deterministic}
	Let $d\in\N$, $\beta\in\Rpo$, $\ell := \llfloor\beta\rrfloor$, $L_0,\dots,L_\ell,\Lbeta \in\Rpp$.
    Set $\FSmooth := \Sigma^{d\to d}(\beta, \indset{L}{0}{\ell}, \Lbeta)$.
    Let $p\in \Rpo$.
	Then there is $C\in\Rpp$ large enough, depending only $d, \beta, L_0, \dots, L_\ell, \Lbeta, p$, with the following property:
	Let $x_0\in\hypercube$ and $\stepsize \in\Rpp$ with $\stepsize \leq C^{-1}$.
	Then there are $f_0, f_1 \in \FSmooth$ such that
	\begin{equation}\label{eq:periodicallyIdentical}
		U(f_0, x, i\stepsize) = U(f_1, x, i\stepsize)
	\end{equation}
	for all $i\in\Z$, $x\in\R^d$ and
	\begin{equation}\label{eq:stubble:determinstic:separated}
		C \euclOf{f_0 - f_1}(x_0) \geq \stepsize^{\beta}
		\qquad\text{and}\qquad
		C \LpNormOf{p}{\hypercube}{\euclOf{f_0 - f_1}} \geq \stepsize^{\beta}
		\eqfs
	\end{equation}
\end{theorem}
\begin{remark}\mbox{}
    \begin{enumerate}[label = (\roman*)]
        \item
        \textbf{Error Tendencies.}
        In contrast to \cref{thm:stubble:probabilistic}, a large observation time $\Tmax$ (which is equivalent to a large constant time step $\stepsize$ between observations), increases the error lower bound. Here larger time steps mean more unobserved time in which the trajectories can deviate undetected.
    \end{enumerate}
\end{remark}
\cref{thm:stubble:deterministic} is a direct consequence of \cref{thm:stubble:deterministic:details}, see appendix \ref{sec:app:stubble:deterministic}, where the proof is given in detail. Here, we present a sketch of the proof.
\begin{proof}[Proof sketch of \cref{thm:stubble:deterministic}]
	It is enough to make this construction in one dimension, $d = 1$, as other dimension can be set to $0$. Furthermore, we can scale any construction easily to be appropriate for a given $\stepsize$ and $L_0$. Hence, we work with $\stepsize=1$ and $L_0$ large enough. We start by setting the first model function to be constant, $f_0(x) = 1$. The solutions are the linear functions $U(f_0, x, t) = x + t$. Next, we want to construct $f_1 \neq f_0$ with the property \eqref{eq:periodicallyIdentical}. We define a function $g \colon \R \to \R, x \mapsto x + h(x)$, where $h$ is a non-constant, periodic function with period 1. We choose $h$ such that $g$ is smooth and invertible. Then we set $f_1(x) = g\pr(g^{-1}(x))$. This yields $U(f_1, x, t) = g(g^{-1}(x) + t)$. As $h$ is periodic, one can show $g(g^{-1}(x) + i) = x + i$ for $i\in\Z$. The technical part concerning the smoothness conditions is left out here.
	An illustration of the construction is given in  \cref{fig:perioNonId}.
\end{proof}
\begin{figure}
	\begin{center}
		\includegraphics[height=0.3\textheight]{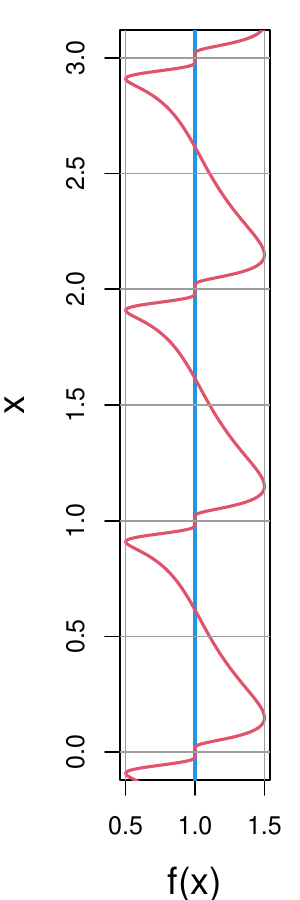}
		\quad
		\includegraphics[height=0.3\textheight]{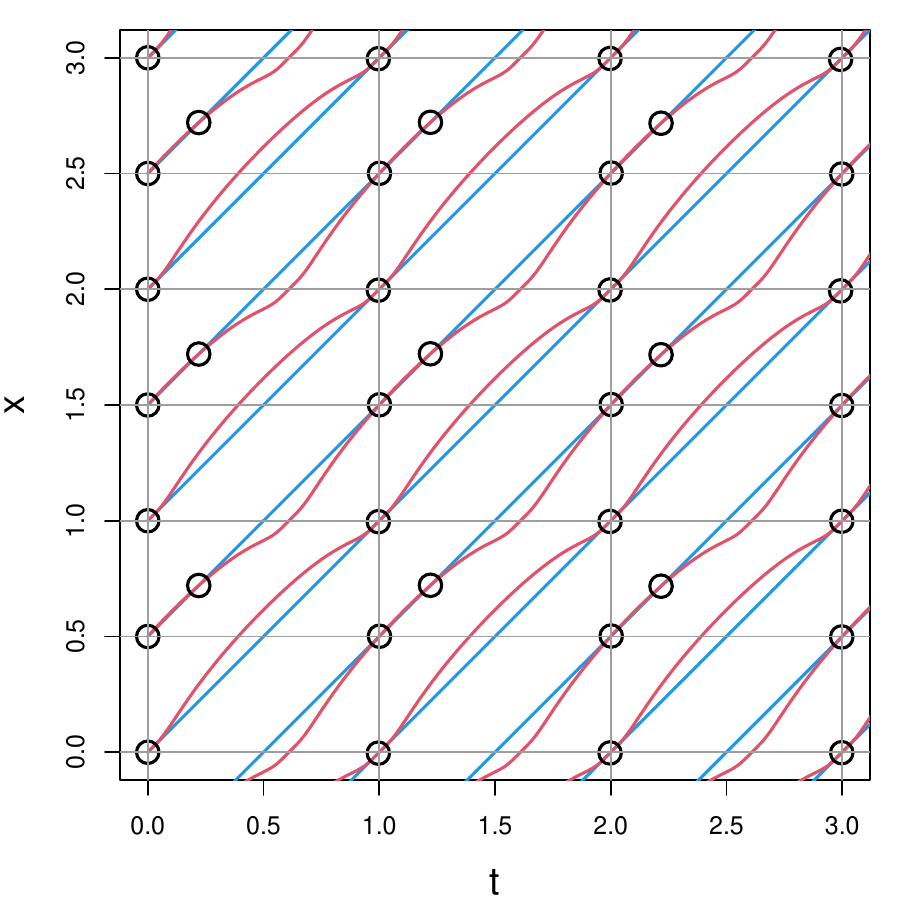}
	\end{center}
	\caption{Periodically intersecting but non-identical flows. In blue, we have a constant model function $f_0$ with linear trajectories $t\mapsto U(f_0, x, t)$.  In red, we have a periodic model function $f_1$ with trajectories $t\mapsto U(f_1, x, t)$ that periodically intersect those of $f_0$. The black circles mark intersections of the trajectories.}
	\label{fig:perioNonId}
\end{figure}
If the spacing of observation times is irregular, the construction in the proof of \cref{thm:stubble:deterministic} must fail as the next proposition shows.
\begin{proposition}\label{stubble:propo:varytimestep}
	Let $t_1, t_2\in\Rpp$ with $t_2/t_1\in \R\setminus\Q$.
	Assume, we have $f_0, f_1  \in \Sigma(1, L)$ for some $L\in\Rpp$ such that $U(f_0, x, t_i) = U(f_1, x, t_i)$ for $i=1,2$ and all $x\in\R$.
	Then $f_0 = f_1$.
\end{proposition}
\begin{proof}[Proof of \cref{stubble:propo:varytimestep}]
	For $k\in\Z$, $\ell\in\{0,1\}$, we have $U(f_\ell, x, kt_i) = U(f_\ell, U(f_\ell, x, (k-1)t_i), t_i)$. Using induction and $U(f_0, x, t_i) = U(f_1, x, t_i)$, we obtain
	$U(f_0, x, kt_i) = U(f_1, x, kt_i)$ for $i=1,2$, all $k\in\Z$, and all $x\in\R$. Hence, $U(f_0, x, k_1t_1+k_2t_2) = U(f_1, x, k_1t_1+k_2t_2)$ for all $k_1, k_2\in\Z$ and $x\in\R$.
	As $t_1/t_2$ is irrational, the set $\{k_1t_1 + k_2t_2 \colon k_1, k_2\in\mathbb Z\}$ is dense in $\R$ (Kronecker's theorem). Hence, $U(f_0, x, t) = U(f_1, x, t)$ for all $x,t\in\R$, which implies $f_0 = f_1$.
\end{proof}
\begin{remark}\mbox{ }
	It is unclear to the authors whether the failure to prove a similar lower error bound in the case of varying time steps is due to a different (faster) minimax error rate or just a lacking in the considered mathematical approach. In any case, the apparent difference between constant and varying timesteps for the statistical analysis seems intriguing due to its links to Diophantine approximation, i.e., the study of approximating real numbers by rational ones.
\end{remark}
\subsection{Combined Lower Bounds}
We combine the results of \cref{thm:stubble:probabilistic} and \cref{thm:stubble:deterministic} to one lower bound in \cref{cor:stubble:nice}.
To increase the clarity of the result, we use a more specific setting than is possible. In particular, we show results for the expected squared error, and we restrict ourselves to viewing $\nmax$ as a constant (not allowing $\nmax\to\infty$) and to equidistant observation times.
\begin{assumption}\mbox{ }
	\begin{itemize}
		\item \newAssuRef{EquidistantTime}: There is $\stepsize\in\Rpp$ such that $t_{j,i+1} - t_{j,i} = \stepsize$ for all $j\in\nnset m$, $i\in\nnset {n_j}$.
	\end{itemize}
\end{assumption}
\begin{corollary}\label{cor:stubble:nice}
	Use the model of \cref{sec:modelDef}.
	Assume \assuRef{Noise}, \assuRef{Cover}, and \assuRef{EquidistantTime}.
	Then there is $C\in\Rpp$ large enough, depending only on $d, \nmax, \beta, L_0, \dots, L_\ell, \Lbeta, \const{noise}, \const{cvr}$ with the following property:
	Assume
	\begin{equation}
		\stepsize \leq C^{-1}  \qquad \text{and}\qquad n \stepsize^2 \geq C\eqfs
	\end{equation}
	Then
	\begin{align}
		\forall x_0\in\hypercube\colon\ C \inf_{\festi} \sup_{\ftrue\in\FSmooth} \EOff{\ftrue}{\euclOf{\festi(x_0)-\ftrue(x_0)}^2}
		&\geq
		\br{n \stepsize^2}^{-\frac{2\beta}{2\beta+d}} + \stepsize^{2\beta}
		\eqcm\\
		C \inf_{\festi} \sup_{\ftrue\in\FSmooth} \EOff{\ftrue}{\sup_{x\in\hypercube}\euclOf{\festi - \ftrue}^2(x)}
		&\geq
		\br{\frac{n \stepsize^2}{\log\brOf{n \stepsize^2}}}^{-\frac{2\beta}{2\beta+d}} + \stepsize^{2\beta}
		\eqcm\\
		C \inf_{\festi} \sup_{\ftrue\in\FSmooth} \EOff{\ftrue}{\LpNormOf{2}{\hypercube}{\euclOf{\festi - \ftrue}}^2}
		&\geq
		\br{n \stepsize^2}^{-\frac{2\beta}{2\beta+d}} + \stepsize^{2\beta},
	\end{align}
    where the infima range over all estimators $\festi$ of $\ftrue$ based on the observations $Y_{\setIdx}$.
\end{corollary}

\begin{remark}\mbox{}\label{rem:stubble:nice}
    \begin{enumerate}[label = (\roman*)]
        \item
            \textbf{Minimax Optimality.}
            In the setting of \cref{cor:stubble:nice}, \cite[Corollary 3.13]{upperbounds} presents the upper bound
    		\begin{equation}
    			C^{-1}\EOff{\ftrue}{\euclOf{\festi(x_0)-\ftrue(x_0)}^2} \leq \br{n \stepsize^2}^{-\frac{2\beta}{2\beta+d}} + \stepsize^{2\beta}
    		\end{equation}
    		for $x_0\in\hypercube$ and some constant $C\in\Rpp$, and all $n\in\N$ large enough. This implies the minimax optimality of the presented estimation procedure proposed in \cite[Corollary 3.13]{upperbounds}.
        \item
            \textbf{Time Duration.}
            As $C$ in the corollary may depend on $\nmax$, we essentially assume that $\nmax$ is constant. Then, to compare with \cref{thm:stubble:probabilistic}, we see that $\Tmax = \nmax \stepsize$ is of the same order as $\stepsize$.
        \item
            \textbf{Exponent.}
            Roughly speaking, the function $f$ is $\beta$ smooth, making the trajectories $t\mapsto U(f, x, t)$ $(\beta+1)$-smooth due to the smoothing property of the forward map $f \mapsto U(f,x,\cdot)$. As a $\tilde \beta$-smooth object with $\tilde\beta:=\beta+1$ is observed, of which we want to estimate its first derivative, which has a $d$-dimensional domain, one may expect a rate of the order of
            \begin{equation}
                n^{-\frac{\tilde \beta - 1}{2\tilde \beta + d}} = n^{-\frac{\beta}{2(\beta+1) + d}}
            \end{equation}
            from classical nonparametric regression theory, see \cref{thm:lower:regression}. So, the exponent $\frac{\beta}{2\beta+d}$ may be surprising. But, after balancing the probabilistic and the deterministic terms, the expected term $2(\beta+1)+d$ indeed appears in the denominator, as we will see in \cref{cor:stubble:nice:onlyn}.
        \item
            \textbf{Balancing.}
            For balancing the two error terms, note that
    		\begin{equation}\label{eq:stubble:nice:compare}
    			\br{n \stepsize^2}^{-\frac{2\beta}{2\beta+d}} \geq \stepsize^{2\beta}
    		\end{equation}
    		if and only if
    		\begin{equation}
    			\stepsize \leq n^{-\frac{1}{2(\beta +1) + d}}
    			\eqfs
    		\end{equation}
  \end{enumerate}
\end{remark}
\begin{proof}[Proof of \cref{cor:stubble:nice}]
	The proof consists of first applying the probabilistic lower bound \cref{thm:stubble:probabilistic} and the deterministic lower bound \cref{thm:stubble:deterministic} together with the deterministic master theorem \cref{thm:lower:deterministic:master}. Then \cref{thm:lower:reduction:expectation} transfers results in probability to results in expectation. We have to check that all assumptions are fulfilled:

	As  $C^{-1} \geq \stepsize$, $n \stepsize^2 \geq C$ implies $n \geq C^3$. Thus,
	\begin{equation}
		m^{2\beta}
		\geq
		\br{\frac{n}\nmax}^{2\beta}
		\geq
		\br{\frac{C^3}{\nmax}}^{2\beta}
		=
		C\pr \nmax^3 C^{-2}
		\geq
		C\pr \nmax^3 \stepsize^2
		=
		C\pr \nmax \Tmax^2
	\end{equation}
	with
	\begin{equation}
		C\pr = \br{\frac{C^3}{\nmax}}^{2\beta} \br{\nmax^3 C^{-2}}^{-1}
		\eqfs
	\end{equation}
	Furthermore,
	\begin{equation}
		m \nmax \Tmax^2
		\geq
		n \nmax^2 \stepsize^2
		\geq
		C \nmax^2
		\geq
		C\prr
	\end{equation}
	with
	\begin{equation}
		C\prr := C \nmax^2
		\eqfs
	\end{equation}
	Thus, we can choose $C$ large enough, so that the conditions for \cref{thm:stubble:probabilistic} are fulfilled.
	As we assume $\stepsize \leq C^{-1}$ we can also apply \cref{thm:stubble:deterministic}.
\end{proof}
Based on a trade-off of the error terms in \cref{cor:stubble:nice}, the following corollary yields lower bounds for the estimation risk independent of $\stepsize$.
\begin{corollary}\label{cor:stubble:nice:onlyn}
	Use the setting of \cref{cor:stubble:nice}.
	Then
	\begin{align}
		C \inf_{\festi} \sup_{\ftrue\in\FSmooth} \EOff{\ftrue}{\euclOf{\festi(x_0)-\ftrue(x_0)}^2}
		&\geq
		n^{-\frac{2\beta}{2(\beta +1) + d}}
		\eqcm\\
		C \inf_{\festi} \sup_{\ftrue\in\FSmooth} \EOff{\ftrue}{\sup_{x\in\hypercube}\euclOf{\festi - \ftrue}^2(x)}
		&\geq
		\br{\frac{n}{\log(n)}}^{-\frac{2\beta}{2(\beta+1)+d}}
		\eqcm\\
		C \inf_{\festi} \sup_{\ftrue\in\FSmooth} \EOff{\ftrue}{\LpNormOf{2}{\hypercube}{\euclOf{\festi - \ftrue}}^2}
		&\geq
		n^{-\frac{2\beta}{2(\beta +1) + d}},
	\end{align}
    where the infima range over all estimators $\festi$ of $\ftrue$ based on the observations $Y_{\setIdx}$.
\end{corollary}
\begin{remark}\mbox{ }\label{rem:stubble:nice:onlyn}
	\begin{enumerate}[label=(\roman*)]
		\item
        \textbf{Generality.}
        These bounds are the minima of the lower bounds of \cref{cor:stubble:nice} and hold for all $\stepsize$ that fulfill the conditions in \cref{cor:stubble:nice}.
        \item
        \textbf{Exponent.}
        As mentioned in \cref{rem:stubble:nice}, the exponents in the error rate are intuitively meaningful when compared to the minimax rates of nonparmetric regression (see \cref{sec:nonparamreg}):
        \begin{itemize}
            \item The $2$ in the numerator is due to considering the squared error.
            \item The $\beta$ in the numerator is the smoothness of the observed object $U(f,x,\cdot)$, which is $\beta+1$, minus one for the first derivative of that object that is estimated in the form of $f$.
            \item The $2$ in the denominator is due to condition \assuRef{Noise}, see also \cref{rem:noise}.
            \item The term $\beta+1$ in the denominator is the smoothness of the observed object.
            \item The $d$ in the denominator is the dimension of the domain of the estimated object $f$.
        \end{itemize}
	\end{enumerate}
\end{remark}
\begin{proof}[Poof of \cref{cor:stubble:nice:onlyn}]
	Set
	\begin{equation}
		\stepsize = n^{-\frac{1}{2(\beta +1) + d}}
	\end{equation}
	for the bounds at a point and in $L^2$ and
	\begin{equation}
		\stepsize = \br{\frac{n}{\log(n)}}^{-\frac{1}{2(\beta+1)+d}}
	\end{equation}
	for the bound in sup-norm.
	Then apply
	\cref{cor:stubble:nice}.
	Note that these choices minimize the bounds presented in \cref{cor:stubble:nice} with respect to $\stepsize$. Thus, the lower bounds in \cref{cor:stubble:nice:onlyn} are true for general $\stepsize$.
\end{proof}

%% file: sec_snake.tex
\section{Snake Model}\label{sec:snake}
In this section, we consider the model described in \cref{sec:modelDef} with the target of estimating $\ftrue$ on the domain of interest $\hypercube$. We make a restriction on the location of the observed trajectories in the state space to avoid trivial settings in which consistent estimation is impossible. We call the resulting model the \textit{snake model}. In this model, we derive two different types of lower bounds -- one \textit{deterministic}, which is independent of the noise and sample size $n$ and another \textit{probabilistic} depending on these objects. Then the rates are combined for a full picture. The obtained lower bounds on the error rates are shown to be minimax optimal in at least one specific setting of interest.
\subsection{Model}\label{sec:snake:model}
If there is a location in state space which has a large distance to every point of the observed trajectories, we are not able to estimate the model function at this location well. In other words, the error rate for estimation of the model function $\ftrue$ depends on how densely the trajectories cover the domain of interest. We make this dependence visible by designing a parameter class that not only requires a minimum smoothness of $\ftrue$, but also requires a bound $\delta$	 on the maximal distance of any point in the domain of interest to the observed trajectories. Clearly, longer trajectories have a higher potential to fulfill the latter requirement. Moreover, the length of a trajectory $u_j := U(\ftrue, x_j, \cdot)\colon[0,T_j]\to\R^d$ can be bounded from above by the product of its maximum speed $\supNormof{\dot u_j}$ and the length of time interval $T_j$. If the corresponding model function is in $\FSmooth = \Sigma^{d\to d}(\beta, \indset{L}{0}{\ell}, \Lbeta)$, we have $\supNormof{\dot u_j} \leq L_0$. With these considerations in mind, we design the parameter class depending on the smoothness parameters $\beta, \indset{L}{0}{\ell}, \Lbeta$, the covering parameter $\delta$, and the total observed time $\Tsum = \sum_{j=1}^m T_j$.

In contrast to the stubble model, we do not prescribe initial conditions, but make them part of the parameter vector that describes the probability distribution of observations: For $m\in\N$ trajectories, the parameter vector $\theta = (f,\indset{x}{1}{m},\indset{T}{1}{m}) \in \FSmooth\times (\R^d)^m\times(\Rpp)^m$ contains the model function $f$, the initial conditions $\indset{x}{1}{m}$ and the maximum observation time $\indset{T}{1}{m}$ for each trajectory. Our observations are described by $\ptrue = (\ftrue,\indset{x}{1}{m},\indset{T}{1}{m})$ and the equation
\begin{equation}\label{eq:snake:lower:data}
	Y_{j,i} = U(\ftrue, \xany j, t_{j,i}) + \noise_{j,i}\eqcm \ i\in\nnset{n_j}, j\in\nnset{m}
\end{equation}
with -- as before -- independent, identically distributed, and centered noise $\noise_{j,i}\sim P^\noise$. We denote by $\Pr_{\ptrue}$ and $\Eoff{\ptrue}{\cdot}$ the corresponding probability measure and expectation, respectively.

As suggested above, to provide lower bounds on estimating $\ftrue$, we need some restrictions on $\ptrue$.
To describe the set of admissible parameters, we need the following definitions.
\begin{definition}[Tube]\label{def:tube}\mbox{}
	\begin{enumerate}[label=(\roman*)]
		\item
		Let $d\in\N$, $T\in\Rpp$. Let $u\colon[0,T]\to\R^d$ be continuously differentiable and denote its derivative as $\dot u\colon[0,T]\to\R^d$.
		Let $\delta\in\Rpp$. Define the (closed) \emph{tube} with radius $\delta$ around $u$ as
		\begin{equation}\label{eq:def:tube}
			\tube\brOf{u, \delta} := \bigcup_{t\in[0, T]} \setByEle{u(t) + v\,}{
				\begin{array}{l}
					\dot u(t) \neq 0,\cr
					 v \tr \dot u(t) = 0,\cr
					 \euclOf{v} \leq \delta
				\end{array}
			}
			\eqfs
		\end{equation}
		\item
		Let $d,m\in\N$, $f\colon\R^d\to\R^d$ Lipschitz continuous, $x_1, \dots, x_m \in \R^d$, and $T_1, \dots, T_m\in\Rpp$.
		Let $\theta := (f, \indset{x}{1}{m}, \indset{T}{1}{m})$. Let $u_j := U(\ftrue, x_j, \cdot)\colon[0,T_j]\to\R^d$. Let $A\subset \R^d$. Define the \emph{tube distance} as
		\begin{equation}
			d_{\mc T}(A, \theta) := \inf\setByEle{ \delta\in\Rpp}{ A\subset \bigcup_{j=1}^m\tube\brOf{u_j, \delta}}
			\eqfs
		\end{equation}
	\end{enumerate}
\end{definition}

\begin{remark}\label{rem:snake:lower:tube}
	Our definition here yields $\tube(u, \delta) = \emptyset$ if $\dot u = 0$. If we would remove the condition $\dot u(t) \neq 0$ from \eqref{eq:def:tube}, the model function $f = 0$, would be able to cover the hypercube with a sufficient amount of trajectories $m$, initial conditions on a uniform grid, and $T_j = 0$. But $f(x) = \epsilon \mo{e}_1$ with sufficiently small $\epsilon\in\Rpp$ would not fulfill this condition. Our convention avoids this discontinuity. For an analysis of settings with a large number of very short trajectories, see the stubble model.
\end{remark}
Now we can define the parameter class $\ParamSnake$, for which we show error lower bounds in this section.
\begin{definition}[Snake Parameter Class]\label{def:snake:smoothnessclass}
	Let $d\in\N_{\geq 2}$ and $\beta,L_0, \dots, L_\ell,\Lbeta,\delta,\Tsum\in\Rpp$. Set $\FSmooth := \Sigma^{d\to d}(\beta, \indset{L}{0}{\ell}, \Lbeta)$. Define
	\begin{equation}
		\ParamSnakeBase{d,\beta}{m}:=\ParamSnakeBase{d,\beta}{m}(\indset{L}{0}{\ell}, \Lbeta) :=\FSmooth\times(\R^d)^m \times \Rpp^m
	\end{equation}
	and
	\begin{align}
		\ParamSnake
		&:=
		\ParamSnake(\indset{L}{0}{\ell}, \Lbeta, \delta, \Tsum)
		\\&:=
		\setByEle{\theta = (f, \indset{x}{1}{m}, \indset{T}{1}{m})\in \ParamSnakeBase{d,\beta}{m} }{
			\begin{array}{l}
				m\in\N,\cr
				\sum_{j\in\nnset{m}} T_j = \Tsum,\cr
				d_\tube(\hypercube, \theta) \leq \delta
			\end{array}
		}\eqfs
	\end{align}
	For $\theta = (f, \indset{x}{1}{m}, \indset{T}{1}{m})\in \ParamSnake$, denote $f_\theta := f$.
\end{definition}
\begin{remark}\label{rem:snake:lower:paramset}
    The volume of a tube in dimension $d\in\N_{\geq 2}$ of radius $\delta\in\Rpp$ is bounded from above by the volume of sphere with radius $r$ in $d-1$ dimensions, i.e., $c_d\delta^{d-1}$, times the length of the tube. The length of the tubes generated by trajectories $U(f, x_j, \cdot )\colon [0, T_j] \to\R^d$, $j\in\nnset m$ is bounded from above by $L_0 \Tsum$, where $L_0:=\supNormof{f}$ and $\Tsum :=  \sum_{j\in\nnset m} T_j$.
	Thus, we require $c_d\delta^{d-1} L_0 \Tsum \geq 1$ to make it possible to achieve $d_\tube(\hypercube, \theta) \leq \delta$ for $\theta = (f, \indset{x}{1}{m}, \indset{T}{1}{m})$. In other words, if $L_0 \Tsum < \br{c_d\delta^{d-1}}^{-1}$, then $\ParamSnake(\indset{L}{0}{\ell}, \Lbeta, \delta, \Tsum) = \emptyset$.
\end{remark}
\begin{figure}[H]
    \begin{subfigure}{0.45\textwidth}
        \centering
         \includegraphics[width=\textwidth]{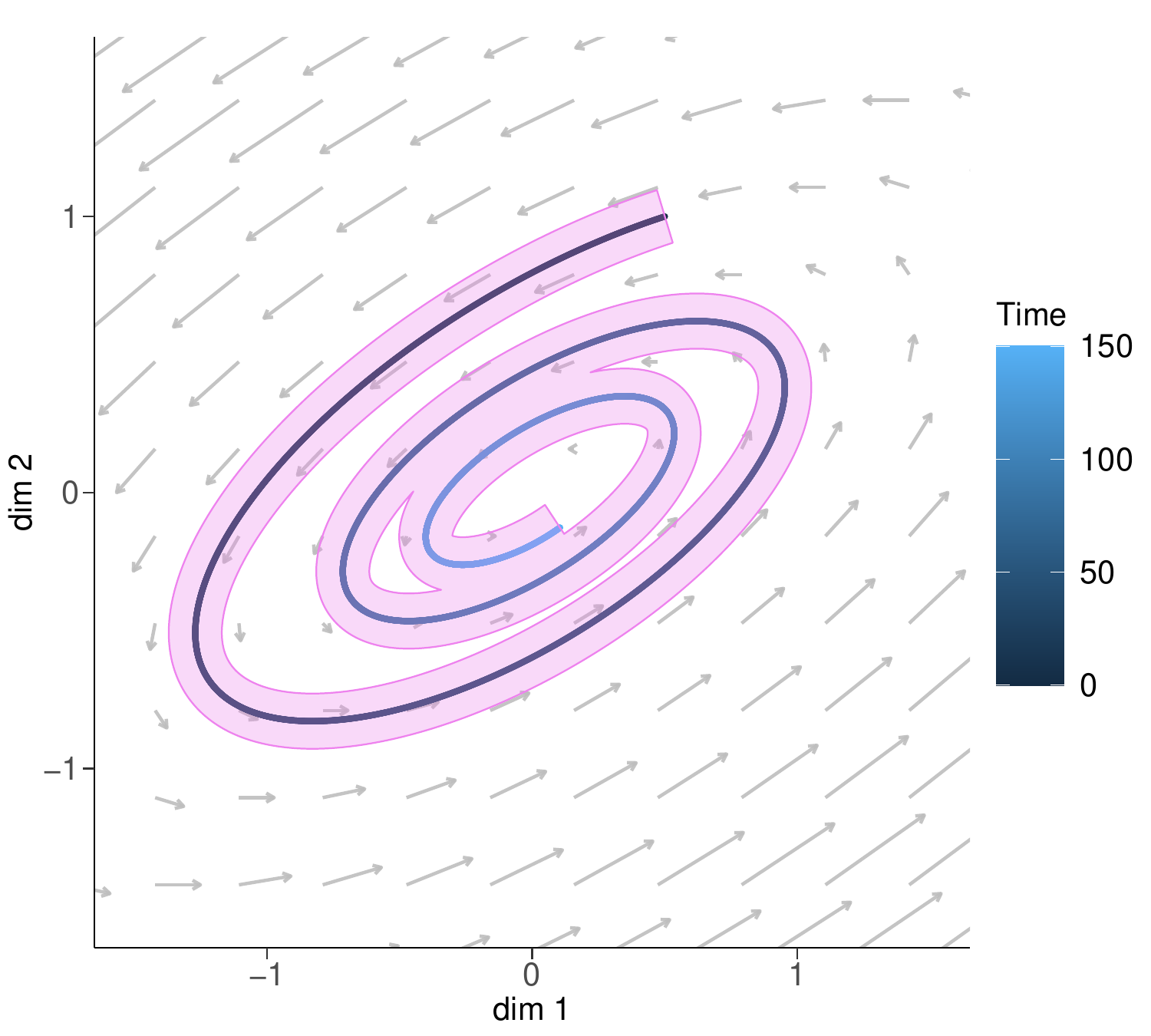}
          \caption{An example of a tube for a given trajectory.}

        \label{fig:tubedef}
    \end{subfigure}
    \hfill
    \begin{subfigure}{.45\textwidth}
        \centering
         \includegraphics[width=\textwidth]{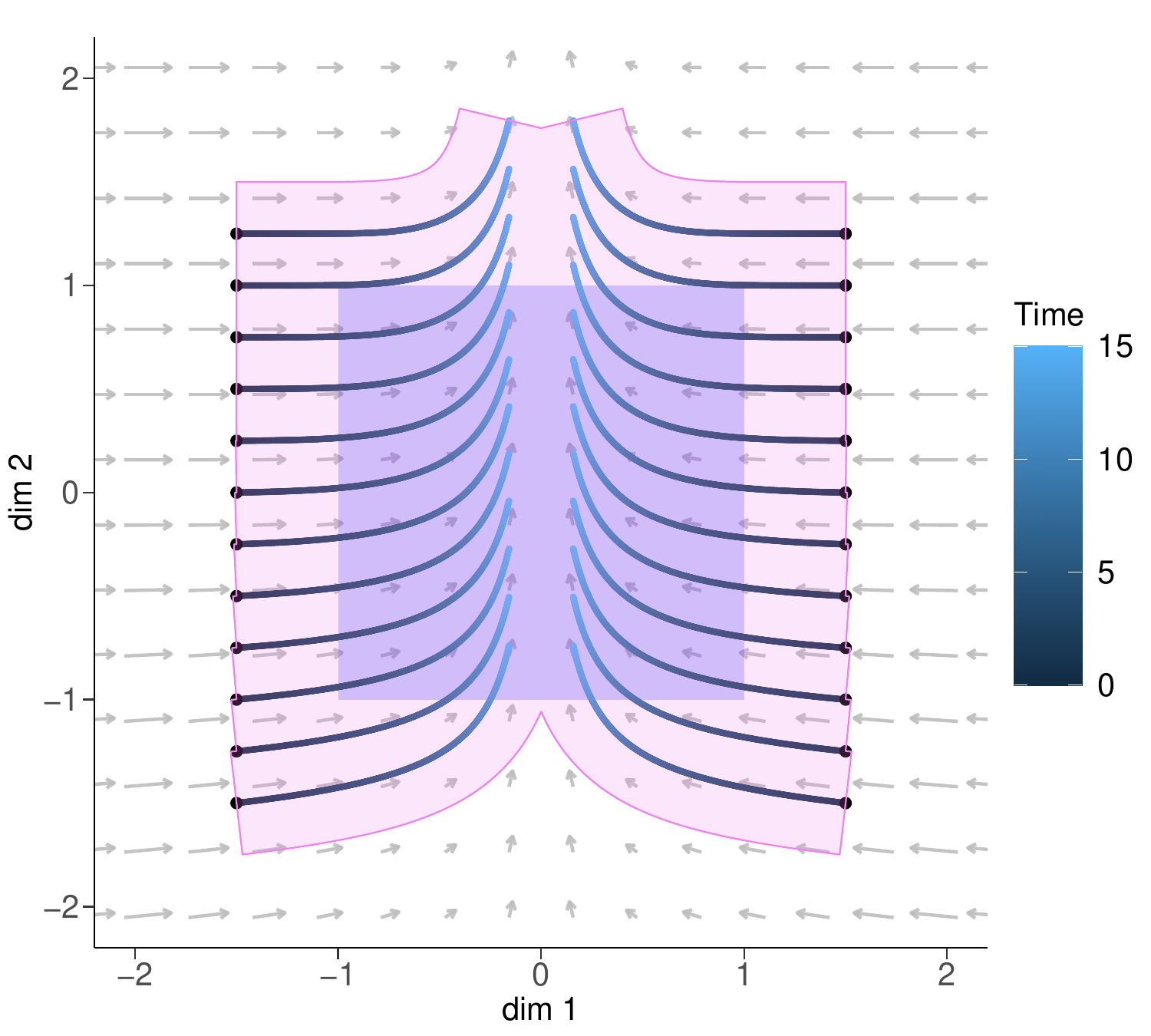}
            \caption{Continuing illustration \cref{fig:sub2}, the initial values, time horizons and model function fulfill the covering property with $\delta = 0.25$.}
        \label{fig:tubesnake}
    \end{subfigure}
    \caption{Illustration of tubes as defined in \cref{def:tube}. Object description as in \cref{fig:test}. Areas shaded in pink are covered by tubes around the trajectories.}
    \label{fig:tubes}
\end{figure}
\subsection{Lower Bound -- Probabilistic}
We consider the model described by \cref{eq:snake:lower:data}. We assume that the data points do not accumulate too much within an area, which is guaranteed by the following assumption.
\begin{assumption}\label{ass:snake:lower:cover}\mbox{ }
	\begin{itemize}
		\item \newAssuRef{CoverTime}:
			There is $\const{cvrtm}\in\Rpp$ with the following property:
			For all $j\in\nnset{m}$, $a,b\in[0,T_j]$ with $a < b$, we have
			\begin{equation}
				\sum_{i\in\nnset{n_j}}\mathds{1}_{[a,b]}(t_{i,j}) \leq \max\brOf{1, \const{cvrtm} (b-a)\frac{n}{\Tsum}}
				\eqfs
			\end{equation}
	\end{itemize}
\end{assumption}
\begin{remark}\mbox{}
    \begin{enumerate}[label=(\roman*)]
        \item
        \textbf{Example.}
        \assuRef{EquidistantTime} implies \assuRef{CoverTime} with $\const{cvrtm} = 3$.
        \item
        \textbf{Interpretation.}
       	The term $n/\Tsum$ is the average time destiny of the observations. The condition requires that the observation density is never too far above average, which would lead to potentially better estimates at that location. Compare with \assuRef{Cover}, which requires a similar bound in space instead of time.
    \end{enumerate}
\end{remark}
We have the following lower bound results.
\begin{theorem}\label{thm:snake:probabilistic}
	Use the model of \ref{sec:modelDef} with parameter class $\ParamSnake$ from \cref{def:snake:smoothnessclass}.
	Assume $d\geq 2$, \assuRef{Noise}, and \assuRef{CoverTime}.
	Let $p\in\Rpo$.
	Then there is a constant $C\in\Rpp$ large enough, depending only on $\beta, d, L_0, \dots, L_\ell, L_{\beta}, \const{noise}, \const{cvrtm}, p$, with the following property:
	Assume
	\begin{equation}
        C \max\brOf{\br{\frac {\Tsum}n }^{2\beta+d+1},\,\Tsum^{-1}} \leq \delta^{d-1} \leq C^{-1} \min\brOf{1,\, \frac n{\Tsum}}
    \end{equation}
    Then
    \begin{align}
        \forall x_0\in\hypercube\colon\ \inf_{\festi} \sup_{\ptrue\in\ParamSnake} \Pr_{\ptrue}\brOf{C \euclOf{\festi - f_{\ptrue}}(x_0) \geq \br{ \delta^{-(d-1)} \Tsum^{-1} n}^{-\frac{\beta}{2(\beta+1)+d}}}
        &\geq\frac14 \eqcm\\
		\inf_{\festi}
        \sup_{\ptrue\in\ParamSnake}
        \Pr_{\ptrue}\brOf{
             C \sup_{x\in\hypercube}\euclOf{\festi - f_{\ptrue}}(x)
             \geq
             \br{\frac{\delta^{-(d-1)} \Tsum^{-1} n}{\log\brOf{\delta^{-(d-1)} \Tsum^{-1} n}}}^{-\frac{\beta}{2(\beta+1)+d}}
        }
		&\geq\frac14 \eqcm\\
		\inf_{\festi} \sup_{\ptrue\in\ParamSnake}\Pr_{\ptrue}\brOf{C \LpNormOf{p}{\hypercube}{\euclOf{\festi - f_{\ptrue}}} \geq \br{\delta^{-(d-1)} \Tsum^{-1} n}^{-\frac{\beta}{2(\beta+1)+d}}}
		&\geq\frac14\eqcm
    \end{align}
    where the infima range over all estimators $\festi$ of $f_{\ptrue}$ based on the observations $Y_{\setIdx}$.
\end{theorem}
\begin{remark}\mbox{}
    \begin{enumerate}[label=(\roman*)]
    	\item
    		\textbf{Dimension.}
    		In our construction, we require $d\geq 2$. Furthermore, the tubes of \cref{def:tube} are only meaningful in $d\geq 2$. Note that $d = 1$ is a special case, where all solutions $U(f, x, \cdot)$ are monotone functions.
    	\item
    		\textbf{Equidistant Time.}
    		If we assume \assuRef{EquidistantTime}, then \assuRef{CoverTime} is fulfilled and $\Tsum^{-1} n = \stepsize^{-1}$.
        \item
            \textbf{Relation Assumptions.}
            The lower bounds are formulated in terms of the overall sample size $n$, the total observation time $\Tsum$, and the radius $\delta$ of the tube which is required to cover the domain of interest. A typical setting we may be interested in is
            \begin{equation}
                \delta \xrightarrow{n\to\infty} 0
                \text{ and }
                \Tsum  \xrightarrow{n\to\infty} \infty
                \text{ and }
                \frac{n}{\Tsum} \xrightarrow{n\to\infty} \infty
                \eqfs
            \end{equation}
            Then the assumptions reduce to the requirement that $\delta \xrightarrow{n\to\infty} 0$ is in some sense slower than  $\Tsum/n \xrightarrow{n\to\infty} 0$ and $\Tsum^{-1} \xrightarrow{n\to\infty} 0$. For example,
            \begin{equation}
                \delta = n^{-\frac{1}{2(\beta+1)+d}}\text{ and } \Tsum = C \delta^{-(d-1)} = C n^{\frac{d-1}{2(\beta+1)+d}}
            \end{equation}
            fulfill the requirements. This choice becomes relevant when balancing against the deterministic error presented below.
        \item
            \textbf{Error Tendencies.} We can understand $\Tsum^{-1} n$ as the observation density in time. If it increases, more accurate estimations can be made, reducing the error lower bound. The term $\delta^{-(d-1)}$ can be interpreted as the minimum required length of the (combined) trajectories that is needed to cover the domain of interest with $\delta$-tubes. A smaller $\delta$ means a smaller parameter class $\ParamSnake$, decreasing the error lower bound.
        \item
            \textbf{Exponent.}
            In comparison to the lower bounds in \cref{thm:stubble:probabilistic}, the exponent for the error term \cref{thm:snake:probabilistic} does match the intuitions for nonparametric regression models, cf.\ \cref{rem:stubble:probabilistic}. If the available time $\Tsum$ is used in an optimal way, one can balance $\Tsum = c \delta^{-(d-1)}$ so that the exponent stays the same when considering a lower error bound that only depends on $n$.
    \end{enumerate}
\end{remark}
\cref{thm:snake:probabilistic} is a direct consequence of \cref{thm:snake:probabilistic:details}, see appendix \ref{sec:app:snake:probabilistic}, where the proof is given in detail. Here, we present a sketch of the proof.
\begin{proof}[Sketch of Proof of \cref{thm:snake:probabilistic}]
    As for \cref{thm:stubble:probabilistic}, we use our master theorem for regression lower bounds, \cref{thm:lower:master}, to prove \cref{thm:snake:probabilistic}.

    We use the null hypothesis that the true model function is
    \begin{equation}
        f_0(x) = (L_0 , 0, \dots, 0)\tr\eqfs
    \end{equation}
    The resulting trajectories are linear in the first dimension and constant in all other dimensions. The alternatives are again based on scaled and shifted version of a prototype $\hpulse{d}{\beta} \colon \R^d \to \R$. Instead of a \textit{bump} in the model function as for the stubble model, we now want the bump to be in the trajectories. As the model function is essentially the derivative of the trajectory, we need a \textit{pulse}\footnote{Here a \textit{pulse} refers to a smooth function $\R\to\R$ with compact support that first increases into positive values, than decreases becoming negative, and then increases again until it is 0.} that is added to $f_0$ to obtain the alternative model functions. See \cref{fig:proofsketch:bumps} for an illustration. The pulse function $\hpulse{d}{\beta}$ is based on a symmetric kernel with compact support and its derivative. See \cref{lmm:bumpandpulse} for a definition.

    For a location $z\in\R^d$ and a scale $r\in\Rpp$, we define the alternative model function as
    \begin{equation}
        f_{z,r}(x) = \br{L_0, \Lbeta r^\beta\hpulse{d}{\beta}\brOf{\frac{x-z}{r}},\, 0,\, \dots,\, 0}\eqfs
    \end{equation}
    The pulse is only acting in the second dimension so that the trajectories move with constant speed in the first dimension. It is chosen so that the resulting trajectories are identical to those of $f_0$ outside the support of the distortion of the model function induced by $\hpulse{d}{\beta}$.

    On one hand, we need to make sure that the observed trajectories cover the domain of interest. On the other hand, we obtain the best lower bound if as many observations as possible are uninformative. So, if $\Tsum$ is larger than necessary to cover the domain of interest, the remaining time is \textit{wasted}, i.e., it is used to extend the first trajectory outside the domain of interest. See the top trajectory in \cref{fig:proofsketch:bumps}.
\end{proof}
\begin{figure}[H]
    \begin{subfigure}{0.45\textwidth}
        \centering
        \includegraphics[width=\textwidth]{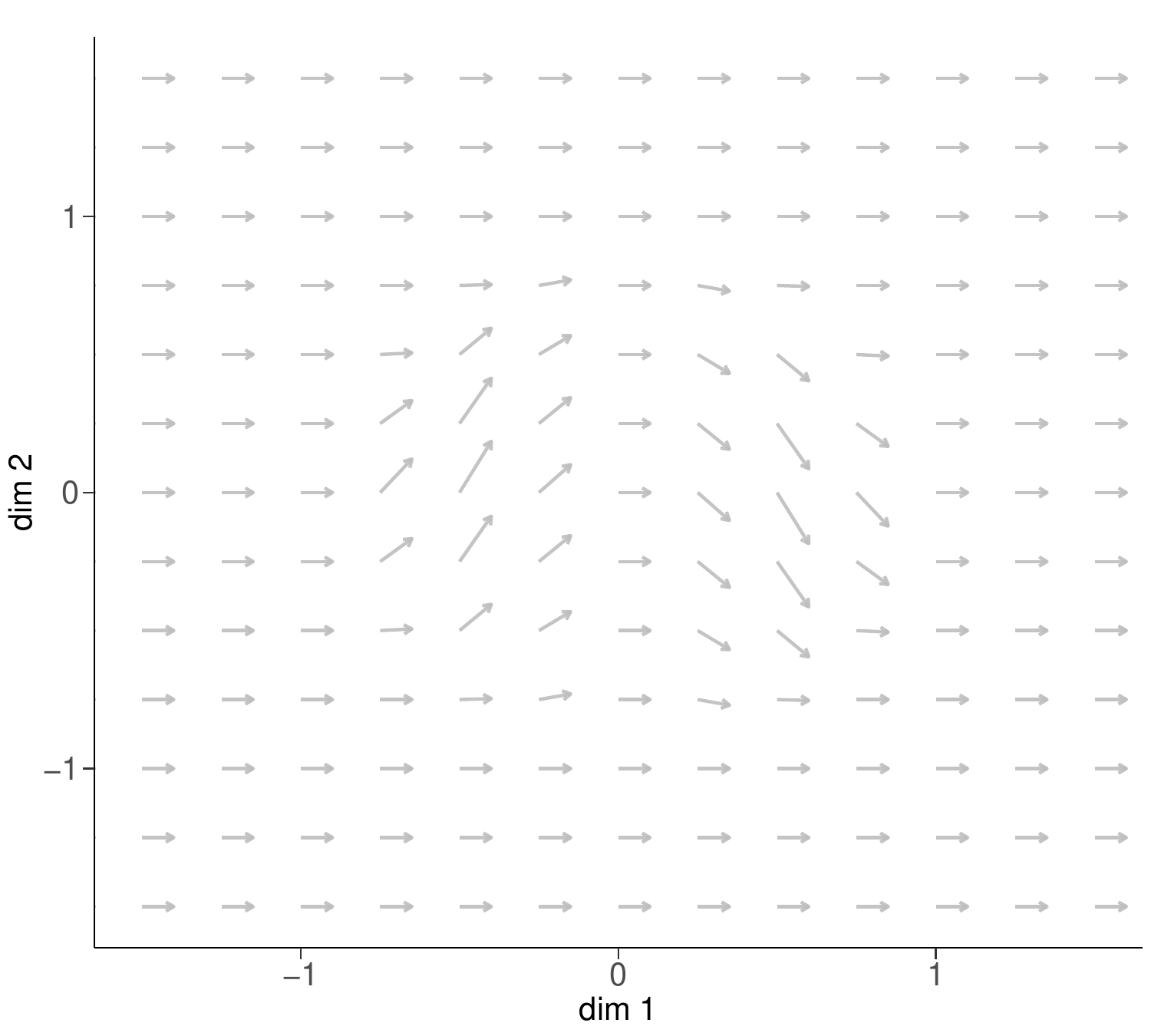}
        \caption{The model function $f_{z,r}$.}
        \label{fig:bumpvf}
    \end{subfigure}
    \hfill
    \begin{subfigure}{.45\textwidth}
        \centering
        \includegraphics[width=\textwidth]{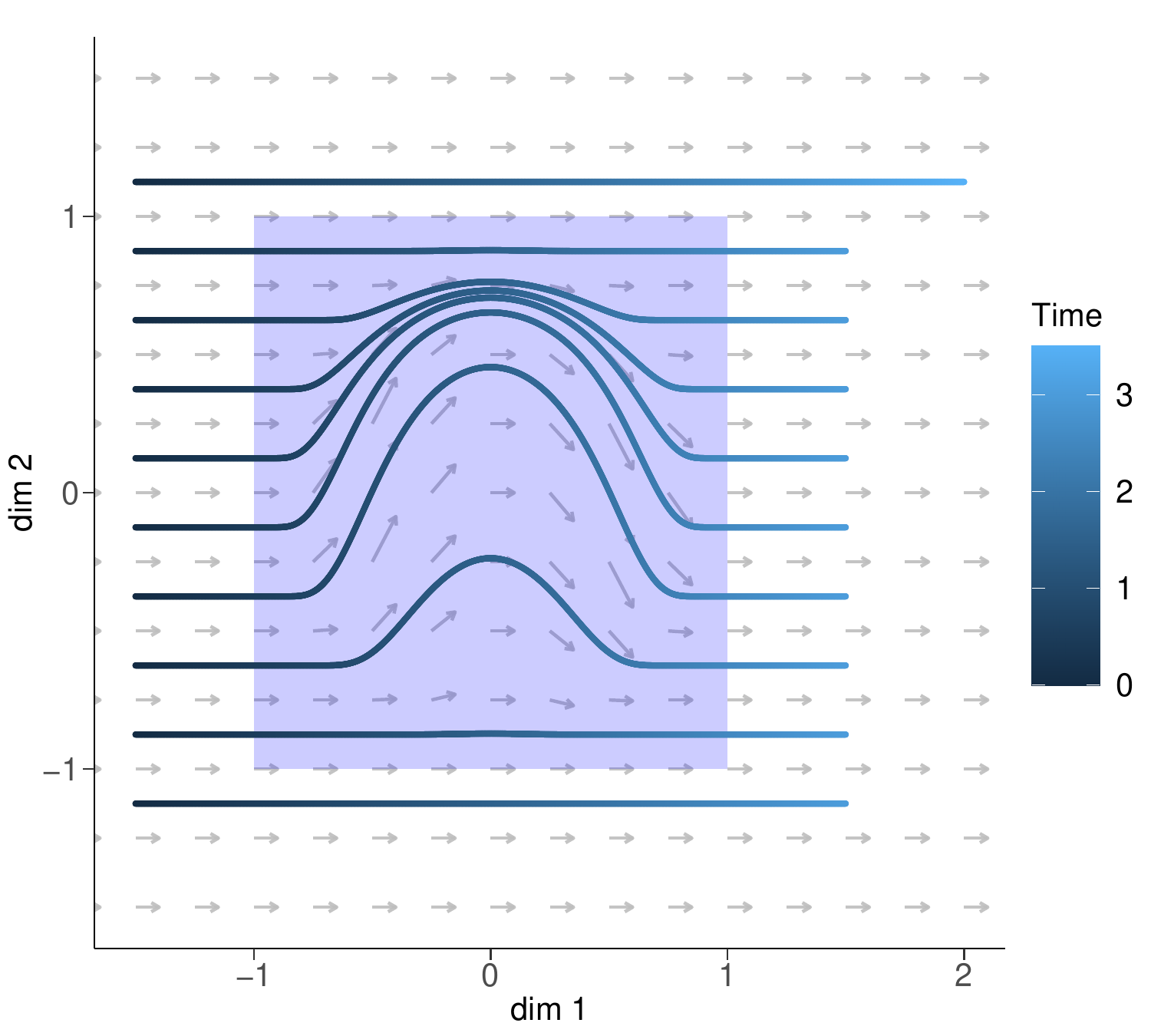}
        \caption{The model function $f_{z,r}$ with a row of solutions.}
        \label{fig:bumptraj}
    \end{subfigure}
    \caption{The model function used in the proof of error lower bounds in the snake model. Object description as in \cref{fig:test}.}
    \label{fig:proofsketch:bumps}
\end{figure}
\subsection{Lower Bound -- Deterministic}
With $\ptrue\in\ParamSnake$, we always have an error term of order $\delta^\beta$ -- even when observing trajectories continuously in time over an arbitrarily long time range, and even when observing without any measurement noise. This is due to the points in the domain of interest where the closest trajectory is a distance $\delta$ away.
This is demonstrated with the following theorem. We show the existence of two different model functions that induce trajectories that coincide for certain initial conditions and cover the domain of interest as required. It is impossible to distinguish the two model functions from observations.
\begin{theorem}\label{thm:snake:deterministic}
	Let $d\in\N_{\geq 2}$, $\beta\in\Rpo$, $\ell := \llfloor\beta\rrfloor$, $L_0,\dots,L_\ell,\Lbeta \in\Rpp$.
	Set $\FSmooth := \Sigma^{d\to d}(\beta, \indset{L}{0}{\ell}, \Lbeta)$.
	Let $p\in \Rpo$.
	Then there is $C\in\Rpp$ large enough, depending only $d, \beta, L_0, \dots, L_\ell, \Lbeta, p$, with the following property:
	Let $x_0\in\hypercube$ and $\delta\in\Rpp$ with $\delta \leq C^{-1}$.
	Assume $C \Tsum^{-1} \leq \delta^{d-1}$.
	Then there are $\theta_0, \theta_1 \in \ParamSnake$ of the form $\theta_k = (f_k, \indset{x}{1}{m},\indset{T}{1}{m})$, $k\in\nnzset1$ such that
	\begin{equation}\label{eq:snake:sameTraj}
		U(f_0, x_j, t) = U(f_1, x_j, t)
	\end{equation}
	for all $t\in\R$, $j \in \nnset m$ and
	\begin{equation}
		C \euclOf{f_0 - f_1}(x_0) \geq \delta^\beta
		\qquad\text{and}\qquad
		C \LpNormOf{p}{\hypercube}{\euclOf{f_0 - f_1}} \geq \delta^\beta
		\eqfs
	\end{equation}
\end{theorem}
\begin{remark}\mbox{}
    \begin{enumerate}[label=(\roman*)]
        \item
        \textbf{Relation Assumptions.}
        If a trajectory has constant positive speed, it needs at least $c_d \delta^{-(d-1)}$ time to cover the domain of interest $[0,1]^d$ with $\delta$-tubes. Thus, the assumption $C \Tsum^{-1} \leq \delta^{d-1}$ is quite natural if the speed of trajectories has na upper bound, which it has as $\supNormOf{f} \leq L_0$. Note that this assumption is also present in the probabilistic bound, \cref{thm:snake:probabilistic}.
        \item
        \textbf{Comparison to Stubble.}
        Compare this result to \cref{thm:stubble:deterministic}, where we have 	$U(f_0, x, i\stepsize) = U(f_1, x, i\stepsize)$ instead of $U(f_0, x_j, t) = U(f_1, x_j, t)$. This illustrates again that the snake and the stubble models are complementary.
    \end{enumerate}
\end{remark}
\cref{thm:snake:deterministic} is a direct consequence of \cref{thm:snake:deterministic:details}, see appendix \ref{sec:app:snake:deterministic}, where the proof is given in detail. Here, we present a sketch of the proof.
\begin{proof}[Sketch of Proof of \cref{thm:snake:deterministic}]
    We set
    \begin{equation}
        f_0(x) = (L_0 , 0, \dots, 0)\tr
        \eqfs
    \end{equation}
    We then create initial conditions $x_j$ with 0 in the first dimension, $\Pi_1 x_j = 0$, and so that the values in the other dimensions $\{\Pi_{-1}x_1, \dots, \Pi_{-1}x_m\} \subset \R^{d-1}$ form a regular grid in with grid size $2r$, $r = \delta/\sqrt{d}$. (See \cref{not:analytical} \ref{not:analytical:projection} for a definition of the projector operators.) Now we add bumps of the form
    \begin{equation}
        f_{2,z,r}(x) = \Lbeta r^\beta\hbump{d}{\beta}\brOf{\frac{x-z}{r}}
    \end{equation}
    in the second dimension to $f_0$ to obtain $f_1$ as
    \begin{equation}
        f_1(x) = f_0(x) + \sum_{k\in\Z^d} f_{2,z+2rk,r}(x)\mo{e}_1
        \eqfs
    \end{equation}
    The precise locations of the bumps and the initial conditions are chose so that the trajectories $t \mapsto U(f_1, x_j, t)$ do not intersect the support of the bumps so that \eqref{eq:snake:sameTraj} is fulfilled. Furthermore, the grid is spaced, so that the tubes of radius $\delta$ around the trajectories cover the domain of interest. See \cref{fig:grid} for an illustration.
\end{proof}
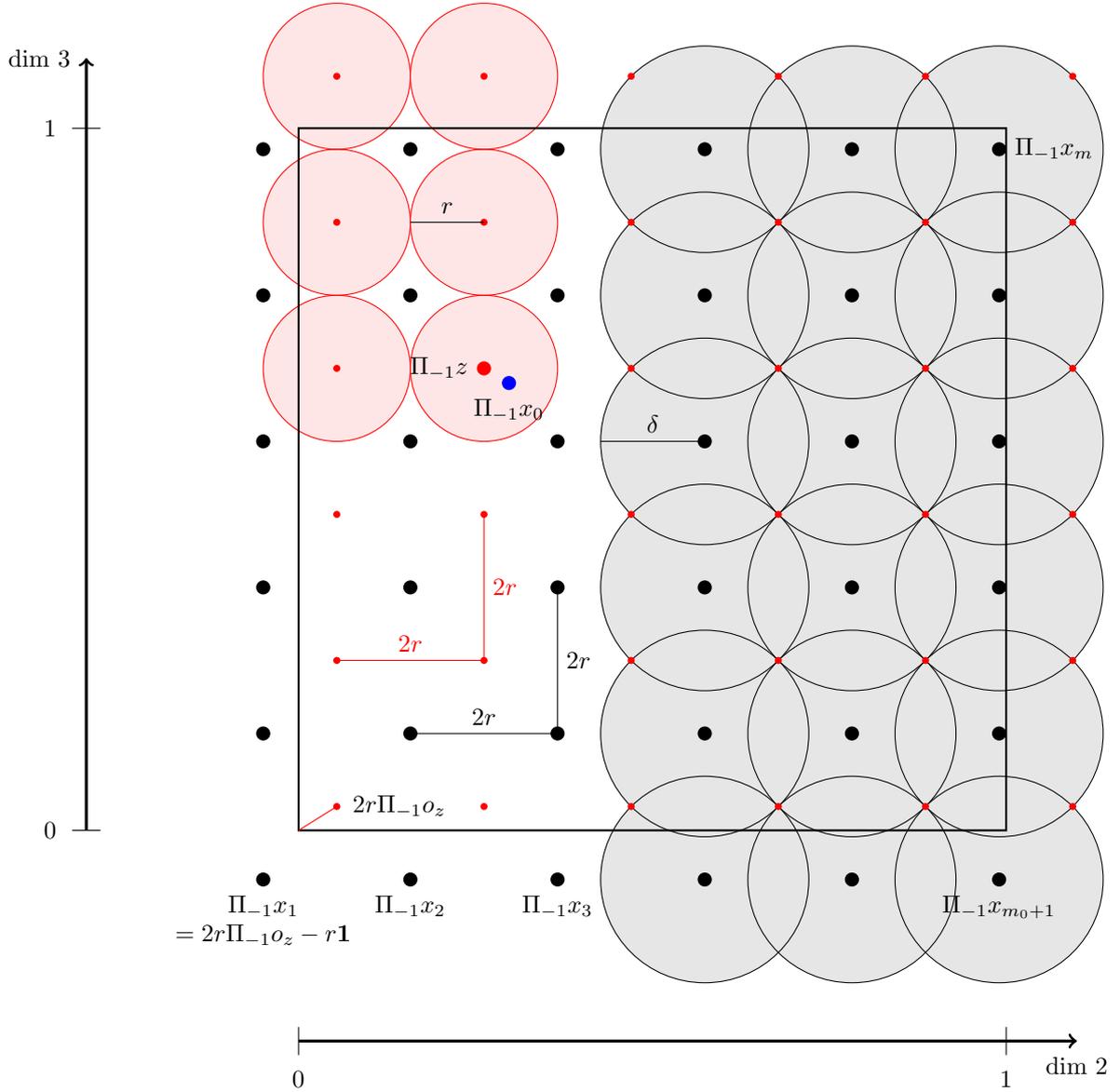
\begin{figure}
    \begin{center}
        \def\myxoffset{-0.05}
        \def\myn{5}
        \def\mys{0.104}
        \def\myrdot{0.01}
        \def\myrdotsmall{0.005}
        \def\myyoffset{-0.07}
        \def\myxaxisoffset{-0.3}
        \def\myyaxisoffset{-0.3}
        \def\myticklen{0.02}
        \begin{tikzpicture}[scale=10]
            \pgfmathtruncatemacro{\startValue}{\myn-2}
            \pgfmathtruncatemacro{\endValue}{\myn}
            \foreach \x in {\startValue,...,\endValue}{
                \foreach \y in {0,...,\myn}{
                    \fill[black!10] (2*\x*\mys+\myxoffset, 2*\y*\mys+\myyoffset) circle ({sqrt(2)*\mys});
                }
            }
            \foreach \x in {\startValue,...,\endValue}{
                \foreach \y in {0,...,\myn}{
                    \draw[black] (2*\x*\mys+\myxoffset, 2*\y*\mys+\myyoffset) circle ({sqrt(2)*\mys});
                }
            }
            \foreach \x in {0, ..., 1}{
                \foreach \y in {3,...,5}{
                    \fill[red!10] ({2*\x*\mys+\myxoffset+\mys}, {2*\y*\mys+\myyoffset+\mys}) circle ({\mys});
                }
            }
            \foreach \x in {0, ..., 1}{
                \foreach \y in {3,...,5}{
                    \draw[red] ({2*\x*\mys+\myxoffset+\mys}, {2*\y*\mys+\myyoffset+\mys}) circle ({\mys});
                }
            }
            \draw[thick] (0,0) rectangle (1, 1);
            \foreach \x in {0,...,\myn}{
                \foreach \y in {0,...,\myn}{
                    \fill (2*\x*\mys+\myxoffset, 2*\y*\mys+\myyoffset) circle (\myrdot);
                }
            }

            \foreach \x in {0,...,\myn}{
                \foreach \y in {0,...,\myn}{
                    \fill[red] ({2*\x*\mys+\myxoffset+\mys}, {2*\y*\mys+\myyoffset+\mys}) circle (\myrdotsmall);
                }
            }
            \coordinate (z) at ({2*2*\mys+\myxoffset-\mys}, {4*2*\mys+\myyoffset-\mys});
            \coordinate (oz) at ({1*2*\mys+\myxoffset-\mys}, {1*2*\mys+\myyoffset-\mys});
            \coordinate (x0) at ($(z) + ({0.17*2*\mys}, {-0.1*2*\mys})$);
            \coordinate (x1) at ({\myxoffset}, {\myyoffset});
            \coordinate (x2) at ({1*2*\mys+\myxoffset}, {\myyoffset});
            \coordinate (x3) at ({2*2*\mys+\myxoffset}, {\myyoffset});
            \coordinate (xm0) at ({\myn*2*\mys+\myxoffset}, {\myyoffset});
            \coordinate (xm) at ({\myn*2*\mys+\myxoffset}, {\myn*2*\mys+\myyoffset});
            \fill[red] (z) circle (\myrdot);
            \node[left=0.1cm] at (z) {$\Pi_{-1}z$};
            \fill[blue] (x0) circle (\myrdot);
            \node[below=0.1cm] at (x0) {$\Pi_{-1}x_0$};
            \node[below=0.1cm,align=center] at (x1) {$\Pi_{-1} x_1$\\$= 2r\Pi_{-1}o_z - r \mo 1$};
            \node[below=0.1cm] at (x2) {$\Pi_{-1}x_2$};
            \node[below=0.1cm] at (x3) {$\Pi_{-1}x_3$};
            \node[below=0.1cm] at (xm0) {$\Pi_{-1}x_{m_0+1}$};
            \node[right=0.1cm] at (xm) {$\Pi_{-1}x_{m}$};
            \draw[red] (0, 0) -- (oz);
            \node[right=0.1cm] at (oz) {$2r\Pi_{-1}o_z$};
            \draw[black] ({3*2*\mys\myxoffset}, {3*2*\mys+\myyoffset}) -- ({3*2*\mys+\myxoffset-sqrt(2)*\mys}, {3*2*\mys+\myyoffset}) node [midway, above, sloped] {$\delta$};
            \draw[black] ({1*2*\mys+\myxoffset+\mys}, {4*2*\mys+\myyoffset+\mys}) -- ({1*2*\mys+\myxoffset}, {4*2*\mys+\myyoffset+\mys}) node [midway, above, sloped] {$r$};
            \draw[black] ({1*2*\mys+\myxoffset}, {1*2*\mys+\myyoffset}) -- ({2*2*\mys+\myxoffset}, {1*2*\mys+\myyoffset}) node [midway, above] {$2r$};
            \draw[black] ({2*2*\mys+\myxoffset}, {2*2*\mys+\myyoffset}) -- ({2*2*\mys+\myxoffset}, {1*2*\mys+\myyoffset}) node [midway, right] {$2r$};
            \draw[red] ({1*2*\mys+\myxoffset-\mys}, {1*2*\mys+\myyoffset+\mys}) -- ({2*2*\mys+\myxoffset-\mys}, {1*2*\mys+\myyoffset+\mys}) node [midway, above] {$2r$};
            \draw[red] ({2*2*\mys+\myxoffset-\mys}, {2*2*\mys+\myyoffset+\mys}) -- ({2*2*\mys+\myxoffset-\mys}, {1*2*\mys+\myyoffset+\mys}) node [midway, right] {$2r$};

            \draw[->,very thick] (0,\myxaxisoffset) -- (1.1, \myxaxisoffset);
            \foreach \x in {0, 1} {
                \draw (\x,\myxaxisoffset+\myticklen) -- (\x, \myxaxisoffset-\myticklen);
                \node[below=0.1cm] at (\x,\myxaxisoffset-\myticklen) {$\x$};
            }
            \node[below=0.1cm] at (1.1, \myxaxisoffset){dim 2};

            \draw[->,very thick] (\myyaxisoffset, 0) -- (\myyaxisoffset, 1.1);
            \foreach \y in {0, 1} {
                \draw (\myyaxisoffset+\myticklen, \y) -- (\myyaxisoffset-\myticklen, \y);
                \node[left=0.1cm] at (\myyaxisoffset-\myticklen, \y) {$\y$};
            }
            \node[left=0.1cm] at (\myyaxisoffset, 1.1){dim 3};

        \end{tikzpicture}
    \end{center}
    \caption{Construction of the grid for proof of \cref{thm:snake:deterministic} for the lower error bound at $x_0$, illustrated with $d=3$. As the trajectories are linear in the first dimension and constant in the other dimension, we show the projection to the second and third dimension. The black dots are the projected initial conditions, which coincide with the trajectories. Red dots mark the centers of the bumps of $f_1$. The tubes covering the domain of interest are shaded in gray.}
    \label{fig:grid}
\end{figure}
\subsection{Combined Lower Bounds}
We have now proven two different lower bounds, which van be combined to the following lower bounds of the estimation risk.
\begin{corollary}\label{cor:snake:combined}
	Use the model of \ref{sec:modelDef} with parameter class $\ParamSnake$ from \cref{def:snake:smoothnessclass}.
	Assume $d\geq 2$, \assuRef{Noise}, and \assuRef{CoverTime}.
	Then there is a constant $C\in\Rpp$ large enough, depending only on $\beta, d, L_0, \dots, L_\ell, L_{\beta}, \const{noise}, \const{cvrtm}$, with the following property:
	Assume
	\begin{equation}\label{eq:snake:condition}
        C \max\brOf{\br{\frac {\Tsum}n }^{2\beta+d+1},\,\Tsum^{-1}} \leq \delta^{d-1} \leq C^{-1} \min\brOf{1,\, \frac n{\Tsum}}
        \eqfs
    \end{equation}
	Then
    \begin{align}
    	\forall x_0\in\hypercube\colon\ C \inf_{\festi} \sup_{\ptrue\in\ParamSnake} \E_{\ptrue}\abOf{ \euclOf{\festi - f_{\ptrue}}^2(x_0)}
    	&\geq
    	\delta^{2\beta}+ \br{\delta^{-(d-1)} n \Tsum^{-1}}^{-\frac{2\beta}{2(\beta+1)+d}}
    	\\
		C \inf_{\festi} \sup_{\ptrue\in\ParamSnake} \E_{\ptrue}\abOf{ \sup_{x\in\hypercube}\euclOf{\festi - f_{\ptrue}}^2(x) }
		&\geq
		\delta^{2\beta}+ \br{\frac{\delta^{-(d-1)} n\Tsum^{-1}}{\log(\delta^{-(d-1)} n\Tsum^{-1})}}^{-\frac{2\beta}{2(\beta+1)+d}}
		\\
		C \inf_{\festi} \sup_{\ptrue\in\ParamSnake}\E_{\ptrue}\abOf{ \LpNormOf{2}{\hypercube}{\euclOf{\festi - f_{\ptrue}}}^2}
		&\geq
		\delta^{2\beta} +  \br{\delta^{-(d-1)} n\Tsum^{-1}}^{-\frac{2\beta}{2(\beta+1)+d}}
		\eqcm
    \end{align}
    where the infima range over all estimators $\festi$ of $f_{\ptrue}$ based on the observations $Y_{\setIdx}$.
\end{corollary}
\begin{remark}\mbox{}
    \begin{enumerate}[label=(\roman*)]
    	\item
    	\textbf{Wasting Time.} One only needs $\Tsum = c_d  \br{L_0 \delta^{d-1}}^{-1}$ to cover the domain of interest $[0, 1]^d$, see \cref{rem:snake:lower:paramset}. If $\Tsum$ is much larger, this time and the corresponding observations can be wasted for non-informative observations to achieve a large lower bound.
    	\item \textbf{Optimal Covering.} If $\Tsum^{-1} = c_d L_0 \delta^{d-1}$ then \eqref{eq:snake:condition} becomes
    	\begin{equation}
    		C n^{-\frac{2(\beta+1)+d-1}{2(\beta+1)+d}} \leq \delta^{d-1} \leq C^{-1}
    		\eqfs
    	\end{equation}
    	Setting $\delta = c n^{-\frac{1}{2(\beta+1)+d}}$ fulfills this condition.
    	\item
    	\textbf{Balancing.} The deterministic error term dominates,
    	\begin{equation}
    		\delta^{2\beta} \geq c \br{\delta^{-(d-1)} n\Tsum^{-1}}^{-\frac{2\beta}{2(\beta+1)+d}}
    		\eqcm
    	\end{equation}
    	if and only if
    	\begin{equation}
    		\delta \geq c \br{\frac{n}{\Tsum}}^{-\frac{1}{2(\beta+1)+1}}
    		\eqfs
    	\end{equation}
    	If we further assume that no time is wasted, $\Tsum^{-1} = c_d L_0 \delta^{d-1}$, we arrive at
    	\begin{equation}
    		\delta \geq c n^{-\frac{1}{2(\beta+1)+d}}
    		\eqfs
    	\end{equation}
        \item
        \textbf{Minimax Optimality.}
        In \cite[Corollary 4.17]{upperbounds}, for $m=1$, an upper bound of the form
        \begin{equation}
            \sup_{x\in\hypercube}\euclOf{\festi - f_{\ptrue}}^2(x) \in \mathbf{O}_\Pr\brOf{
            \delta^{2\beta}
             +
             \br{\frac{n}{\Tsum \log n}}^{-\frac{2\beta}{2(\beta+1)+1}}}
        \end{equation}
        is shown for a certain estimator. That estimator is known to be flawed as the smoothness of $f$ is only taken into account along an observed trajectory but not orthogonal to it. In comparison with our lower error bound, it still achieves the minimax optimal rate of convergence (ignoring log factors) if
        \begin{equation}
            \delta \geq c \br{\frac{n}{\Tsum}}^{-\frac{1}{2(\beta+1)+1}}
            \eqfs
        \end{equation}
        In other words, the estimator in \cite[Corollary 4.17]{upperbounds} is minimax optimal, if the deterministic error term $\delta^{2\beta}$ dominates the error.
    \end{enumerate}
\end{remark}
\begin{proof}[Proof of \cref{cor:snake:combined}]
	The proof consists of first applying the probabilistic lower bound \cref{thm:snake:probabilistic} and the deterministic lower bound \cref{thm:snake:deterministic} together with the deterministic master theorem \cref{thm:lower:deterministic:master}. Then \cref{thm:lower:reduction:expectation} transfers results in probability to results in expectation.
\end{proof}
Similar as seen in the \textit{stubble model} for the time step $\stepsize$, we can choose optimal parameters $\delta$ and $\Tsum$, from which lower bounds follow only depending on the sample size $n$.
\begin{corollary}\label{cor:snake:combined:nice}
	Use the setting of \cref{cor:snake:combined}.
	Then
    \begin{align}
      \forall x_0\in\hypercube\colon\  \inf_{\festi} \sup_{\ptrue\in\ParamSnake} \E_{\ptrue}\abOf{ \euclOf{\festi - f_{\ptrue}}^2(x_0)}
      &\geq
      C n^{-\frac{2\beta}{2(\beta+1)+d}}
      \eqcm\\
      \inf_{\festi} \sup_{\ptrue\in\ParamSnake} \E_{\ptrue}\abOf{ \sup_{x\in\hypercube}\euclOf{\festi - f_{\ptrue}}^2(x)}
      &\geq
      C\left(\frac{n}{\log(n)}\right)^{-\frac{2\beta}{2(\beta+1)+d}}
      \eqcm\\
      \inf_{\festi} \sup_{\ptrue\in\ParamSnake}\E_{\ptrue}\abOf{ \LpNormOf{2}{\hypercube}{\euclOf{\festi - f_{\ptrue}}}^2}
      &\geq
      C n^{-\frac{2\beta}{2(\beta+1)+d}}
      \eqcm
    \end{align}
    where the infima range over all estimators $\festi$ of $f_{\ptrue}$ based on the observations $Y_{\setIdx}$.
\end{corollary}
\begin{remark}\mbox{}
    \begin{enumerate}[label=(\roman*)]
        \item
        \textbf{Generality.}
        These bounds are the minima of the lower bounds of \cref{cor:snake:combined} and hold for all $\Tsum$ and $\delta$ that fulfill the conditions in \cref{cor:snake:combined}.
        \item
        \textbf{Exponent.} Confer \cref{rem:stubble:nice:onlyn}.
        \item
        \textbf{Comparison to Stubble.}
        It is remarkable, that the two complementary models, snake and stubble, achieve the same universal lower error bound, see \cref{cor:stubble:nice:onlyn}.
    \end{enumerate}
\end{remark}
\begin{proof}[Proof of \cref{cor:snake:combined:nice}]
    Set
    \begin{equation}
        \delta = n^{-\frac{1}{2(\beta+1)+d}}\text{ and } \Tsum = C \delta^{-(d-1)} = C n^{\frac{d-1}{2(\beta+1)+d}}
    \end{equation}
    for the bounds at a point and in $L^2$ and
    \begin{equation}
        \delta = \left(\frac{n}{\log(n)}\right)^{-\frac{1}{2(\beta+1)+d}}\text{ and } \Tsum = C \delta^{-(d-1)}  =  C \left(\frac{n}{\log(n)}\right)^{\frac{d-1}{2(\beta+1)+d}}
    \end{equation}
    for the bound in the sup-norm.
    In the former case, the requirement
    \begin{equation}
        C \br{\frac {\Tsum}n }^{2\beta+d+1} \leq \delta^{d-1}
    \end{equation}
    becomes
    \begin{equation}
        C \br{n^{-\br{2(\beta+1)+1}\br{2\beta+d+1}}}  \leq  n^{-(d-1)}
    \end{equation}
    which is fulfilled even for $\beta=0$ if $n$ is large enough. The other conditions are also easily checked. Thus, we can apply \cref{cor:snake:combined}. As these choices minimizes the bounds of \cref{cor:snake:combined} with respect to $\delta$ and $\Tsum$, the lower bounds in \cref{cor:snake:combined:nice} are true for general $\delta$ and $\Tsum$.
\end{proof}
\begin{remark}
	The snake parameter class of \cref{def:snake:smoothnessclass} allows for an arbitrary number of trajectories $m\in\N$. Furthermore, the constructions in the proofs of \cref{thm:snake:deterministic} and \cref{thm:snake:probabilistic} use $m = c_{d, \beta} \delta^{-(d-1)}$ initial conditions. We show in \cref{sec:app:run} that we can achieve the same lower bound for $d=2$ and $\beta = 1$ with $m=1$ and conjecture that this can be generalized to arbitrary $d\in\N_{\geq 2}$ and $\beta\in\Rpo$.
\end{remark}

%% file: sec_app_kernel.tex
\section{Analytical Results}\label{sec:analystical}
\begin{notation}\mbox{}
	Subsequently (in all sections of the appendix), all lower case $c$, with or without index, are elements of $\Rpp$ and universal insofar as they only depend on the variables written as index, e.g., $c_{d,\beta}$ depends only on $\beta$ and $d$. In particular, a constant $c$ with no index refers to a fixed positive number. Every occurrence of such a variable may refer to a different value.
\end{notation}
In this section, we present basic analytical result that will be used later in the main proofs.
\begin{notation}\mbox{}\label{not:analytical}
    \begin{enumerate}[label = (\roman*)]
        \item Let $X$ be a set, $d\in\N$, and $f \colon X \to \R^d$. Denote the Euclidean norm as $\euclof{v}$ for $v\in\R^d$. Define $\supNormOf{f}:=\sup_{x\in X} \euclOf{f}$.
        \item \label{not:analytical:projection}
	       For $d\in\N$ fixed, denote by $\mo{e}_j\in\R^d$, $j\in\nnset{d}$, the $j$-th unit vector, meaning that all entries are equal to zero, but not the $j$-th, which equals 1.
	        \begin{itemize}
	        	\item  Denote by $\Pi_j:\R^d\to\R$, $j\in\nnset{d}$, the projection operator on the $j$-th component, namely
	        	\begin{equation}
	        		\Pi_j(x) := x\tr \mo{e}_j,
	        	\end{equation}
	        	for all $x\in\R^d$.
	        	\item Denote by $\Pi_{-j}:\R^d\to\R^{d-1}$, $j\in\nnset{d}$, the  projection operator on all but the $j$-th component, namely
        		\begin{equation}
        			\Pi_{-j}(x) := \br{\Pi_1(x),\, \dots,\, \Pi_{j-1}(x),\, \Pi_{j+1}(x),\, \dots,\,\Pi_d(x)}\tr
        		\end{equation}
        		for all $x\in\R^d$.
	        \end{itemize}
    \end{enumerate}
\end{notation}
\subsection{Derivatives}
For any finite-dimensional $\R$-vector space $V$, we denote the Euclidean norm as $\euclof{x}$ for $x\in V$.
For $k\in\N$ and finite-dimensional $\R$-vector spaces $V$ and $W$, let $\mc L_k(V, W)$ be the set of $k$-multilinear functions from $V^k$ to $W$.
An element of $\mc L_k(V, W)$ is called symmetric if it is invariant under permutations of its $k$ arguments.
A function $f\colon V \to W$ is differentiable at $x\in V$ if there is $A \in \mc L_1(V, W)$ such that
\begin{equation}\label{eq:derivative:first}
	\lim_{\euclOf{v} \to 0} \frac{\euclOf{f(x+v)-f(x) - A(v)}}{\euclOf{v}} = 0
	\eqfs
\end{equation}
In this case we write $D f(x) = A$.
The function $f$ is differentiable if it is differentiable at all $x\in V$. Denote the set of differentiable functions from $V$ to $W$ as $\mc D(V, W)$.
From now on, assume $V = \R^d$ for a $d\in\N$. Define $\mb D_d := \setByEle{v\in\R^d}{\euclOf{v} = 1}$.
The directional derivative of $f\in\mc D(V, W)$ at $x\in V$ in the direction $v\in\mb D_d$ is
\begin{equation}
	D_v f(x) := D f(x)(v)
	\eqfs
\end{equation}
We set $D^0$ to the identity, i.e., $D^0 f = f$.
Let $k\in\N$. Define the set of $k$-times differentiable functions $\mc D^k(V, W)$ and the $k$-th derivative operator $D^k$ recursively:
For $k = 1$, we set $\mc D^1(V, W) := \mc D(V, W)$ and $D^1 := D$.
Let $k \in \N_{\geq 2}$.
A function $f\in\mc D^k(V,W)$ is $k$-times differentiable at $x\in V$ if there is symmetric $A \in \mc L_{k}(V,W)$ such that
\begin{equation}
	\lim_{\euclOf v\to 0} \sup_{\mo v\in\mb D_d^{k-1}} \frac{\euclOf{D^{k-1}_{\mo v}f(x+v)-D^{k-1}_{\mo v}f(x) - A((\mo v, v))}}{\euclOf{v}} = 0
	\eqfs
\end{equation}
In this case we write $D^{k} f(x) = A$.
The function $f$ is $k$-times differentiable if it is $k$-times differentiable at all $x\in V$. Denote the set of $k$-times differentiable functions from $V$ to $W$ as $\mc D^k(V, W)$.
The $k$-th directional derivative of $f\in\mc D^k(V, W)$ at $x\in V$ in the directions $\mo v\in\mb D_d^k$ is
\begin{equation}
	D_{\mo v}^k f(x) := D^k f(x)(\mo v)
	\eqfs
\end{equation}
Let $p\in\N$. For $f\in\mc D^k(\R^d, \R^p)$ the derivative operator acts component-wise, i.e., for $f = (f_1, \dots, f_p)$ with $f_\ell\in\mc D^k(\R^d, \R)$, we have
\begin{equation}
	D^k_{\mo v}f(x) = \begin{pmatrix}	D^k_{\mo v}f_1(x) \\ \vdots \\ D^k_{\mo v}f_p(x) \end{pmatrix}
	\eqcm
\end{equation}
for $x\in\R^d$ and $\mo v\in \mb D_d^k$.
The operator norm for $A \in \mc L_k(V, W)$ is
\begin{equation}
	\opNormOf{A} := \sup_{\mo v \in\mb D_d^k} \euclOf{A(\mo v)}
	\eqfs
\end{equation}
Define the sup-norm of the $k$-th derivative as
\begin{equation}
	\supNormof{D^kf} := \sup_{x\in\R^d} \opNormof{D^kf(x)} = \sup_{x\in\R^d}  \sup_{\mo v\in \mb D_d^k} \euclOf{D^k_{\mo v}f(x)}
	\eqfs
\end{equation}
\begin{definition}[Hölder-smoothness classes]\label{def:Hoelder}
	Let $d\in\N$,
	For $L\in\Rpp$, define
	\begin{equation}
		\Sigma^{d\to1}(0; L) := \setByEle{f\colon\R^d\to\R}{\supNormOf{f} \leq L}\eqfs
	\end{equation}
	For $L\in\Rpp$, $\beta\in(0,1]$, define
	\begin{equation}
		\Sigma^{d\to1}(\beta; L) := \setByEle{f\colon\R^d\to\R}{\abs{f(x)-f(\tilde x)} \leq L \euclof{x-\tilde x}^{\beta}}\eqfs
	\end{equation}
	Let $\beta\in\Rppo$ and $\ell:=\llfloor\beta\rrfloor$. Define
	\begin{equation}
		\Sigma^{d\to1}(\beta; L) := \setByEle{f \in \mc D^{\ell}(\R^d, \R)}{\opNormof{D^\ell f(x)-D^\ell f(\tilde x)} \leq L \euclof{x-\tilde x}^{\beta - \ell}}\eqfs
	\end{equation}
	For $\beta\in\Rp$, use the short notation
	\begin{equation}
		\Sigma(\beta; L) := \Sigma^{1\to1}(\beta; L)\eqfs
	\end{equation}
	Let $L_0, \dots, L_\ell, \Lbeta\in\Rpp$. Define
	\begin{equation}
		\Sigma^{d\to 1}(\beta; \indset{L}{0}{\ell}, \Lbeta) := \Sigma^{d\to1}(\beta; \Lbeta) \cap \bigcap_{k=0}^\ell \Sigma^{d\to1}(k; L_k)\eqfs
	\end{equation}
	Let $\dm{in},\dm{out}\in\N$.  Define
	\begin{equation}
		\Sigma^{\dm{in}\to \dm{out}}(\beta; \indset{L}{0}{\ell}, \Lbeta) := \br{\Sigma^{\dm{in}\to 1}(\beta; \indset{L}{0}{\ell}, \Lbeta)}^{\dm{out}}\eqcm
	\end{equation}
	where $f = (f_1, \dots, f_{\dm{out}})\in \Sigma^{\dm{in}\to \dm{out}}(\beta; \indset{L}{0}{\ell}, \Lbeta)$ is treated as a function
	\begin{equation}
		f\colon \R^{\dm{in}} \to \R^{\dm{out}}, x\mapsto (f_1(x), \dots, f_{\dm{out}}(x))\tr
		\eqfs
	\end{equation}
\end{definition}
\begin{lemma}\label{app:diff:lem}
	Let $d\in\N$.
	Let $k\in\N$.
	Let $f\in \mc D^{k}(\R^d, \R)$.
	Let $L\in\Rpp$.
	Assume $\supNormof{D^{k}f} \leq L$.
	Then $f \in \Sigma^{d\to1}(k, L)$.
\end{lemma}
\begin{proof}[Proof of \cref{app:diff:lem}]
	Let $\mo v\in\mb D_d^{k-1}$. Then, by the mean value theorem,
	\begin{equation}
		\abs{D^{k-1}_{\mo v}f(x) - D^{k-1}_{\mo v}f(\tilde x)} \leq \sup_{\tilde v \in \mb D_d}\supNormOf{D^{k}_{(\mo v, \tilde v)}f} \euclOf{x - \tilde x}
	\end{equation}
    with $(\mo v, \tilde v)\in\mb D_d^{k}$. Thus,
	\begin{equation}
		\opNormOf{D^{k-1}f(x) - D^{k-1}f(\tilde x)} \leq \supNormOf{D^{k}f} \euclOf{x - \tilde x}
		\eqfs
	\end{equation}
\end{proof}
\subsection{Smoothness and Kernels}
The construction of the hypotheses in the proofs of \cref{sec:stubble} and \cref{sec:snake} are based on kernel functions and their smoothness properties. First, we show that a particular construction of function is contained in the function class $\Sigma^{d\to1}(\beta; \indset{L}{0}{\ell}, \Lbeta)$.
\begin{lemma}[Smoothness]\label{lem:general:smoothness}
	Let $d\in\N$.
	For $\beta\in\Rpp$ set $\ell := \llfloor\beta\rrfloor$.
	Assume $h \in \Sigma^{d\to 1}(\beta, 1)$.
	For $z\in\R^d$, $L,r\in\Rpp$, $b\in\R$,
	define
	\begin{equation}
		f_r \colon\R^d\to \R ,\, x \mapsto L r^\beta h\brOf{\frac{x-z}{r}} + b
		\eqfs
	\end{equation}
	For $L_{0}, \dots, L_\ell, \Lbeta\in\Rpp$,
	assume one of the following (equivalent) conditions:
	\begin{enumerate}[label=(\roman*)]
		\item \label{lmm:general:smoothness:fToSigma}
		We have
		\begin{align}
			L_0 &\geq L r^\beta \supNormOf{h} + \abs{b}\eqcm\\
			L_k &\geq L r^{\beta-k} \supNormOf{D^k h}\eqcm\qquad k\in\nnset\ell\eqcm\\
			\Lbeta &\geq L
			\eqfs
		\end{align}
		\item \label{lmm:general:smoothness:SigmaTof}
		We have
		\begin{equation}
			L \leq \min\brOf{
				\frac{L_0 -\abs{b}}{r^\beta \supNormOf{h}}
				\eqcm\
				\frac{L_1}{r^{\beta-1} \supNormOf{D^1 h}}
				\eqcm\
				\dots
				\eqcm\
				\frac{L_k}{r^{\beta-\ell} \supNormOf{D^\ell h}}
				\eqcm\
				\Lbeta
			}
			\eqfs
		\end{equation}
		\item \label{lmm:general:smoothness:rmax}
		We have $\abs{b} < L_0$, $L \leq \Lbeta$, and
		\begin{equation}
			r \leq \min\brOf{
				\br{\frac{L_0 -\abs{b}}{L \supNormOf{h}}}^{\frac{1}{\beta}}
				\eqcm\
				\br{\frac{L_1}{L \supNormOf{D^1 h}}}^{\frac{1}{\beta-1}}
				\eqcm\
				\dots
				\eqcm\
				\br{\frac{L_k}{L \supNormOf{D^\ell h}}}^{\frac{1}{\beta-\ell}}
			}
			\eqfs
		\end{equation}
	\end{enumerate}
	Then
	\begin{equation}
		f_r\in \Sigma^{d\to1}(\beta; \indset{L}{0}{\ell}, \Lbeta)
		\eqfs
	\end{equation}
\end{lemma}
\begin{proof}[Proof of \cref{lem:general:smoothness}]\mbox{}
	\begin{enumerate}[label=(\roman*)]
		\item
			For the bound on $L_0$, note
			\begin{equation}
				\supNormOf{f_r} = L r^\beta \supNormOf{h} + \abs{b}
				\eqfs
			\end{equation}
			Next, we consider the bounds on $\indset{L}{1}{\ell}$:
			Let $\mo v\in \mb D_d^k$.
			For $k\in\nnset{\ell}$, we have
			\begin{align}
				D^k_{\mo v} f_r(x) &=  L r^{\beta-k} D^k_{\mo v} h (x)\eqcm\\
				\supNormOf{D^k_{\mo v} f_r} &=  L r^{\beta-k} \supNormOf{D^k_{\mo v} h}\eqcm\\
				\supNormOf{D^k f_r} &=  L r^{\beta-k} \supNormOf{D^k h}\eqfs
			\end{align}
			Finally, we turn to the bound on $\Lbeta$: For $x, \tilde x\in\R^d$, as $h \in \Sigma^{d\to 1}(\beta, 1)$, we get
			\begin{align}
				\abs{D^\ell_{\mo v} f_r(x) - D^\ell_{\mo v} f_r(\tilde x)}
				&=
				L r^{\beta-\ell} \abs{D^\ell_{\mo v} h\brOf{\frac{x - x_0}{r}} - D^\ell_{\mo v} h\brOf{\frac{\tilde x - x_0}{r}}}
				\\&\leq
				L r^{\beta-\ell} \euclOf{{\frac{x - x_0}{r}} - {\frac{\tilde x - x_0}{r}}}^{\beta-\ell}
				\\&\leq
				L \euclOf{x - \tilde x}^{\beta-\ell}
				\eqfs
			\end{align}
		\item
			Follows directly from \ref{lmm:general:smoothness:fToSigma}.
		\item
			Follows directly from \ref{lmm:general:smoothness:fToSigma}.
	\end{enumerate}
\end{proof}
\begin{definition}[Kernel]\label{def:app:Kernel}
	A kernel is a continuous function $K\colon\R\to\R$.
	\begin{enumerate}[label=(\roman*)]
		\item The kernel is nonnegative if $K(x) \geq 0$ for all $x\in\R$.
		\item The kernel is symmetric if $K(x) = K(-x)$ for all $x\in\R$.
		\item A symmetric kernel is monotone if $K(x) \geq K(y)$ for all $x,y\in\Rp$ such that $x\leq y$.
		\item The support of a kernel $K$ is $\supp(K) := \setByEleInText{x\in\R}{K(x) \neq 0}$.
	\end{enumerate}
\end{definition}

\begin{example}[Standard kernel]\label{def:StandardKernel}
    Define the \textit{standard kernel} $\stdKernel\colon\R\to\R$ as
	\begin{equation}
		\stdKernel(w):=\exp\left(-\frac{1}{1-w^2}\right)\indOfOf{(-1,1)}{w}.
	\end{equation}
	for all $w\in\R$.
    Elementary computations show that the standard kernel $\stdKernel$ is a nonnegative, symmetric, monotone kernel with support $\supp(\stdKernel) = (-1, 1)$. Furthermore, $\stdKernel$ is smooth, i.e., $\stdKernel\in\smoothC(\R)$. See also \cite[equations (2.33) and (2.34)]{Tsybakov09Introduction}.
\end{example}
\begin{lemma}[Bump and Pulse]\label{lmm:bumpandpulse}
	Let $d\in\N$ and $\beta\in\Rpp$.
	\begin{enumerate}[label=(\roman*)]
		\item \label{lmm:bumpandpulse:bump}
			There is a nonnegative, symmetric, and monotone kernel $K_\beta$ with support $\supp(K_\beta) = (-1,1)$ with following property:
			Define $\hbump{d}{\beta}(x) := K_\beta(\euclof{x})$.
			Then $\hbump{d}{\beta}\in\Sigma^{d\to1}(\beta, 1)$.
		\item \label{lmm:bumpandpulse:pulse}
			There is a nonnegative, symmetric, and monotone kernel $\tilde K_\beta$ with support $\supp(\tilde K_\beta) = (-1,1)$ with following property:
			Define $\hpulse{d}{\beta}(x) := \tilde K_{\beta}\brOf{\euclOf{x}} \tilde K_{\beta}\pr\brOf{\Pi_1 x}$.
			Then $\hpulse{d}{\beta}\in\Sigma^{d\to1}(\beta, 1)$.
	\end{enumerate}
	Set $\ell := \llfloor\beta\rrfloor$.
	Let $z\in\R^d$, $r\in\Rpp$, $b\in\R$.
	Let $L_0, \dots, L_\ell, \Lbeta\in\Rpp$.
	Let $h \in \cb{\hbump{d}{\beta}, \hpulse{d}{\beta}}$.
	Assume $\abs{b} < L_0$.
	Set
	\begin{align}
		\rmax
		&:=
		\rmax\brOf{\beta, \indset{L}{0}{\ell}, \Lbeta, h, b}
		\\&:=
		\min\brOf{
			\br{\frac{L_0 -\abs{b}}{\Lbeta \supNormOf{h}}}^{\frac{1}{\beta}}
			\eqcm\
			\br{\frac{L_1}{\Lbeta \supNormOf{D^1 h}}}^{\frac{1}{\beta-1}}
			\eqcm\
			\dots
			\eqcm\
			\br{\frac{L_\ell}{\Lbeta \supNormOf{D^\ell h}}}^{\frac{1}{\beta-\ell}}
		}
		\eqfs
	\end{align}
	Define
	\begin{equation}
		f_r \colon\R^d\to \R ,\, x \mapsto \Lbeta r^\beta h\brOf{\frac{x-z}{r}} + b
		\eqfs
	\end{equation}
	Assume $r\in(0, \rmax]$.
	\begin{enumerate}[label=(\roman*),start=3]
		\item \label{lmm:bumpandpulse:smooth}
			Then $f_r\in \Sigma^{d\to1}(\beta; \indset{L}{0}{\ell}, \Lbeta)$.
	\end{enumerate}
\end{lemma}
\begin{proof}[Proof of \cref{lmm:bumpandpulse}]\mbox{}
	\begin{enumerate}[label=(\roman*)]
		\item
			Let $\stdKernel$ be the standard kernel from \cref{def:StandardKernel}. As $\stdKernel$ is smooth and symmetric, so is $\R\to\R,\,x\mapsto \stdKernel(\abs{x})$.
			As $\R\to\R,\,\lambda \mapsto \euclof{a + \lambda b}$ for $a, b \in\R^d$ is either smooth or equal to $\lambda \mapsto \abs{a_0 + \lambda b_0}$ for some $a_0, b_0\in\R$, we have that $\R^d\to\R,\,x\mapsto \stdKernel(\euclof{x})$ is smooth.
			Thus, for $\alpha_\beta\in\Rpp$ small enough, we can set $K_\beta(x) := \alpha_\beta \stdKernel(x)$, such that $\hbump{d}{\beta}(x) := K_\beta(\euclof{x})$ fulfills $\hbump{d}{\beta}\in\Sigma(\beta, 1)$.
		\item
			With the same construction as before, for $\tilde\alpha_\beta\in\Rpp$ small enough, we can set $\tilde K_\beta(x) := \tilde\alpha_\beta \stdKernel(\euclof{x})$, such that $\hpulse{d}{\beta}(x) := \tilde K_{\beta}\brOf{\euclOf{x}} \tilde K_{\beta}\pr\brOf{\Pi_1 x}$ fulfills $\hpulse{d}{\beta}\in\Sigma(\beta, 1)$.
		\item
			As in both cases $h\in\Sigma(\beta, 1)$,
			\cref{lem:general:smoothness} yields $f_r\in \Sigma^{d\to1}(\beta; \indset{L}{0}{\ell}, \Lbeta)$ for all $r\in(0, \rmax]$.
	\end{enumerate}
\end{proof}
\begin{lemma}\label{lem:snake:lower:ODEprop}
	Let $d\in\N_{\geq 2}$ and $\beta\in\Rpp$. Set $\ell := \llfloor\beta\rrfloor$.
	Let $z\in\R^d$, $r\in\Rpp$.
	Let $L_0, \dots, L_\ell, \Lbeta\in\Rpp$.
	Let $\hpulse{d}{\beta}$ and $\rmax = \rmax(\beta, \indset{L}{0}{\ell}, \Lbeta, \hpulse{d}{\beta}, 0)$ as in \cref{lmm:bumpandpulse}.
	Define
	\begin{equation}
		f \colon\R^d\to \R^d ,\, x \mapsto \Lbeta r^\beta \hpulse{d}{\beta}\brOf{\frac{x-z}{r}}\mo{e}_2 + b \mo{e}_1
		\eqfs
	\end{equation}
	\begin{enumerate}[label=(\roman*)]
		\item \label{lem:snake:lower:ODEprop:smooth}
			Assume $r\in(0, \rmax]$ and $\abs{b} < L_0$.
			Then $f\in\Sigma^{d\to d}(\beta, \indset{L}{0}{\beta}, \Lbeta)$.
		\item \label{lem:snake:lower:ODEprop:symm}
			Assume $b\neq 0$. Let $x\in\R^d$. Set $t_z := \Pi_1(z-x)/b$. Then, for all $t\in\R$ and $k\in\nset{2}{d}$,
			\begin{equation}
				\Pi_{k}U(f, x , t_z + t) = \Pi_{k}U(f, x , t_z - t)
				\eqfs
			\end{equation}
	\end{enumerate}
\end{lemma}
\begin{proof}[Proof of \cref{lem:snake:lower:ODEprop}]
	The first point follows directly from \cref{lmm:bumpandpulse}.

	Let $u(t) := U(f, x, t_z + t)$ and $v(t) := U(f, x , t_z - t)$ for all $t\in\R$.
	Denote $u_k := \Pi_k u$, $v_k := \Pi_k v$ for $k\in\nnset d$.
	Set $w := x - z$. Then, for all $k\in\nset{3}{d}$ and $t\in\R$,
	\begin{align}
		u_1(t) &= x_1 + b(t_z + t) = z_1 + bt\eqcm
		& u_k(t) &= x_{k}\eqcm \\
        v_1(t) &= x_1 + b(t_z - t) = z_1 - bt\eqcm
        & v_k(t) &= x_{k}\eqcm
	\end{align}
	and
    \begin{align}
	   \dot{u}_2(t)
       &=
       \Lbeta r^\beta
       \tilde K_{\beta}\brOf{\euclOf{\frac{u(t)-z}{r}}} \tilde K_{\beta}\pr\brOf{\frac{u_1(t)-z_1}{r}}
       \\&=
       \Lbeta r^\beta \tilde K_{\beta}\brOf{\frac{\sqrt{a + (u_2(t)-z_2)^2}}{r}}
       \tilde K_{\beta}\pr\brOf{\frac{bt}{r}}
       \eqcm\\
       \dot{v}_2(t)
       &=
       -\Lbeta r^\beta
       \tilde K_{\beta}\brOf{\euclOf{\frac{v(t)-z}{r}}} \tilde K_{\beta}\pr\brOf{\frac{v_1(t)-z_1}{r}}
       \\&=
       -\Lbeta r^\beta \tilde K_{\beta}\brOf{\frac{\sqrt{a + (v_2(t)-z_2)^2}}{r}}
       \tilde K_{\beta}\pr\brOf{-\frac{bt}{r}}
       \eqcm
    \end{align}
	where
	\begin{equation}
		a := b^2 t^2 + \sum_{k=3}^d (x_{k}-z_k)^2
		\eqfs
	\end{equation}
	As $\tilde K_{\beta}$ is an even function, $\tilde K_{\beta}\pr$ is an odd function, i.e.,
	\begin{equation}
		\tilde K_{\beta}\pr(-t) = -\tilde K_{\beta}\pr(t)
		\eqfs
	\end{equation}
	Thus, $v_2$ and $u_2$ solve the same ODE. As $u(0) = v(0)$, we obtain $u_2 = v_2$.
\end{proof}
The following lemma states the well-known \textit{Grönwall's inequality} in a simple form.
\begin{lemma}[Grönwall's inequality]\label{lem:snake:lower:Gronwall}
	Let $I=[a,b]\subset\R$ be an arbitrary interval and $v:I\to\R$ a continuous function. If there exist constants $A\in\R$ and $B\in\Rpp$, such that
	\begin{equation}
		v(t) \leq A + B\int_a^t v(s)\mathrm{d}s
	\end{equation}
	for all $t\in I$, then
	\begin{equation}
		v(t) \leq A\exp(B(t-a))
	\end{equation}
	for all $t\in I$.
\end{lemma}
We now apply \cref{lem:snake:lower:Gronwall} to analyze the behavior of solutions to the initial value problem discussed in the proof of \cref{lem:snake:lower:ODEprop}.
\begin{lemma}\label{lem:snake:lower:Gronwall:ODE}
	Let $d\in\N_{\geq 2}$ and $\beta\in\Rpp$.
	Let $z\in\R^d$, $r\in\Rpp$.
	Let $\Lbeta\in\Rpp$.
	Let $\hpulse{d}{\beta}$ as in \cref{lmm:bumpandpulse}.
 	Define
	\begin{equation}
		f \colon\R^d\to \R^d ,\, x \mapsto \Lbeta r^\beta \hpulse{d}{\beta}\brOf{\frac{x-z}{r}}\mo{e}_2 + b \mo{e}_1
		\eqfs
	\end{equation}
	For two initial conditions $x_1,x_2\in\R^d$ denote $u_j := U(f, x_j, \cdot)$, $j\in\nnset{2}$.
    Then, for $t\in\R$, we have
    \begin{enumerate}[label=(\roman*)]
    	\item
	    	\begin{equation}
	    		\euclOf{u_{1}(t)-u_{2}(t)} \leq
	    		c \euclOf{x_1-x_2} + c_{d, \beta} \Lbeta b^{-1} r^{\beta+1}\eqcm
	    	\end{equation}
    	\item
	    	\begin{equation}
	    		\euclOf{u_{1}(t)-u_{2}(t)} \leq
	    		c \euclOf{x_1-x_2} \exp\brOf{c_{d, \beta} \Lbeta b^{-1} r^{\beta}}\eqfs
	    	\end{equation}
    \end{enumerate}
\end{lemma}
\begin{proof}[Proof of \cref{lem:snake:lower:Gronwall:ODE}]
    First, observe that we only need to discuss in situations, where the initial values are given in such a way that the resulting trajectories are reaching a point, where $\hpulse{d}{\beta}$ is not zero. Otherwise, the the trajectories have the form
    \begin{equation}\label{eq:gronwall:ode:default:traj}
        u_{j}(t) = x_j + (bt,0,\dots,0)\tr
    \end{equation}
    for all $t\in\R$ and $j\in\nnset{2}$, such that
     \begin{equation}
        \euclOf{u_{1}(t)-u_{2}(t)} = \euclOf{x_1-x_2}
        \eqfs
    \end{equation}
    For any initial conditions, $x_1, x_2\in\R^d$, we have
    \begin{equation}
        u_{j}(t) = x_j + \left(bt,\Lbeta r^\beta\int_0^t \hpulse{d}{\beta}\brOf{\frac{u_j(s)-z}{r}}\mathrm{d}s,0,\dots,0\right)\tr,
    \end{equation}
    for all $t\in\R$. In particular,
    \begin{equation}
        \Pi_l(u_{1}(t)-u_{2}(t)) = \Pi_l(x_1-x_2)
    \end{equation}
    for $l\in\nnset{d}\setminus\{2\}$ and
    \begin{align}
         \abs{\Pi_2(u_{1}(t)-u_{2}(t))}
         &\leq
         \abs{\Pi_2(x_1-x_2)} +
         \Lbeta r^\beta \abs{\int_0^t\hpulse{d}{\beta}\brOf{\frac{u_1(s)-z}{r}} - \hpulse{d}{\beta}\brOf{\frac{u_2(s)-z}{r}}\dl s}
         \eqfs
    \end{align}
    On one hand, we can use the bound $\abs{\hpulse{d}{\beta}(v)-\hpulse{d}{\beta}(w)} \leq \supNormof{D\hpulse{d}{\beta}} \euclof{v-w}$ to obtain
    \begin{equation}
    	\abs{\int_0^t\hpulse{d}{\beta}\brOf{\frac{u_1(s)-z}{r}} - \hpulse{d}{\beta}\brOf{\frac{u_2(s)-z}{r}}\dl s}
    	\leq
    	c_{d, \beta} \frac1r \int_0^t \euclOf{u_1(s)-u_2(s)} \dl s
    	\eqfs
    \end{equation}
    Thus, we have
    \begin{align}
        \euclOf{u_{1}(t)-u_{2}(t)} & \leq
        c \euclOf{x_1-x_2} + c_{d, \beta} \Lbeta r^{\beta-1} \int_0^t \euclOf{u_1(s)-u_2(s)} \dl s
        \eqfs
    \end{align}
    By Grönwall's inequality, \cref{lem:snake:lower:Gronwall}, we obtain
    \begin{equation}
    	\euclOf{u_{1}(t)-u_{2}(t)} \leq
    	c\euclOf{x_1-x_2} \exp\brOf{c_{d, \beta} \Lbeta r^{\beta-1} t}
    	\eqfs
    \end{equation}
    On the other hand, we can use the bound $\abs{\hpulse{d}{\beta}(v)-\hpulse{d}{\beta}(w)} \leq 2 \supNormof{\hpulse{d}{\beta}}$ to obtain
    \begin{equation}
    	\abs{\int_0^t\hpulse{d}{\beta}\brOf{\frac{u_1(s)-z}{r}} - \hpulse{d}{\beta}\brOf{\frac{u_2(s)-z}{r}}\dl s}
    	\leq
    	c_{d, \beta} t
    	\eqfs
    \end{equation}
    Trajectories a disturbed only in balls of radius $r$. Outside these balls, they are identical to those trajectories described in \eqref{eq:gronwall:ode:default:traj} due to the time symmetry inside the balls, \cref{lem:snake:lower:ODEprop} \ref{lem:snake:lower:ODEprop:symm}.
    Hence, we only need to consider time intervals during which the trajectories are inside balls of radius $r$: For $t\in[0,2r/b]$, we obtain
    \begin{equation}
        \euclOf{u_{1}(t)-u_{2}(t)} \leq
        \min\brOf{
        	c \euclOf{x_1-x_2} + c_{d, \beta} \Lbeta b^{-1} r^{\beta+1}
        	\,,\
        	c \euclOf{x_1-x_2} \exp\brOf{c_{d, \beta} \Lbeta b^{-1} r^{\beta}}
       	}
        \eqfs
    \end{equation}
\end{proof}
In the following lemma we derive a covering property of trajectories driven by the model function $f$ of \cref{lem:snake:lower:Gronwall:ODE}. A sufficiently fine grid of fixed initial values results in the domain of interest $[0,1]^d$ being covered by the set of all tubes around these trajectories.
\begin{lemma}\label{lem:snake:lower:Covering}
	Let $d\in\N_{\geq 2}$, $\beta\in\Rpp$, $\Lbeta\in\Rpp$.
	Let $z\in\R^d$, $b\in\Rpp$.
	Let $T\in\Rpp$. Let $\delta\in\Rpp$.
	Let $\hpulse{d}{\beta}$ as in \cref{lmm:bumpandpulse}.
	Define two model functions $f_\cdot:\R^d\to\R^d$ by
    \begin{equation}
        f_0(x):= b \mo{e}_1
    \end{equation}
    and, for $r\in\Rpp$,
    \begin{equation}
		f_r(x):= \Lbeta r^\beta \hpulse{d}{\beta}\brOf{\frac{x-z}{r}}\mo{e}_2 + b \mo{e}_1
		\eqfs
	\end{equation}
	Then there is $C\in\Rpp$ large enough, depending only on $\beta$ and $d$, with the following property:
    Assume $bT \geq C$.
    Let
    \begin{equation}
    	m := \br{\left\lceil C \delta^{-1} \right\rceil + 1}^{d-1}
    	\eqfs
    \end{equation}
    Denote by $\mathcal{T}(\cdot,\delta)$ the tube as defined in \cref{def:tube}.
    Then there is a grid of initial conditions $x_1, \dots, x_m\in\R^d$, such that for all $r$ with
    \begin{equation}\label{eq:snake:lower:Covering:cond}
    	0 \leq r \leq C^{-1} \min\brOf{1,\,\br{\frac b\Lbeta}^{\frac1\beta}}
    \end{equation}
 	the tubes of radius $\delta$ around the $m$ trajectories started at $\indset{x}1m$ cover the domain of interest $\hypercube$, i.e.,
    \begin{equation}
    	\bigcup_{j=1}^{m} \mathcal{T}(u_j,\delta)\supset [0,1]^d
    	\eqcm
    \end{equation}
    where $u_j = U(f_r, x_j, \cdot) \colon [0,T]\to\R^d$.
\end{lemma}
\begin{proof}[Proof of \cref{lem:snake:lower:Covering}]
	Tubes of trajectories from $f_0$ starting from initials conditions that form a uniform grid in $\cb{0} \times [0,1]^{d-1}$ can cover the domain of interest, if the side length of the grid and the tube radius are in a sufficient relation. To make this work for $f_r$, $r>0$, we have to make sure that the disturbance introduced by the pulse in $f_r$ does not create holes in the covering. To this end, we extend the grid by at least the diameter $2r$ of the disturbance area: We assume $r\leq 1$ and use an extended grid of initial conditions in $\cb{-2} \times [-2,3]^{d-1}$. Then we apply \cref{lem:snake:lower:Gronwall:ODE} to show that two neighboring trajectories do not depart too far from each other.
	\begin{enumerate}[label=\arabic*.]
		\item
		\textbf{Definition of grid.} We define a uniform grid of points $x_1, \dots, x_m \in \cb{-2} \times [-2,3]^{d-1}$.
		Set the side length $s := C^{-1}\delta(d-1)^{-1/2}$.
		Set the number of points in the grid to $m := (5\lceil s \rceil+1)^{d-1}$.
		Choose $x_1, \dots, x_m\in\R^d$ such that $\Pi_{1}x_j = -2$ for all $j\in\nnset m$ and
		\begin{equation}
			\{\Pi_{-1}x_1, \dots, \Pi_{-1}x_m\}
			=
			\setByEle{sk}{k\in\nnzset{\lceil s \rceil}^{d-1}}
			\eqfs
		\end{equation}
		\item
		\textbf{Initial values cover one side.}
		We assume $C\geq \frac12$. Then, for each $x\in\cb{-2} \times [0,1]^{d-1}$, there is $j \in\nnset{m}$, such that
		\begin{equation}
			\euclOf{x_j - x} \leq \frac12 (d-1)^{\frac12} s = \frac12 C^{-1} \delta \leq \delta
			\eqfs
		\end{equation}
		\item
		\textbf{Tubes of $f_0$-trajectories cover the domain of interest.}
		The solution of the ODE $\dot u = f_0(u)$ starting from $x_j$ is
		\begin{align}
			U(f_0, x_j, t) &= x_j + bt\mo e_1\\
			\dot U(f_0, x_j, t) &= b\mo e_1
		\end{align}
		for $t\in\R$. Thus, we obtain
		\begin{equation}
			\bigcup_{j=1}^{m} \mathcal{T}(u_j,\delta)\supset [0,1]^d
		\end{equation}
		for $u_j = U(f_0, x_j, \cdot) \colon [0,T]\to\R^d$ given $C$ is chosen such that $r \leq C^{-1} \leq 1$ and $bT \geq C \geq 3$.
		\item
		\textbf{Tubes of $f_r$-trajectories, $r>0$, cover the domain of interest.}
	    We have to check whether trajectories that are perturbed by the pulse in $f_r$ also exhibit such a covering property.
        Let  $u_j = U(f_r, x_j, \cdot) \colon [0,T]\to\R^d$. Fix $x \in \hypercube$. Denote index of the trajectory with the closest state and the respective time point as
        \begin{equation}
            (j_0, t_0) \in \argmin_{j\in\nnset{m}, t\in[0,T]} \euclOf{x - u_j(t)} \eqfs
        \end{equation}
        Define the squared distance function of $u_{j_0}$ to $x$ as
        \begin{equation}
            \eta \colon [0, T] \to \Rp, t\mapsto \euclOf{x - u_{j_0}(t)}^2
            \eqfs
        \end{equation}
        Observe that
        \begin{equation}
            \dot \eta(t) = 2 \br{x - u_{j_0}(t)}\tr \dot u_{j_0}(t)
        \end{equation}
        and $\dot \eta(t_0) = 0$. Define $v := x - u_{j_0}(t_0)$. We have proven that $x\in\tube\brOf{u_{j_0}, \euclof{v}}$ and it remains to show that $\euclof{v} \leq \delta$.

        Let $j_1, j_2 \in \nnset m$. \cref{lem:snake:lower:Gronwall:ODE} yields
        \begin{equation}
            \euclOf{u_{j_1}(t)-u_{j_2}(t)}
            \leq
            c \euclOf{x_{j_1}-x_{j_2}} \exp\brOf{c_{d, \beta} \Lbeta b^{-1} r^{\beta}}
            \eqfs
        \end{equation}
        Assuming $C\geq c_{d,\beta}$, \eqref{eq:snake:lower:Covering:cond} implies
        \begin{equation}
            \exp\brOf{c_{d, \beta} \Lbeta b^{-1} r^{\beta}} \leq \exp(1)
            \eqfs
        \end{equation}
        Thus, if $x_{j_1}, x_{j_2}$ are neighboring, i.e., $\euclof{x_{j_1}-x_{j_2}} = s$, then
        \begin{equation}\label{eq:cubetrajbound}
            \euclOf{u_{j_1}(t) - u_{j_2}(t)} \leq c s = c C^{-1}\delta(d-1)^{-1/2}
            \eqfs
        \end{equation}
        Let $J \subset \nnset m$, be the indices of a grid cell of initial conditions, i.e., $(\Pi_{-1} x_j)_{j\in J}$ forms a hypercube of side length $s$ in $\R^{d-1}$. By \eqref{eq:cubetrajbound}, at time $t$, the convex hull $H \subset \R^{d-1}$ of $(\Pi_{-1} u_j(t))_{j\in J}$ has a diameter of at most $c_d C^{-1} \delta$. Hence, if $C$ is large enough, every point $x\in H$ is at most a distance $\delta$ away from the closest point of $(\Pi_{-1} u_j(t))_{j\in J}$. Furthermore, at any time $t\in[0, T]$ the convex hull of all points $(\Pi_{-1} u_j(t))_{j\in \nnset m}$ covers the hypercube $[0,1]^{d-1}$:
        First note that $(\Pi_{-1}x_j)_{j\in\nnset m}$ forms a grid in $[-2, 3]^{d-1}$. Next, a trajectory spends at most $\frac{2r}{b}$ time inside the support of $f_r$. Thus, by \eqref{eq:snake:lower:Covering:cond},
        \begin{equation}
            \euclOf{\Pi_{-1} (u_j(t) - x_j)} \leq \frac{2r}{b} \sup_{x\in \R^d} \abs{\Pi_2 f_r(x)} \leq c_{d,\beta} L_\beta b^{-1} r^{\beta+1} \leq c_{d,\beta}  C^{-1} \leq 2
        \end{equation}
        for $C$ large enough.

        Thus, we conclude that, for any $x\in[0,1]^d$, the closest point over all trajectories has a distance $\euclof{v} \leq \delta$.
	\end{enumerate}
\end{proof}

%% file: sec_app_master.tex
\section{Lower Bounds in General}\label{sec:app:lower}
In this section, we revisit some well-known techniques for proving lower error bounds as they can be found in \cite{Tsybakov09Introduction}. Furthermore, we derive a master theorem for lower bounds in regression settings. We illustrate the master theorem by applying it to the well-known nonparametric regression setting. In \cref{sec:app:stubble} and \cref{sec:app:snake}, the theorem is then applied to prove the statements about lower bounds in the stubble and snake models made in \cref{sec:stubble} and \cref{sec:snake}, respectively.
\subsection{General Reduction Scheme}\label{sec:lower:reduction}
Let $(\Theta,d)$ be a pseudometric space, i.e., a set $\Theta$ equipped with a map $d\colon\Theta\times\Theta\to\Rgeq{0}$ satisfying
\begin{enumerate}[label=(\roman*)]
    \item $d(\theta, \theta) = 0$,
    \item $d(\theta, \tilde\theta) = d(\tilde\theta, \theta)$, and \hfill (Symmetry)
    \item $d(\theta, \tilde\theta) \leq d(\theta, \theta\pr) +  d(\theta\pr, \tilde\theta)$ \hfill (Triangle Inequality)
\end{enumerate}
for all $\theta, \tilde\theta, \theta\pr \in\Theta$. In the following, let us consider a \textit{statistical model} $\{\Pr_\theta|\,\theta\in\Theta\}$, i.e., a family of probability measures index by the \textit{parameter space} $\Theta$.
The next two lemmas show how a lower bound on the error of estimating $\theta$ can be achieved by reducing $\Theta$ to only finitely many elements.
Recall that $\kullback(\cdot, \cdot)$ denotes the Kullback--Leibler divergence of two probability measures.
\begin{lemma}[Two Hypotheses]\label{thm:lower:reduction:two}
	Let $\theta_0,\theta_1\in\Theta$.
	Assume
	\begin{equation}
		\kullback(\Pr_{\theta_1}, \Pr_{\theta_0}) \leq \frac12
		\eqfs
	\end{equation}
	Then
	\begin{equation}
		\inf_{\hat\theta} \sup_{\ptrue\in\Theta} \Pr_{\ptrue}\brOf{d(\hat\theta, \ptrue) \geq \frac12 d(\theta_0, \theta_1)} \geq \frac14
		\eqcm
	\end{equation}
    where the infimum ranges over all estimators $\hat{\theta}$ of $\ptrue$ based on an observation $Y\sim \Pr_{\ptrue}$.
\end{lemma}
\begin{proof}[Proof of \cref{thm:lower:reduction:two}]
	\cite[Equation (2.9) and Theorem 2.2]{Tsybakov09Introduction} with $M=1$ and $\alpha=\frac12$.
\end{proof}
\begin{lemma}[Many Hypotheses]\label{thm:lower:reduction:many}
	Let $M\in\N$, $M \geq 2$.
	Let $\theta_0,\theta_1,\dots,\theta_M\in\Theta$.
	Assume
	\begin{equation}
		\frac{1}{M}\sum_{j=1}^M\kullback(\Pr_{\theta_j}, \Pr_{\theta_0}) \leq \frac13 \log(M)
		\eqfs
	\end{equation}
	Then
	\begin{equation}
		\inf_{\hat\theta} \sup_{\ptrue\in\Theta} \Pr_{\ptrue}\brOf{d(\hat\theta, \ptrue) \geq \frac12 \inf_{j,k\in\nset0M}d(\theta_j, \theta_k)} \geq \frac14
		\eqcm
	\end{equation}
    where the infimum ranges over all estimators $\hat{\theta}$ of $\ptrue$ based on an observation $Y\sim \Pr_{\ptrue}$.
\end{lemma}
\begin{proof}[Proof of \cref{thm:lower:reduction:many}]
	\cite[Equation (2.9) and Corollary 2.6]{Tsybakov09Introduction} with $\alpha=\frac13$ noting that
	\begin{equation}
		\inf_{M\geq 2}\brOf{\frac{\log(M+1)-\log(2)}{\log(M)} - \alpha} = \frac{\log(3/2)}{\log(2)} - \frac13 \geq \frac14
		\eqfs
	\end{equation}
\end{proof}
If the parameter of a statistical model is not identifiable, we have a trivial lower bound on the error. This is a deterministic version of \cref{thm:lower:reduction:two}.
\begin{lemma}\label{thm:lower:deterministic:general}
	Let $\theta_0, \theta_1\in\Theta$. Assume
	\begin{equation}
		\Pr_{\theta_0} = \Pr_{\theta_1}
		\eqfs
	\end{equation}
	Then
	\begin{equation}
		\inf_{\hat\theta} \sup_{\ptrue\in\Theta} \Pr_\ptrue\brOf{d(\hat\theta, \ptrue) \geq \frac 12 d(\theta_0, \theta_1)} = 1,
	\end{equation}
	where the infimum ranges over all estimators $\hat{\theta}$ of $\ptrue$ based on an observation $Y\sim \Pr_{\ptrue}$.
\end{lemma}
\begin{proof}[Proof of \cref{thm:lower:deterministic:general}]
	Set $a := d(\theta_0, \theta_1)$.
	As $d(\theta_0, \theta_1)  \leq d(\hat\theta, \theta_0) + d(\hat\theta, \theta_1)$, we have
	\begin{equation}
		\max\brOf{d(\hat\theta, \theta_0), d(\hat\theta, \theta_1)} \geq \frac12 d(\theta_0, \theta_1) = \frac a2\eqfs
	\end{equation}
	Thus,
	\begin{align}
		\sup_{\theta\in\Theta} \Pr_\theta\brOf{d(\hat\theta, \theta) \geq \frac a2}
		&\geq
		\max_{s\in\cb{0, 1}} \Pr_{\theta_s}\brOf{d(\hat\theta, \theta_s) \geq \frac a2}
		\\&=
		\Pr_0\brOf{\max\brOf{d(\hat\theta, \theta_0), d(\hat\theta, \theta_1)} \geq \frac a2}
		\\&=
		1\eqfs
	\end{align}
\end{proof}
Next, we note that it suffices to show lower bounds on the probability of error, since this also implies bounds on expectation.
\begin{theorem}[Reduction to bounds in probability]\label{thm:lower:reduction:expectation}
	Let $s,p\in\Rp$. Assume
	\begin{equation}
		\inf_{\hat\theta} \sup_{\ptrue\in\Theta} \Pr_{\ptrue}\brOf{d(\hat\theta, \ptrue) \geq s} \geq p
		\eqfs
	\end{equation}
	Let $\nu\colon\Rp\to\Rp$ be increasing.
	Then
	\begin{equation}
		\inf_{\hat\theta} \sup_{\ptrue\in\Theta} \E_{\ptrue}\abOf{\nu\brOf{d(\hat\theta, \ptrue)}} \geq p \nu(s)
		\eqfs
	\end{equation}
    Here, all infima range over all estimators $\hat{\theta}$ of $\ptrue$ based on an observation $Y\sim \Pr_{\ptrue}$.
\end{theorem}
\begin{proof}[Proof of \cref{thm:lower:reduction:expectation}]
	Let $\theta\in\Theta$ and $\hat{\theta}$ be an arbitrary estimator. Assume $\nu(s)\neq 0$. Then
	\begin{align}
		p
		&\leq
		\Pr_{\theta}\brOf{d(\hat\theta, \theta) \geq s}
		\leq
		\Pr_{\theta}\brOf{\nu\brOf{d(\hat\theta, \theta)} \geq \nu(s)}
		\leq
		\frac{\E_{\theta}\abOf{\nu\brOf{d(\hat\theta, \theta)}}}{\nu(s)}
		\eqcm
	\end{align}
	where we used Markov's inequality in the last step. As all inequalities do not depend on the choice of $\theta\in\Theta$ and $\hat{\theta}$, the claim follows immediately.
\end{proof}
\subsection{Master Theorem for Lower Bounds in Regression}\label{sec:app:master}
We now derive a lower bound on the estimation error in a general regression-type setting. It captures the common elements of lower bound proofs for the classical nonparametric regression setup as well as the snake and stubble ODE models.

Let $n, \dm q, \dm u, \dm f\in\N$.
Let $\xdomain\subset\R^{\dm q}$ be a bounded measurable set, subsequently called the \textit{domain of interest}.
Let $\Theta$ to be the set of parameters.
Each parameter $\theta\in\Theta$ is assumed to have the following attributes:
\begin{enumerate}[label=(\roman*)]
	\item $q_k(\theta) \in \R^{\dm q}$ for $k\in\nnset n$,\hfill (Location)
	\item  $u_k(\theta) \in\R^{\dm u}$ for $k\in\nnset n$, \hfill (Observed Object)
	\item  $f_\theta \colon \xdomain \to \R^{\dm f}$ continuous. \hfill(Function to Estimate)
\end{enumerate}
We will later require certain properties of these attributes. Consider the following regression model equation
\begin{equation}
	Y_k = u_k(\theta) + \noise_k \eqcm\quad k\in\nnset n\eqcm
\end{equation}
where $\noise_1, \dots, \noise_n$ are independent and identically distributed copies of the $\R^{\dm u}$-valued and centered random variable $\noise$ following the law $P^\noise$. We observe $Y_1, \dots, Y_n$ and want to estimate the function $f_\theta$ on the domain of interest $\xdomain$. For $\theta\in\Theta$, the distribution of $Y_1, \dots, Y_n$ is given by
\begin{equation}
	\Pr_\theta := \bigotimes_{k=1}^n \br{u_k(\theta) + P^{\noise}}
	\eqcm
\end{equation}
where we write $v + P$ for the distribution of $v + X$ with $X \sim P$, $v \in \R^{\dm u}$.
We make the following assumptions to achieve error lower bounds.
\begin{assumption}\label{ass:lower:general}\mbox{ }
	\begin{enumerate}[label=(\roman*)]
		\item \label{ass:lower:general:ref}
			There is a continuous \textit{reference function} $g\colon\R^{\dm q}\to\R^{\dm f}$ with $\supp(g) \subset \ball^{\dm q}(0,1)$ and a constant $v_0\in\R^{\dm f}$.
			There is $\theta_0\in\Theta$ with $f_{\theta_0}(x) = v_0$ for all $x\in\xdomain$.
			There are $\zeta\in\Rpp$ and $\rho_n^-, \rho_n^+ \in\Rpp$ with following property:
			For $z\in\R^{\dm q}$ and $r\in[\rho_n^-, \rho_n^+)$, there is $\theta_{z,r}\in\Theta$ such that
			\begin{equation}
				f_{\theta_{z,r}}(x) = r^\zeta g\brOf{\frac{x-z}r} + v_0
				\eqcm
			\end{equation}
			for all $x\in\xdomain$. We sometimes call $\theta_0$ the \textit{null hypothesis} and $\theta_{z,t}$ the \textit{alternatives}.
			Furthermore, if $q_k(\theta_0)\in\R^{\dm q}\setminus \ball^{\dm q}(z, r)$, then $q_k(\theta_{z,r}) = q_k(\theta_0)$ and $u_k(\theta_{z,r}) = u_k(\theta_0)$ for all $k\in\nnset n	$.
		\item \label{ass:lower:general:obs}
			For $r\in[\rho_n^-, \rho_n^+)$, we define the maximum observed distance between null hypothesis and alternatives, $\psi_{n}(r)$, and the number of locations in the support, $\chi_{n}(r)$, as
			\begin{align}
				\psi_{n}(r) &:= \sup_{z\in\R^{\dm q}} \sup_{k\in\nnset{n}} \euclOf{u_k(\theta_{z,r}) - u_k(\theta_0)}\eqcm\\
				\chi_{n}(r) &:= \sup_{z\in\R^{\dm q}} \#\setByEle{k\in\nnset n}{q_k(\theta_{z,r}) \in \ball^{\dm q}(z,r)}
				\eqfs
			\end{align}
			There are $a_n\in\Rpp$ and $\gamma\in\Rpp$ with the following property:
			For all $r\in[\rho_n^-, \rho_n^+)$, we have
			\begin{equation}
				\psi_{n}(r)^2 \chi_{n}(r) \leq a_n r^\gamma
				\eqfs
			\end{equation}
		\item \label{ass:lower:general:pack}
			For $r\in\Rpp$, let $\eta(r)$ be the maximal number of disjoint balls with radius $r$ that fit into $\xdomain$, i.e.
			\begin{equation}
				\eta(r) = \sup\!\setByEle{k\in\N}{\exists x_{\nnset{k}}\subseteq\xdomain \colon \ \forall j,\tilde j\in\nnset k, j\neq\tilde j\colon\ \euclOf{x_j - x_{\tilde j}} \geq 2r,\, \ball^{\dm q}(x_j, r)\subset \xdomain}
				\eqfs
			\end{equation}
			There is $\const{pack}\in\Rpp$, such that, for all $r\in[\rho_n^-, \rho_n^+)$,
			\begin{equation}
				\eta(r) \geq \const{pack} r^{-\dm q}
				\eqfs
			\end{equation}
		\item \label{ass:lower:general:add}
			For $r\in[\rho_n^-, \rho_n^+)$, if $z_1, \dots, z_J \in \R^{\dm q}, J\in\N$ fulfill $\inf_{j,i\in\nnset k} \euclof{z_j - z_i} \geq 2r$, then there is  $\theta\in\Theta$ such that,
			\begin{enumerate}[label=(\arabic*)]
				\item $f_{\theta}-f_{\theta_0} = \sum_{j=1}^J \br{f_{\theta_{z_j, r}}-f_{\theta_0}}$
				\item $u_k(\theta)- u_k(\theta_0) = \sum_{j=1}^J \br{u_k(\theta_{z_j, r})- u_k(\theta_0)}$, for all $k\in\nnset n$,
				\item $q_k(\theta) - q_k(\theta_0) = \sum_{j=1}^J \br{q_k(\theta_{z_j, r}) - q_k(\theta_0)}$, for all $k\in\nnset n$.
			\end{enumerate}
	\end{enumerate}
\end{assumption}
\begin{theorem}\label{thm:lower:master}
	Assume \cref{ass:lower:general} \ref{ass:lower:general:ref}, \ref{ass:lower:general:obs}.
	Assume the noise distribution $P^\noise$ fulfills \assuRef{Noise} with $d = \dm u$.
	\begin{enumerate}[label=(\roman*)]
		\item \label{thm:lower:master:pointwise}
		Set $C_1 := 2 \const{noise}$, $C_2 : = \frac12 \supNormOf{g}$.
		Assume $(\rho_n^-)^{-\gamma} \geq C_1a_n \geq (\rho_n^+)^{-\gamma}$.
		Then, for all $x_0\in\xdomain$,
		\begin{equation}
			\inf_{\festi} \sup_{\theta\in\Theta} \Pr_\theta\brOf{\euclOf{\festi - f_\theta}(x_0) \geq C_2\br{C_1 a_n}^{-\frac\zeta\gamma}} \geq \frac14
			\eqfs
		\end{equation}
		\item\label{thm:lower:master:sup}
		Assume \cref{ass:lower:general} \ref{ass:lower:general:pack}.
		Set $C_1 := \frac{12 \gamma \const{noise}}{\dm q}$, $C_2 : = \frac12 \supNormOf{g}$.
		Assume $(\rho_n^-)^{-\gamma} \geq C_1 a_n \log(a_n)^{-1} \geq (\rho_n^+)^{-\gamma}$, and $a_n \geq \max\brOf{2\eqcm \br{\frac{12 \gamma \const{noise}}{\dm q\const{pack}^{\frac{\gamma}{\dm q}}}}^{4}}$.
		Then
		\begin{equation}
			\inf_{\festi} \sup_{\theta\in\Theta} \Pr_\theta\brOf{\sup_{x\in\xdomain}\euclOf{\festi - f_\theta}(x) \geq C_2\br{C_1 a_n \log(a_n)^{-1}}^{-\frac\zeta\gamma}} \geq \frac14
			\eqfs
		\end{equation}
		\item\label{thm:lower:master:lp}
		Let $p\in\Rpo$.
		Assume \cref{ass:lower:general} \ref{ass:lower:general:pack}, \ref{ass:lower:general:add}.
		Set $C_1 := 36 \const{noise}$ and $C_2 : = 2^{-1-3/p} \const{pack}^{1/p} \LpNormOf{p}{\R^{\dm q}}{g}$.
		Assume $(\rho_n^-)^{-\gamma} \geq C_1 a_n \geq (\rho_n^+)^{-\gamma}$ and $C_1 a_n \geq \br{\frac18 \const{pack}}^{-\frac\gamma{\dm q}}$.
		Then
		\begin{equation}
			\inf_{\festi} \sup_{\theta\in\Theta} \Pr_\theta\brOf{\LpNormOf{p}{\xdomain}{\euclOf{\festi - f_\theta}} \geq C_2 \br{C_1 a_n}^{-\frac\zeta\gamma}}  \geq \frac14
			\eqfs
		\end{equation}
	\end{enumerate}
	Here, all infima range over all estimators $\festi$ of $f_{\theta}$ based on the observations $Y_{\nnset n}$.
\end{theorem}
\begin{proof}[Proof of \cref{thm:lower:master}]
	The general strategy for the proofs of all parts of this theorem is to find suitable hypotheses $\theta_0, \dots, \theta_M\in\Theta$ that have a relatively large distance but relatively small Kullback--Leibler divergence of their associated probability measures and then apply the generic lower bound of \cref{thm:lower:reduction:two} or \cref{thm:lower:reduction:many}.
	\begin{enumerate}[label=(\roman*)]
		\item
			\begin{enumerate}[label=(\arabic*)]
			\item \textbf{Construction of Hypotheses:}\\
				Let $z\in\R^{\dm q}$ and $r\in[\rho_n^-, \rho_n^+)$ to be chosen later. Consider the two hypotheses $\theta_0$ and $\theta_{z,r}\in\Theta$.
			\item \textbf{Bound of Kullback--Leibler divergence:}\\
				By \assuRef{Noise} and \cref{ass:lower:general} \ref{ass:lower:general:ref}, \ref{ass:lower:general:obs}, for $r\in[\rho_n^-, \rho_n^+)$, we have
				\begin{align}
					\kullback(\Pr_{\theta_{z, r}}, \Pr_{\theta_0})
					&\leq
					\const{noise} \sum_{k=1}^n \euclOf{u_k(\theta_{z, r}) - u_k(\theta_{0})}^2
					\\&\leq
					\const{noise} \sum_{k=1}^n \indOfOf{\ball^{\dm q}(z, r)}{q_k(\theta_{z,r})} \psi_{n}(r)^2
					\\&\leq
					\const{noise} \psi_{n}(r)^2 \chi_n(r)
					\\&\leq
					\const{noise} a_n r^\gamma
					\eqfs
				\end{align}
				Choosing
				\begin{equation}
					r_n := \br{C_1 a_n}^{-\frac1\gamma} = \br{2\const{noise} a_n}^{-\frac1\gamma}
				\end{equation}
				yields $\kullback(\Pr_{\theta_{z, r_n}}, \Pr_{\theta_0}) \leq 1/2$.
			\item \textbf{Distance between Hypotheses:}\\
				As $g$ is continuous, $\ball^{\dm q}(0, 1)$ is precompact, and $g = 0$ on $\overline{\ball^{\dm q}(0, 1)} \setminus \ball^{\dm q}(0, 1)$, we can find $z\in\R^{\dm q}$ such that
				\begin{equation}
					\euclOf{f_{\theta_{z,r_n}}(x_0) - f_{\theta_0}(x_0)} =
					\euclOf{r^\zeta g\brOf{\frac{x_0-z}{r}}} =
					r^\zeta \supNormOf{g}
					\eqfs
				\end{equation}
			\item \textbf{Application of reduction scheme:}\\
				Hence, \cref{thm:lower:reduction:two} yields
				\begin{equation}
					\inf_{\festi} \sup_{\theta\in\Theta} \Pr_{\theta}\brOf{\euclOf{\festi - f_\theta}(x_0) \geq \frac12 \supNormOf{g} r_n^\zeta} \geq \frac14
					\eqfs
				\end{equation}
			\end{enumerate}
		\item
			\begin{enumerate}[label=(\arabic*)]
				\item \textbf{Construction of Hypotheses:}\\
					Let $r\in[\rho_n^-, \rho_n^+)$ to be chosen later.
					Using \cref{ass:lower:general} \ref{ass:lower:general:pack}, let $z_1,\dots z_{\eta}\in\xdomain$ with $\eta := \eta(r) \geq 2$ such that $\normof{z_\kappa-z_{\tilde\kappa}} \geq 2r$ and $\ball^{\dm q}(z_\kappa, r)\subset \xdomain$ for all $\kappa,\tilde\kappa\in\nnset{{\eta}}$.
					We consider the hypotheses $\theta_0, \theta_{z_1, r}, \dots, \theta_{z_\eta, r}$.
				\item \textbf{Bound of Kullback--Leibler divergence:}\\
					By \assuRef{Noise} and \cref{ass:lower:general} \ref{ass:lower:general:ref}, \ref{ass:lower:general:obs}, for $r\in[\rho_n^-, \rho_n^+)$, we have
					\begin{align}
						\frac1{\eta}\sum_{\kappa=1}^{\eta}\kullback(\Pr_{\theta_{z_\kappa, r}}, \Pr_{\theta_0})
						&\leq
						\const{noise} \frac1{\eta}\sum_{\kappa=1}^{\eta} \sum_{k=1}^n \euclOf{u_k(\theta_{z_\kappa, r}) - u_k(\theta_{0})}^2
						\\&\leq
						\const{noise} \frac1{\eta}\sum_{\kappa=1}^{\eta} \sum_{k=1}^n \indOfOf{\ball^{\dm q}(z_\kappa, r)}{q_k(\theta_{z_\kappa,r})} \psi_{n}(r)^2
						\\&\leq
						\const{noise} \psi_{n}(r)^2 \chi_n(r)
						\\&\leq
						\const{noise} a_n r^\gamma
						\eqfs
					\end{align}
					We assume $a_n > 1$ and choose
					\begin{equation}
						r_n := \br{\frac{C_{\kullback} \log(a_n)}{a_n}}^{\frac1\gamma}
						\eqcm
					\end{equation}
					where we set
					\begin{equation}
						C_{\kullback} := \frac{\dm q}{12 \gamma \const{noise}}
						\eqfs
					\end{equation}
					Then, with \cref{ass:lower:general}  \ref{ass:lower:general:pack}
					\begin{align}
						\log(\eta(r_n))
						&\geq
						\log\brOf{\const{pack}r_n^{-\dm q}}
						\\&=
						\log(\const{pack}) - \dm q \log\brOf{\br{\frac{C_{\kullback} \log(a_n)}{a_n}}^{\frac1\gamma}}
						\\&\geq
						C_{\dm q,\gamma} + \frac{\dm q}{\gamma}\br{\log(a_n) - \log \log(a_n)}
						\\&\geq
						C_{\dm q,\gamma} + \frac{\dm q}{2\gamma}\log(a_n)
						\eqcm
					\end{align}
					with $\log(x) \geq \euler \log(\log(x))$ for all $x>1$, where
					\begin{equation}
						C_{\dm q,\gamma} :=
						\log(\const{pack}) - \frac{\dm q}{\gamma} \log\brOf{C_{\kullback}}
						\eqfs
					\end{equation}
					We have
					\begin{equation}
						\log(\eta(r_n))
						\geq
						C_{\dm q,\gamma} + \frac{\dm q}{2\gamma}\log(a_n)
						\geq
						\frac{\dm q}{4\gamma}\log(a_n)
					\end{equation}
					if $C_{\dm q,\gamma} \geq -\frac{\dm q}{4\gamma}\log(a_n)$. We have
					\begin{align}
						C_{\dm q,\gamma} \geq -\frac{\dm q}{4\gamma}\log(a_n)
						& \Leftrightarrow
						a_n \geq \exp\brOf{-C_{\dm q,\gamma} \frac{4\gamma}{\dm q}}
						\\&\Leftrightarrow
						a_n \geq \br{\frac{12 \gamma \const{noise}}{\dm q\const{pack}^{\frac{\gamma}{\dm q}}}}^{4} =: a_0
						\eqfs
					\end{align}
					Then, assuming $r_n\in[\rho_n^-, \rho_n^+)$, $a_n > 1$, and $a_n \geq a_0$, we obtain
					\begin{align}
						\frac1{\eta(r_n)}\sum_{\kappa=1}^{\eta(r_n)} \kullback(\Pr_{\theta_{x_0, r_n}}, \Pr_{\theta_0})
						&\leq
						\const{noise} a_n r_n^\gamma
						\\&=
						\const{noise} C_{\kullback} \log(a_n)
						\\&\leq
						\frac13 \frac{\dm q}{4\gamma} \log(a_n)
						\\&\leq
						\frac13 \log(\eta(r_n))
						\eqfs
					\end{align}
				\item \textbf{Distance between Hypotheses:}\\
					As $\ball^{\dm q}(z_{\kappa}, r) \cap \ball^{\dm q}(z_{\tilde\kappa}, r)  = \emptyset$ and $\ball^{\dm q}(z_{\kappa}, r) \subset \xdomain$, we have
					\begin{equation}
						\sup_{x\in\xdomain} \euclOf{f_{\theta_{z_\kappa, r}} - f_{\theta_{z_{\tilde\kappa}, r}}}(x) = \sup_{x\in\ball^{\dm q}(0,1)} \euclOf{ r^\zeta g\brOf{\frac{x-z}{r}}} = r^\zeta \supNormOf{g}
					\end{equation}
					for all $\kappa,\tilde\kappa\in\nnset{\eta(r)}$, $\kappa\neq\tilde\kappa$, by \cref{ass:lower:general} \ref{ass:lower:general:ref}.
				\item \textbf{Application of reduction scheme:}\\
					Hence, \cref{thm:lower:reduction:many} yields
					\begin{equation}
						\inf_{\festi} \sup_{\theta\in\Theta} \PrOf{\sup_{x\in\xdomain}\euclOf{\festi - f_\theta}(x) \geq \frac12 \supNormOf{g} r_n^\zeta} \geq \frac14
						\eqcm
					\end{equation}
					for $a_n \geq \const{pack}^{-\frac{4\gamma}{\dm q}} \const{noise}^4$, $a_n>1$, and $r_n\in[\rho_n^-, \rho_n^+)$, where
					\begin{equation}
						r_n = \br{\frac{\dm q \log(a_n)}{12 \gamma \const{noise} a_n}}^{\frac1\gamma}
						\eqfs
					\end{equation}
			\end{enumerate}
		\item
		\begin{enumerate}[label=(\arabic*)]
			\item \textbf{Construction of Hypotheses:}\\
				Let $r\in[\rho_n^-, \rho_n^+)$ to be chosen later.
				Using \cref{ass:lower:general} \ref{ass:lower:general:pack}, let $z_1,\dots z_{\eta}\in\xdomain$ with $\eta := \eta(r) \geq \const{pack} r^{-\dm q}$ such that $\normof{z_\kappa-z_{\tilde\kappa}} \geq 2r$ and $\ball^{\dm q}(z_\kappa, r)\subset \xdomain$ for all $\kappa,\tilde\kappa\in\nnset{{\eta}}$.
				We assume $r \leq \br{\frac18 \const{pack}}^{\frac1{\dm q}}$ to get $\eta \geq 8$.
				Then, according to the Varshamov–Gilbert bound \cite[Lemma 2.9]{Tsybakov09Introduction}, there is a set $\mc B \subset \{0, 1\}^{\eta}$ with following properties:
				\begin{enumerate}[label=(\alph*)]
					\item $0\in\mc B$
					\item $\# \mc B = 2^{\eta/8}$
					\item the pairwise Hamming distance of $\# \mc B$ is at least $\eta/8$, i.e., $\sum_{\kappa=1}^\eta \abs{b_\kappa - \tilde b_\kappa} \geq\frac\eta8$ for all $b,\tilde b\in\mc B$.
				\end{enumerate}
				Use \cref{ass:lower:general} \ref{ass:lower:general:add} to obtain $\theta_b\in\Theta$ such that
				\begin{align}
					f_{\theta_b} - f_{\theta_0} &= \sum_{\kappa=1}^{\eta} b_\kappa \br{f_{\theta_{z_\kappa,r}} - f_{\theta_0}}\eqcm\\
					q_k(\theta_b) - q_k(\theta_0) &= \sum_{\kappa=1}^{\eta} b_\kappa \br{q_k(\theta_\kappa) - q_k(\theta_0)}\eqcm\\
					u_k(\theta_b) - u_k(\theta_0) &= \sum_{\kappa=1}^{\eta} b_\kappa \br{u_k(\theta_\kappa) - u_k(\theta_0)}
					\eqfs
				\end{align}
				Consider the hypotheses $\theta_b$, $b\in\mc B$.
			\item \textbf{Bound of Kullback--Leibler divergence:}\\
				By \assuRef{Noise} and \cref{ass:lower:general} \ref{ass:lower:general:ref} and \ref{ass:lower:general:obs}, for $r\in[\rho_n^-, \rho_n^+)$, we have
				\begin{align}
					\frac1{\#\mc B}\sum_{b\in\mc B}\kullback(\Pr_{\theta_{b}}, \Pr_{\theta_0})
					&\leq
					\const{noise} \frac1{\#\mc B}\sum_{b\in\mc B}
					\sum_{k=1}^n \euclOf{u_k(\theta_{b}) - u_k(\theta_{0})}^2
					\\&\leq
					\const{noise} \frac1{\#\mc B}\sum_{b\in\mc B}
					\sum_{k=1}^n
					\sum_{\kappa=1}^\eta
					b_\kappa \indOfOf{\ball^{\dm q}(z_\kappa, r)}{q_k(\theta_b)} \psi_{n}(r)^2
					\\&\leq
					\const{noise} \eta \psi_{n}(r)^2 \chi_n(r)
					\\&\leq
					\const{noise}\eta  a_n r^\gamma
					\eqfs
				\end{align}
				By definition of $\mc B$, we have
				\begin{equation}
					\eta
					\leq
					8 \log(2)^{-1} \log(\# \mc B)
					\leq
					12 \log(\# \mc B)
					\eqfs
				\end{equation}
				Choosing
				\begin{equation}
					r_n := \br{C_1 a_n}^{-\frac1\gamma} = \br{36 \const{noise} a_n}^{-\frac1\gamma}
				\end{equation}
				yields
				\begin{equation}
					\frac1{\#\mc B}\sum_{b\in\mc B}\kullback(\Pr_{\theta_{b}}, \Pr_{\theta_0})
					\leq
					\frac1{3} \log(\#\mc B)
				\end{equation}
				as soon as $r_n\in[\rho_n^-, \rho_n^+)$.
			\item \textbf{Distance between Hypotheses:}\\
				By  \cref{ass:lower:general} \ref{ass:lower:general:ref}, we have
				\begin{align}
					\int_{\xdomain} \euclOf{f_{\theta_{z,r}}(x) - f_{\theta_0}(x)}^p \dl x
					&=
					\int_{\ball^{\dm q}(z, r)} \euclOf{r^\zeta g\brOf{\frac{x-z}r}}^p \dl x
					\\&=
					r^{\zeta p + \dm q} \int_{\ball^{\dm q}(0, 1)} \euclOf{g\brOf{x}}^p \dl x
					\eqfs
				\end{align}
				Thus, using \cref{ass:lower:general} \ref{ass:lower:general:pack},
				\begin{align}
					\LpNormOf{p}{\xdomain}{\euclOf{f_{\theta_b} - f_{\theta_{\tilde b}}}}^p
					&=
					\int_{\xdomain} \euclOf{f_{\theta_b}(x) - f_{\theta_{\tilde b}}(x)}^p \dl x
					\\&=
					\sum_{\kappa=1}^{\eta} \abs{b_\kappa -\tilde b_\kappa} \int_{\xdomain} \euclOf{f_{\theta_{z_\kappa}}(x) - f_{\theta_{0}}(x)}^p \dl x
					\\&=
					r^{\zeta p + \dm q} \LpNormOf{p}{\R^{\dm q}}{g}^p
					 \sum_{\kappa=1}^{\eta} \abs{b_\kappa -\tilde b_\kappa}
					\\&\geq
					r^{\zeta p + \dm q} \LpNormOf{p}{\R^{\dm q}}{g}^p \eta/8
					\\&\geq
					2^{-3} \const{pack} r^{\zeta p} \LpNormOf{p}{\R^{\dm q}}{g}^p
					\eqfs
				\end{align}
			\item \textbf{Application of reduction scheme:}\\
				Hence, \cref{thm:lower:reduction:many} yields
				\begin{equation}
					\inf_{\festi} \sup_{\theta\in\Theta} \Pr_\theta\brOf{\LpNormOf{p}{\xdomain}{\euclOf{\festi - f_\theta}} \geq 2^{-1-3/p} \const{pack}^{1/p} \LpNormOf{p}{\R^{\dm q}}{g} r_n^{\zeta}}  \geq \frac14
				\end{equation}
				as soon as $r_n\in[\rho_n^-, \rho_n^+)$ and $r_n \leq \br{\frac18 \const{pack}}^{\frac1{\dm q}}$. The second condition is equivalent to
				\begin{equation}
					C_1 a_n \geq \br{\frac18 \const{pack}}^{-\frac\gamma{\dm q}}
					\eqfs
				\end{equation}
		\end{enumerate}
	\end{enumerate}
\end{proof}
In this article, we apply the master theorem in different setting. \cref{tbl:symbolsMaster} gives as overview of how the abstract objects in of the master theorem map to the objects of its application.
\begin{table}
	\begin{center}
		\renewcommand{\arraystretch}{1.4}
		\setlength{\tabcolsep}{0.3em}
		\rowcolors{2}{gray!20}{white}
		\begin{tabular}{p{4.1cm}|p{0.8cm}|p{2.85cm}|p{2.85cm}|p{3.5cm}}
			\rowcolor{gray!40}
			Description & Mast. & Nonparam. Regr. & Stubble &  Snake \\
            \hline
            number of observations      & $n$                   & $n$                                           & $n$ & $n$\\
            observation index           & $k$                   & $k$                                           & $(j,i)$ & $(j,i)$\\
            location dimension          & $\dm q$               & $d$                                           & $d$ & $d$\\
            observation dimension       & $\dm u$               & $1$                                           & $d$ & $d$\\
            function output dimension   & $\dm f$               & $1$                                           & $d$ & $d$\\
            domain of interest          & $\xdomain$            & $[0,R]^d$                                     & $\hypercube$ & $\hypercube$ \\
            set of parameters           & $\Theta$              & $\Sigma^{d\to1}(\beta, \indset{L}{0}{\ell}, \Lbeta)$   & $\Sigma^{d\to d}(\beta, \indset{L}{0}{\ell}, \Lbeta)$ & $\ParamSnake(\indset{L}{0}{\ell}, \Lbeta, \delta, \Tsum)$ \\
            parameter                   & $\theta$              & $f$                                           & $f$                        & $(f, \indset x1m, \indset T1m)$ \\
            location                    & $q_k(\theta)$         & $x_k$                                         & $\xany{j_k}$               & $U(f,x_{j_k},t_{j_k,i_k})$ \\
            observed object             & $u_k(\theta)$         & $f(x_k)$                                      & $U(f,x_{j_k},t_{j_k,i_k})$ & $U(f,x_{j_k},t_{j_k,i_k})$ \\
            function to estimate        & $f_\theta$            & $D^s_{\mo v} f$                               & $f$ & $f$ \\
            reference function          & $g$                   & $\Lbeta D^s_{\mo v} \hbump{d}{\beta}(x)$                & $\Lbeta \hbump{d}{\beta}(x)\mo{e}_1$ & $\Lbeta \hpulse{d}{\beta}\brOf{x}\mo{e}_2$ \\
            reference constant          & $v_0$                 & $0$                                           & $0$ & $L_0\mo{e}_1$ \\
            bump radius exponent        & $\zeta$               & $\beta-s$                                     & $\beta$ & $\beta$ \\
            radius upper bound          & $\rho_n^+$            & $\min\brOf{R/2, \rmax}$                       & $\min\brOf{1/2, \rmax}$ & $\min(1/2, \rmax)$ \\
            radius lower bound          & $\rho_n^-$            & $R (\const{cvrsc} n)^{-\frac1d}$                & $(\const{cvr} m)^{-\frac1d}$ & $\frac12 \const{cvrtm}^{-1} L_0 \Tsum n^{-1}$ \\
            reference parameter         & $\theta_{0}$          & $0$                                           & $0$ & $(f_0,\indset{x^\delta}{1}{m},\indset{T}{1}{m})$ \\
            observed object error & $\psi_n(r)$           & $c_{d,\beta} \Lbeta r^\beta$                & $c_{d,\beta} \Lbeta \Tmax r^\beta$ & $c_{\beta} \Lbeta L_0^{-1} r^{\beta+1}$ \\
            \# locations in support & $\chi_n(r)$           & $\const{cvrsc} R^{-d} r^d n$                  & $\const{cvr} r^d m \nmax$ & $c_d\const{cvrtm} L_0^{-1}\Tsum^{-1} r^d m n$ \\
            error radius exponent       & $\gamma$              & $2\beta+d$                                    & $2\beta+d$ & $2(\beta+1)+d$ \\
            inner packing constant      & $\const{pack}$        & $c_d R^d$                                     & $c_d$ & $c_d$
		\end{tabular}
	\end{center}
	\caption{Symbols of the Master Theorem and their respective realization or bounds in its three applications.}
	\label{tbl:symbolsMaster}
\end{table}

In the same model as for \cref{thm:lower:master}, we provide a deterministic version of the master theorem.
\begin{theorem}\label{thm:lower:deterministic:master}
	Assume there are $\theta_0, \theta_1\in\Theta$ such that $u_k(\theta_0) = u_k(\theta_1)$ for all $k\in\nnset{n}$.
	\begin{enumerate}[label=(\roman*)]
		\item
		Let $x_0\in\R^d$.
		Set $a := \euclOf{f_{\theta_0} - f_{\theta_1}}(x_0)$.
		Then
		\begin{equation}
			\inf_{\festi} \sup_{\theta\in\Theta} \Pr_\theta\brOf{\euclOf{\festi - f_\theta}(x_0) \geq \frac a2} = 1
			\eqfs
		\end{equation}
		\item
		Set $a := \sup_{x\in\xdomain}\euclOf{f_{\theta_0} - f_{\theta_1}}(x)$.
		Then
		\begin{equation}
			\inf_{\festi} \sup_{\theta\in\Theta} \Pr_\theta\brOf{\sup_{x\in\xdomain}\euclOf{\festi - f_\theta}(x) \geq \frac a2}  = 1
			\eqfs
		\end{equation}
		\item
		Let $p\in\Rpo$.
		Set $a := \LpNormOf{p}{\xdomain}{\euclOf{f_{\theta_0} - f_{\theta_1}}}$. Then
		\begin{equation}
			\inf_{\festi} \sup_{\theta\in\Theta} \Pr_\theta\brOf{\LpNormOf{p}{\xdomain}{\euclOf{\festi - f_\theta}} \geq \frac a2} = 1
			\eqfs
		\end{equation}
	\end{enumerate}
	Here, all infima range over all estimators $\festi$ of $\ftrue$ based on the observations $Y_{\nnset n}$.
\end{theorem}
\begin{proof}[Proof of \cref{thm:lower:deterministic:master}]
	As $u_k(\theta_0) = u_k(\theta_1)$ for all $k\in\nnset{n}$, we have $\Pr_{\theta_0} = \Pr_{\theta_1}$.
	Set $d(\theta, \tilde\theta)$ to be any one of the three metrics in which the theorem measures the error (point-wise, supremum, $L^p$) and apply \cref{thm:lower:deterministic:general}.
\end{proof}
\subsection{Example: Nonparametric Regression Lower Bounds}\label{sec:nonparamreg}
In this section, we apply the master theorem, \cref{thm:lower:master}, in the classical nonparametric regression setting: Let $d\in\N$ and $\beta\in\Rpp$.
Set $\ell := \llfloor\beta\rrfloor$.
Let $L_0, \dots, L_\ell, \Lbeta\in\Rpp$.
Set $\FSmooth := \Sigma^{d\to1}(\beta, \indset{L}{0}{\ell}, \Lbeta)$.
Let $n\in\N$ and $x_1, \dots, x_n \in \R^d$.
For $\ftrue\in\FSmooth$, consider the regression model equation
\begin{equation}
	Y_i = \ftrue(x_i) + \noise_i\eqcm \ i\in\nnset n\eqcm
\end{equation}
with independent and identically distributed copies $\noise_1,\dots,\noise_n$ of the $\R$-valued and centered error variable $\noise\sim P^\noise$.
Let $s\in\Nn$ with $s<\beta$. Let $\mo v \in \mb D_d^s$ be $s$-many $d$-dimensional directions. Let $R\in\Rpp$. Our goal is to estimate the $s$-th derivative of $\ftrue$ in the directions $\mo v$, i.e., $D^s_{\mo v}\ftrue$, on the domain $[0, R]^d$ with an estimator $\hat g$ based on the observations $Y_{\nnset n}$.

We require a scaled version of \assuRef{Cover} for our results to hold.
\begin{assumption}\mbox{ }
	\begin{itemize}
		\item \newAssuRef{CoverScaled}:
		Let $R\in\Rpp$. There is $\const{cvrsc}\in\Rpp$ such that, for all $z\in[0,R]^d$, $r\in\Rpp$, we have
		\begin{equation}
			\frac1n\sum_{i=1}^n \indOfOf{\ball^d(z, r)}{x_i} \leq \max\brOf{\frac1n,\, \const{cvrsc}\br{\frac{r}{R}}^d}
			\eqcm
		\end{equation}
		where $\indOfOf{A}{\cdot}$ denotes the indicator function of a given set $A$.
	\end{itemize}
\end{assumption}
\begin{theorem}\label{thm:lower:regression}
	Assume \assuRef{Noise} with constant $\const{noise}\in\Rpp$.
	Assume \assuRef{CoverScaled} with constant $\const{cvrsc}\in\Rpp$.
	Let $p\in\Rpo$.
	Define
	\begin{align}
		C &:= c_{d,\beta,s,p} \Lbeta \br{\const{noise} \const{cvrsc} R^{-d} \Lbeta^2}^{-\frac{\beta-s}{2\beta+d}}
		\eqcm\\
		\rmax &:= c_{d,\beta} \min_{k\in\nnset\ell}\br{\frac{L_k}{\Lbeta}}^{\frac1{\beta-k}}
		\eqfs
	\end{align}
	\begin{enumerate}[label=(\roman*)]
		\item
		For all $x_0\in[0,R]^d$, we have
		\begin{equation}
			\inf_{\hat g} \sup_{\ftrue\in\FSmooth}
			\Pr_{\ftrue}\brOf{\euclOf{\hat g - D_{\mo v}^s \ftrue}(x_0) \geq C{n}^{-\frac{\beta-s}{2\beta+d}}}
			\geq \frac14
		\end{equation}
		if
		\begin{align}
			(\const{cvrsc} R^{-d} n)^{\frac{2\beta}d} &\geq c_{\beta} \const{noise} \Lbeta^2
			\eqcm\\
			n &\geq c_{d,\beta} \const{noise}^{-1} \const{cvrsc}^{-1} \Lbeta^{-2} \max\brOf{R^{-2\beta}, R^{d} \rmax^{-(2\beta+d)}}
			\eqfs
		\end{align}
		\item
		We have
		\begin{equation}
			\inf_{\hat g} \sup_{\ftrue\in\FSmooth}
			\Pr_{\ftrue}\brOf{\sup_{x\in[0,R]^d}\euclOf{\hat g - D_{\mo v}^s \ftrue}(x) \geq C\br{\frac{n}{\log n}}^{-\frac{\beta-s}{2\beta+d}}}
			\geq \frac14
		\end{equation}
		if
		\begin{align}
			n &\geq c_{d,\beta}\max\brOf{1\eqcm\const{noise}^4 R^{4(2\beta+d)}} \const{cvrsc}^{-1} \Lbeta^{-2} R^{d}
			\eqcm\\
			(\const{cvrsc} R^{-d} n)^{\frac{2\beta}d} \log(n)
			&\geq
			c_{d,\beta} \const{noise} \Lbeta^2
			\eqcm\\
			n \log\brOf{n}^{-1}
			&\geq
			c_{d,\beta} \const{noise}^{-1} \const{cvrsc}^{-1} \Lbeta^{-2} \max\brOf{R^{-2\beta}, R^{d} \rmax^{-(2\beta+d)}}
			\eqcm\\
			n &\geq c_\beta \const{cvrsc}^2 \Lbeta^4 R^{-2d}
			\eqfs
		\end{align}
		\item
		We have
		\begin{equation}
			\inf_{\hat g} \sup_{\ftrue\in\FSmooth}
			\Pr_{\ftrue}\brOf{R^{-d/p}\LpNormOf{p}{[0,R]^d}{\euclOf{\hat g - D_{\mo v}^s \ftrue}} \geq C {n}^{-\frac{\beta-s}{2\beta+d}}}
			\geq \frac14
		\end{equation}
		if
		\begin{align}
			(\const{cvrsc} R^{-d} n)^{\frac{2\beta}d} &\geq c_\beta \const{noise} \Lbeta^2\eqcm
			\\
			n &\geq c_{d,\beta} \const{noise}^{-1} \const{cvrsc}^{-1} \Lbeta^{-2} \max\brOf{R^{-2\beta}, R^{d} \rmax^{-(2\beta+d)}}
			\eqfs
		\end{align}
	\end{enumerate}
	Here, all infima range over all estimators $\hat g$ of $D_{\mo v}^s \ftrue$ based on the observations $Y_{\nnset n}$.
\end{theorem}
The relation conditions between the parameters and the sample size can be summarized in the following way.
\begin{corollary}\label{cor:regression:simple}
	Assume \assuRef{Noise} for $P^\noise$ with constant $\const{noise}\in\Rpp$ and \assuRef{CoverScaled} for $(x_i)_{i\in\nnset n}$ with constant $\const{cvrsc}\in\Rpp$. Let $p\in\Rpo$.
	Then there are constants $n_0, C\in\Rpp$ large enough, where $n_0$ is only depending on $\beta, d, \indset{L}{0}{\ell}, \Lbeta, R, \const{noise}, \const{cvrsc}$ and potentially on $p$ and $C$ is only depending on $\beta, d, \Lbeta, \const{noise}, \const{cvrsc}$ and potentially on $p$, with the following property:
	Assume $n \geq n_0$. Then
	\begin{align}
		\forall x_0\in[0,R]^d\colon&\inf_{\hat g} \sup_{\ftrue\in\FSmooth}
		\Pr_{\ftrue}\brOf{C \euclOf{\hat g - D_{\mo v}^s \ftrue}(x_0) \geq \br{\frac{R^d}n}^{\frac{\beta-s}{2\beta+d}}}
		\geq \frac14
		\eqcm\\
		&\inf_{\hat g} \sup_{\ftrue\in\FSmooth}
		\Pr_{\ftrue}\brOf{C \sup_{x\in[0,R]^d}\euclOf{\hat g -D_{\mo v}^s \ftrue}(x) \geq \br{\frac{R^d\log(n)}{n}}^{\frac{\beta-s}{2\beta+d}}}
		\geq \frac14
		\eqcm\\
		&\inf_{\hat g} \sup_{\ftrue\in\FSmooth}
		\Pr_{\ftrue}\brOf{C R^{-d/p}\LpNormOf{p}{[0,R]^d}{\euclOf{\hat g - D_{\mo v}^s \ftrue}} \geq \br{\frac{R^d}n}^{\frac{\beta-s}{2\beta+d}}}
		\geq \frac14
		\eqfs
	\end{align}
  	Here, all infima range over all estimators $\hat g$ of $D_{\mo v}^s \ftrue$ based on the observations $Y_{\nnset n}$.
\end{corollary}
\begin{proof}[Proof of \cref{thm:lower:regression}]
    We intend to apply \cref{thm:lower:master}. Therefore we need to define the necessary objects in this setting. See \cref{tbl:symbolsMaster} for an overview. Firstly, let $h := \hbump{d}{\beta}\in\Sigma^{d\to1}(\beta, 1)$ be the bump function from \cref{lmm:bumpandpulse} and $\rmax = \rmax\brOf{\beta, \indset{L}{0}{\ell}, \Lbeta, h, 0}$ from the same lemma. The dimension-parameters are in this situation given by
    \begin{equation}
        \dm q := d\text{ and }\dm u := \dm f := 1.
    \end{equation}
    The \textit{domain of interest} is given by $\xdomain := [0,R]^d$ and following the model description above, we set
    \begin{equation}
        \Theta := \Sigma^{d\to1}(\beta, \indset{L}{0}{\ell}, \Lbeta)
        \eqfs
    \end{equation}
    As reference function, we define
    \begin{equation}
        g(x) := \Lbeta D^s_{\mo v} h(x)
    \end{equation}
    and set $v_0:= 0$. As hypotheses $\theta_0$ and $\theta_{z, r}$ we set
    \begin{equation}
        \theta_0 := 0
        \qquad\text{and}\qquad
		\theta_{z, r}(x) := \Lbeta r^\beta h\brOf{\frac{x-z}{r}}
		\eqfs
    \end{equation}
    Following the model description, for each $\theta\in\Theta$, we define the attributes
    \begin{enumerate}[label = (\roman*)]
        \item $q_k(\theta) := x_k$,\hfill(Location)
        \item $u_k(\theta) := \theta(x_k)$,\hfill(Observation)
        \item $f_\theta  :=  D^s_{\mo v}\theta$.\hfill(Function to Estimate)
    \end{enumerate}
    We set $\zeta := \beta-s$.
	Moreover, we set the range of the radius $r$ as
    \begin{equation}
        \rho_n^- :=  R (\const{cvrsc} n)^{-\frac1d}
        \qquad\text{and}\qquad
		\rho_n^+ := \min\brOf{R/2, \rmax}
		\eqfs
    \end{equation}
    We will show
	\begin{align}
		\psi_{n}(r) &\leq \supNormof{h} \Lbeta r^\beta\eqcm\\
		\chi_{n}(r) &\leq \const{cvrsc} R^{-d} r^d n\eqcm
	\end{align}
    such that
    \begin{equation}
        a_n :=  \supNormof{h}^2 \const{cvrsc} \Lbeta^2 R^{-d}n
    \end{equation}
    as well as
    \begin{equation}
        \gamma := 2\beta+d
        \eqfs
    \end{equation}
	In order to apply \cref{thm:lower:master}, we need to check \cref{ass:lower:general}:
	\begin{enumerate}[label=(\roman*)]
		\item[\ref{ass:lower:general:ref}:]
			The conditions on $q_k$ and $u_k$, as well as $\theta_0\in\Theta$ are trivial.
			By \cref{lmm:bumpandpulse} \ref{lmm:bumpandpulse:smooth}, we have $\theta_{z,r}\in\Theta$ if $r\leq \rmax$, which is true for $r\leq \rho^+_n$. Denote the $s$-th directional derivative with respect to $x$ as $D^s_{x,\mo v}$. Then, for all $x\in\xdomain$,
			\begin{align}
				f_{\theta_{z, r}}(x)
				&=
				(D^s_{\mo v}\theta_{z, r})(x)
				\\&=
				\Lbeta r^\beta D^s_{x,\mo v} h\brOf{\frac{x-z}{r}}
				\\&=
				\Lbeta r^{\beta-s} (D_{\mo v}^s h)\brOf{\frac{x-z}{r}}
				\\&=
				r^{\beta-s} g\brOf{\frac{x-z}{r}}
				\eqfs
			\end{align}
		\item[\ref{ass:lower:general:obs}:]
			We have
			\begin{align}
				\psi_{n}(r)
				&=
				\sup_{z\in\R^d} \sup_{k\in\nnset{n}} \Lbeta r^\beta h\brOf{\frac{x_k-z}{r}}
				=
				\supNormof{h} \Lbeta r^\beta
			\end{align}
			and, using \assuRef{CoverScaled},
			\begin{align}
				\chi_{n}(r)
				&=
				\sup_{z\in\R^d} \#\setByEle{k\in\nnset n}{x_k \in \ball^d\brOf{z, r}}
				\leq
				\const{cvrsc} R^{-d} r^d n
			\end{align}
			if $\const{cvrsc} R^{-d} r^d n \geq 1$, which is equivalent to $r \geq \rho_n^-$.
			Thus, we have
			\begin{align}
				\psi_{n}(r)^2 \chi_{n}(r)
				&\leq
				\supNormof{h}^2 \const{cvrsc} \Lbeta^2 R^{-d} r^{2\beta+d}  n
				=
				a_n r^\gamma
				\eqfs
			\end{align}
		\item[\ref{ass:lower:general:pack}:]
			For $r \leq R/2$, which is implied by $r \leq \rho^+_n$, we have
			\begin{equation}
				\eta(r) \geq \left\lfloor\frac{R}{2r}\right\rfloor^d \geq c_d R^d r^{-d}
				\eqfs
			\end{equation}
			Hence, we can set $\const{pack} := c_d R^d$.
		\item[\ref{ass:lower:general:add}:]
			The definitions above directly imply these conditions.
	\end{enumerate}
	By \cref{thm:lower:master}, we obtain the following:
	\begin{enumerate}[label=(\roman*)]
		\item
		By \cref{thm:lower:master} \ref{thm:lower:master:pointwise}, we obtain, for $n$ large enough,
		\begin{align}
			\frac14
			&\leq
			\inf_{\festi} \sup_{\theta\in\Theta} \Pr_\theta\brOf{\euclOf{\festi - f_\theta}(x_0) \geq \frac12 \supNormOf{g} \br{2 \const{noise} a_n}^{-\frac\zeta\gamma}}
			\\&\leq
			\inf_{\festi} \sup_{\theta\in\Theta}
			\Pr_\theta\brOf{\euclOf{\festi - f_\theta}(x_0) \geq C{n}^{-\frac{\beta-s}{2\beta+d}}}
		\end{align}
		with
		\begin{equation}
			C := c_{d, \beta, s} \Lbeta \br{\const{noise} \const{cvrsc} R^{-d} \Lbeta^2}^{-\frac{\beta-s}{2\beta+d}}
			\eqfs
		\end{equation}
		The conditions on $n$, which are all fulfilled if $n$ is large enough, are as follows:
		\begin{enumerate}[label=(\Roman*)]
			\item
			\begin{align}
				&
				(\rho_n^-)^{-\gamma} \geq 2 \const{noise} a_n \geq (\rho_n^+)^{-\gamma}
				\\\Leftrightarrow\qquad&
				(\const{cvrsc} R^{-d} n)^{\frac{2\beta+d}d} \geq 2 \supNormof{h}^2 \const{noise} \const{cvrsc} \Lbeta^2 R^{-d} n \geq \max\brOf{2R^{-1}, \rmax^{-1}}^{2\beta+d}
				\\\Leftrightarrow\qquad&
				(\const{cvrsc} R^{-d} n)^{\frac{2\beta}d} \geq c_{\beta} \const{noise} \Lbeta^2
				\\\quad\text{and}\quad&
				n \geq c_\beta \const{noise}^{-1} \const{cvrsc}^{-1} \Lbeta^{-2} \max\brOf{R^{-2\beta}, R^{d} \rmax^{-(2\beta+d)}}
				\eqfs
			\end{align}
		\end{enumerate}
		\item
		By \cref{thm:lower:master} \ref{thm:lower:master:sup}, we obtain, for $n$ large enough,
		\begin{align}
			\frac14
			&\leq
			\inf_{\festi} \sup_{\theta\in\Theta} \Pr_\theta\brOf{\sup_{x\in\xdomain}\euclOf{\festi - f_\theta}(x) \geq  \frac12 \supNormOf{g} \br{\frac{12 \gamma \const{noise}}{d} a_n \log(a_n)^{-1}}^{-\frac1\gamma}}
			\\&\leq
			\inf_{\festi} \sup_{\theta\in\Theta} \Pr_\theta\brOf{\sup_{x\in\xdomain}\euclOf{\festi - f_\theta}(x) \geq  A_1 \br{\frac{n}{\log(A_2 n)}}^{-\frac{\beta-s}{2\beta+d}}}
		\end{align}
		with
		\begin{align}
			A_1 &:=
				\frac12 \supNormof{D_{\mo v}^s h} \Lbeta \br{\frac{12 (2\beta + d) \const{noise}}{d} \const{cvrsc} R^{-d} \supNormof{h}^2 L^2}^{-\frac{\beta-s}{2\beta+d}}
			\\
			A_2 &:= \const{cvrsc} R^{-d} \supNormof{h}^2 \Lbeta^2
			\eqfs
		\end{align}
		If  $n \geq A_2^2$, then $\log(A_2 n) \geq \frac12 \log(n)$. In this case we obtain
		\begin{equation}
			\frac14 \leq
				\inf_{\festi} \sup_{\theta\in\Theta} \Pr_\theta\brOf{\sup_{x\in\xdomain}\euclOf{\festi - f_\theta}(x) \geq  C \br{\frac{n}{\log(n)}}^{-\frac{\beta-s}{2\beta+d}}}
		\end{equation}
		with
		\begin{equation}
			C := c_{d, \beta, s} \Lbeta \br{\const{noise}\const{cvrsc} R^{-d} \Lbeta^2}^{-\frac{\beta-s}{2\beta+d}}
			\eqfs
		\end{equation}
		The conditions on $n$, which are all fulfilled if $n$ is large enough, are as follows:
		\begin{enumerate}[label=(\Roman*)]
			\item
			\begin{align}
				&
				a_n \geq \max\brOf{2\eqcm\br{\frac{12 \gamma \const{noise}}{d\const{pack}^{\frac{\gamma}{d}}}}^{4}}
				\\\Leftrightarrow\qquad&
				\supNormof{h}^2 \const{cvrsc} \Lbeta^2 R^{-d}n\geq c_{d,\beta}\max\brOf{1\eqcm\br{\frac{\const{noise}}{R^{2\beta+d}}}^{4}}
				\\\Leftrightarrow\qquad&
				n\geq c_{d,\beta}\max\brOf{1\eqcm\const{noise}^4 R^{4(2\beta+d)}} \const{cvrsc}^{-1} \Lbeta^{-2} R^{d}
			\end{align}
			\item
			\begin{align}
				&
				(\rho_n^-)^{-\gamma} \geq \frac{12 \gamma \const{noise}}{d} a_n \log(a_n)^{-1} \geq (\rho_n^+)^{-\gamma}
				\\\Leftrightarrow\qquad&
				(\const{cvrsc} R^{-d} n)^{\frac{2\beta+d}d}
				\geq
				c_{d,\beta} \const{noise} \const{cvrsc} \Lbeta^2 R^{-d} n \log\brOf{A_2 n}^{-1}
				\geq
				\max\brOf{2R^{-1}, \rmax^{-1}}^{2\beta+d}
				\\\Leftrightarrow\qquad&
				(\const{cvrsc} R^{-d} n)^{\frac{2\beta}d} \log(n)
				\geq
				c_{d,\beta} \const{noise} \Lbeta^2
				\\\quad\text{and}\quad&
				n \log\brOf{n}^{-1}
				\geq
				c_{d,\beta} \const{noise}^{-1} \const{cvrsc}^{-1} \Lbeta^{-2} \max\brOf{R^{-2\beta}, R^{d} \rmax^{-(2\beta+d)}}
			\end{align}
			\item
			\begin{align}
				&
				n \geq A_2^2
				\\\Leftrightarrow\qquad&
				n \geq \br{\const{cvrsc} R^{-d} \supNormof{h}^2 \Lbeta^2}^2
				\\\Leftrightarrow\qquad&
				n \geq c_\beta \const{cvrsc}^2 \Lbeta^4 R^{-2d}
			\end{align}
		\end{enumerate}
		\item
		By \cref{thm:lower:master} \ref{thm:lower:master:lp}, we obtain, for $n$ large enough
		\begin{align}
			\frac14
			&\leq
			\inf_{\festi} \sup_{\theta\in\Theta} \Pr_\theta\brOf{\LpNormOf{p}{\xdomain}{\euclOf{\festi - f_\theta}} \geq  2^{-1-3/p} \const{pack}^{1/p} \LpNormOf{p}{\R^d}{g} \br{36 \const{noise} a_n}^{-\frac\zeta\gamma}}
			\\&\leq
			\inf_{\festi} \sup_{\theta\in\Theta} \Pr_\theta\brOf{R^{-\frac dp}\LpNormOf{p}{\xdomain}{\euclOf{\festi - f_\theta}} \geq  C n^{-\frac{\beta-s}{2\beta+d}}}
		\end{align}
		with
		\begin{equation}
			C := c_{d, \beta, s, p} \Lbeta \br{\const{noise} \const{cvrsc} R^{-d} \Lbeta^2}^{-\frac{\beta-s}{2\beta+d}}
			\eqfs
		\end{equation}
		The conditions on $n$, which are all fulfilled if $n$ is large enough, are as follows:
		\begin{enumerate}[label=(\Roman*)]
			\item
			\begin{align}
				&
				(\rho_n^-)^{-\gamma} \geq 36 \const{noise} a_n \leq (\rho_n^+)^{-\gamma}
				\\\Leftrightarrow\qquad&
				(\const{cvrsc} R^{-d} n)^{\frac{2\beta+d}d} \geq 36 \supNormof{h}^2 \const{noise} \const{cvrsc} \Lbeta^2 R^{-d} n \geq \max\brOf{2R^{-1}, \rmax^{-1}}^{2\beta+d}
				\\\Leftrightarrow\qquad&
				(\const{cvrsc} R^{-d} n)^{\frac{2\beta}d} \geq c_\beta \const{noise} \Lbeta^2
				\\\quad\text{and}\quad&
				n \geq c_{d,\beta} \const{noise}^{-1} \const{cvrsc}^{-1} \Lbeta^{-2} \max\brOf{R^{-2\beta}, R^{d}}
			\end{align}
			\item
			\begin{align}
				&
				36 \const{noise} a_n \geq \br{\frac18 \const{pack}}^{-\frac\gamma d}
				\\\Leftrightarrow\qquad&
				n \geq c_{d,\beta} \const{noise}^{-1} \const{cvrsc}^{-1} \Lbeta^{-2} R^{-2\beta}
			\end{align}
		\end{enumerate}
	\end{enumerate}
\end{proof}

%% file: sec_app_stubble.tex
\section{Proofs for Section \ref{sec:stubble}}\label{sec:app:stubble}
\subsection{Stubble Model -- Lower Bound -- Deterministic} \label{sec:app:stubble:deterministic}
\begin{theorem}\label{thm:stubble:deterministic:details}
	Let $d\in\N$, $\beta\in\Rpo$, $\ell := \llfloor\beta\rrfloor$, $L_0,\dots,L_\ell,\Lbeta \in\Rpp$.
	Set $\FSmooth := \Sigma^{d\to d}(\beta, \indset{L}{0}{\ell}, \Lbeta)$.
	Let $\stepsize \in\Rpp$ with $\stepsize \leq c L_0^{-1}$.
	Set $L := c_\beta \min\cbOf{1, L_0,\dots,L_\ell,\Lbeta}$.
	Let $x_0\in\R^d$, $p\in\Rpo$.
	Then, there are $f_0, f_1 \in \FSmooth$ with following property: We have
	\begin{equation}
		U(f_0, x, i\stepsize) = U(f_1, x, i\stepsize)
	\end{equation}
	for all $i\in\Z$, $x\in\R^d$ and
	\begin{equation}
		\euclOf{f_0 - f_1}(x_0) \geq c_{\beta} L L_0^{\beta+1} \stepsize^{\beta}
		\qquad\text{and}\qquad
		\LpNormOf{p}{\hypercube}{\euclOf{f_0 - f_1}} \geq c_{d,\beta,p} L L_0^{\beta+1} \stepsize^{\beta}
		\eqfs
	\end{equation}
\end{theorem}
\begin{proof}[Proof of \cref{thm:stubble:deterministic:details}]\mbox{}
\begin{enumerate}[label=(\alph*)]
	\item \textbf{Construction:}\\
	Let $K \colon \R \to \Rp$ be a function with $K \in \smoothC(\R)$ and $K(x)>0 \Leftrightarrow x \in (0,1)$.
	Define $\Kper(x) := K(x - \lfloor x \rfloor)$.
	Set $r:= \frac23 L_0 \stepsize$. Let $z\in\R$ to be chosen later. Define $g\colon\R\to\R$ by
	\begin{equation}
		g(x) := x + L r^{\beta+1} \Kper\brOf{\frac{x-z}{r}}\eqfs
	\end{equation}
	Then $g\in \smoothC(\R)$ with
	\begin{equation}
		g\pr(x) = 1 + L r^\beta  \Kper\pr\brOf{\frac{x-z}{r}}
		\eqfs
	\end{equation}
	As  $\stepsize \leq c L_0^{-1}$, we have $r^\beta \leq c_\beta$. Thus, we can assume that $L$ is defined such that $L r^\beta \supNormof{K\pr} \leq \frac12$. Then $g\pr(x) \in [\frac12, \frac 32]$. In particular, $g$ is a $\smoothC$-diffeomorphism of $\R$, i.e., the inverse $g^{-1}\colon\R\to\R$ exists and $g^{-1}\in \smoothC(\R)$.
	Define
	\begin{equation}
		f_0(x) := \frac23 L_0
		\qquad\text{and}\qquad
		f_1(x) := \frac23 L_0 g\pr(g^{-1}(x))\eqfs
	\end{equation}
	To find the solution to the ODE $\dot u = f_1(u)$, we calculate
	\begin{align}
		\frac{\dl}{\dl t}g\brOf{g^{-1}(x) + \frac23 L_0 t}
		=
		\frac23 L_0 g\pr\brOf{g^{-1}(x) + \frac23 L_0 t}
		=
		f_1\brOf{g\brOf{g^{-1}(x) + \frac23 L_0 t}}
		\eqfs
	\end{align}
	Thus,
	\begin{equation}
		 U(f_0, x, t) = x + \frac23 L_0 t
		 \qquad\text{and}\qquad
		 U(f_1, x, t) = g\brOf{g^{-1}(x) + \frac23 L_0 t}\eqcm
	\end{equation}
	for all $t,x\in\R$.
	\item \textbf{Smoothness:}\\
	Clearly, we have $f_0\in\Sigma^{1\to 1}(\beta, \indset{L}{0}{\ell}, \Lbeta)$. Furthermore, as $g\pr(x) \in [\frac12, \frac 32]$,
	\begin{equation}
		\supNormof{f_1} = \frac23 L_0 \supNormof{g\pr} \leq L_0
		\eqfs
	\end{equation}
	By \cref{lmm:chain:remainder} below, $f_1 \in\Sigma^{1\to 1}(\beta, \indset{L}{0}{\ell}, \Lbeta)$ for $r \leq c_\beta$ and $L \leq c_\beta \min_{k\in\nnset k\cup\cb{\beta}} L_k$, which is assumed.
	\item \textbf{Error at $x_0$:}\\
	For the difference at $x_0$ between the two constructed model functions, we obtain
	\begin{equation}
		\abs{f_0(x_0) - f_1(x_0)} = \frac23 L_0 \abs{1 - g\pr(g^{-1}(x_0))}
		\eqfs
	\end{equation}
	As we can set $z$ arbitrarily, $g\pr$ is continuous, and $g^{-1}$ is a diffeomorphism of $\R$, we can choose $z$ so that
	\begin{equation}
		\abs{1 - g\pr(g^{-1}(x_0))} = \sup_{x\in\R} \abs{1 - g\pr(x)}
		\eqfs
	\end{equation}
	With this choice, we obtain
	\begin{align}
		\abs{f_0(x_0) - f_1(x_0)}
		&=
		\frac23 L_0 L \supNormOf{K\pr} r^\beta
		\\&=
		L \supNormOf{K\pr} \br{\frac23 L_0}^{\beta+1} \stepsize^\beta
		\\&=
		c_\beta L L_0^{\beta+1} \stepsize^\beta
		\eqfs
	\end{align}
	\item \textbf{Error in $L_p$:}\\
	Because of the assumption $\stepsize \leq c L_0^{-1}$, we can continue with $r \leq 1$. Then
	\begin{align}
		\int_{[0,1]} \abs{f_0(x) - f_1(x)}^p \dl x
		&\geq
		\sum_{k=0}^{\lfloor1/r\rfloor-1} \int_{[z + kr,z + (k+1)r]} \abs{f_0(x) - f_1(x)}^p \dl x
		\\&=
		\lfloor1/r\rfloor \int_{[z,z + r]} \abs{f_0(x) - f_1(x)}^p \dl x
		\\&=
		\lfloor1/r\rfloor \br{\frac23 L_0}^p \int_{[z,z + r]}  \abs{1 - g\pr(g^{-1}(x))}^p \dl x
		\\&=
		\lfloor1/r\rfloor \br{\frac23 L_0 L r^\beta }^p \int_{[z,z + r]}  \abs{\Kper\pr\brOf{\frac{g^{-1}(x)-z}{r}}}^p \dl x
		\eqfs
	\end{align}
	As $g\pr(x) \in [\frac12, \frac 32]$, we can find $c_p$ independent of $L, r, \beta, z$ such that
	\begin{equation}
		\int_{[z,z + r]}  \abs{\Kper\pr\brOf{\frac{g^{-1}(x)-z}{r}}}^p \dl x \geq c_p
		\eqfs
	\end{equation}
	Thus,
	\begin{equation}
		\br{\int_{[0,1]} \abs{f_0(x) - f_1(x)}^p \dl x}^{\frac1p} \geq c_{\beta, p} L L_0^{\beta+1} \stepsize^\beta
		\eqfs
	\end{equation}
	\item \textbf{Identical Observations:}\\
	The function $x \mapsto g(x) - x$ is $r$-periodic, i.e., $g(x + ir) - (x + ir) = g(x) - x$ for all $i \in \mathbb Z$. Thus,
	\begin{equation}
		g(g^{-1}(x) + i r) = (g^{-1}(x) + i r) + g(g^{-1}(x)) - g^{-1}(x) = x + i r\eqfs
	\end{equation}
	Hence,
	\begin{align}
		U(f_1, x, i \stepsize)
		&=
		g\brOf{g^{-1}(x) + \frac23 L_0 i \stepsize}
		\\&=
		g(g^{-1}(x) + i r)
		\\&=
		x  + i r
		\\&=
		x + \frac23 L_0 i \stepsize
		\\&=
		U(f_0, x, i \stepsize)
		\eqfs
	\end{align}
	\item \textbf{Extension to arbitrary dimension $d$:}\\
	So far, we have shown the theorem for $d=1$. For $d>1$, extend the two model functions constructed above by letting them depend only on the first dimension and setting their values to 0 in all but the first dimension.
\end{enumerate}
\end{proof}
The proof of \cref{thm:stubble:deterministic:details} uses following lemma, which establishes the smoothness of the model function $f_1(x) = \frac23 L_0 s(x)$.
\begin{lemma}\label{lmm:chain:remainder}
	Let $K \colon \R \to \Rp$ be a function with $K \in \smoothC(\R)\cap\Sigma(\beta+1, L_K)$.
	Let $\const{kern} := \sup_{k\in\nnzset{\ell+1}}\supNormOf{K^{(k)}}$.
	Let $z\in\R$, $L\in\Rpp$, $\beta\in\Rpo$.
	Set $\ell:=\llfloor \beta\rrfloor$.
	Let $r_0\in\Rpp$ and $r\in(0,r_0]$.
	Define $g\colon\R\to\R$ by
	\begin{equation}
		g(x) := x + L r^{\beta+1} K\brOf{\frac{x-z}{r}}
		\eqfs
	\end{equation}
	Assume there is $g_0\in\Rpo$ with $g_0^{-1} \leq g\pr(x) \leq g_0$ for all $x\in\R$ and denote the inverse function of $g$ as $g^{-1}$.
	Define
	\begin{equation}
		s(x) := g\pr(g^{-1}(x))
		\eqfs
	\end{equation}
	Then there is $C\in\Rpp$ large enough, depending only on $g_0, r_0, \ell, \const{kern}$ with the following property:
	We have
	\begin{equation}
		s \in\Sigma^{1\to 1}(\beta, \indset{L}{0}{\ell}, \Lbeta)
	\end{equation}
	with
	$L_0 = C$, $L_k = C L r^{\beta-k}$ for $k\in\nnset\ell$, and $L_\beta = C L L_K$.
\end{lemma}
To prove this lemma, we will make use of Faá di Bruno's formula for the iterated chain rule.
For $k\in\N$, let $\mc P_k$ be the set of all partitions of $\nnset k$, i.e.,
\begin{equation}
	\mc P_k := \setByEle{B\subset 2^{\nnset k}}{\bigcup_{b\in B} b = \nnset k
	\quad\text{and}\quad
	\forall b_1, b_2\in B\colon b_1 \cap b_2 = \emptyset},
\end{equation}
where $2^{\nnset{k}}$ denotes the power set of $\nnset{k}$.
\begin{lemma}[Faá di Bruno's formula \cite{Craik05Prehistory}]\label{lmm:faa:di:bruno}
	Let $f,g \colon \R\to\R$ be $k$-times continuously differentiable functions.
	Then, for all $x\in\R$,
	\begin{equation}
		(f \circ g)^{(k)}(x) = \sum_{B \in \mc P_k} f^{(\# B)}\brOf{g(x)} \prod_{b\in B} \br{g}^{(\# b)}(x)
		\eqfs
	\end{equation}
\end{lemma}
\begin{proof}[Proof of \cref{lmm:chain:remainder}]\mbox{ }
	\begin{enumerate}[label=(\alph*)]
		\item \textbf{Faá di Bruno:}\\
		By Faá di Bruno's formula \cref{lmm:faa:di:bruno}, we have
		\begin{align}
			s^{(k)}(x)
			&=
			(g\pr \circ g^{-1})^{(k)}(x)
			\\&=
			\sum_{B \in \mc P_k} g^{(1+\# B)}\brOf{g^{-1}(x)} \prod_{b\in B} \br{g^{-1}}^{(\# b)}(x)
			\eqfs
		\end{align}
		Let $\ms{inv}\colon\Rpp\to\Rpp, x \mapsto x^{-1}$.
		We have,
		\begin{equation}
			\br{g^{-1}}^{(k+1)} = \br{\ms{inv} \circ g\pr \circ g^{-1}}^{(k)} = \br{\ms{inv} \circ s}^{(k)}
			\eqfs
		\end{equation}
		A second application of Faá di Bruno's formula yields
		\begin{align}
			\br{\ms{inv} \circ s}^{(k)}
			&=
			\sum_{W \in \mc P_k} \br{\ms{inv}}^{(\# W)}\brOf{s(x)} \prod_{w\in W} s^{(\# w)}(x)
			\\&=
			\sum_{W \in \mc P_k} \br{-1}^{\# W} \# W! s(x)^{-(1+\# W)} \prod_{w\in W} s^{(\# w)}(x)
			\eqfs
		\end{align}
		Together, we obtain
		\begin{align}
			s^{(k)}(x)
			&=
			\sum_{B \in \mc P_k} g^{(1+\# B)} \brOf{g^{-1}(x)} s(x)^{-l(B)} \prod_{\substack{b\in B\\\# b>1}} \br{\ms{inv} \circ s}^{(\#b-1)}
			\\&=
			\sum_{B \in \mc P_k} g^{(1+\# B)} \brOf{g^{-1}(x)} s(x)^{-l(B)} \prod_{\substack{b\in B\\\# b>1}} \sum_{W \in \mc P_{\#b-1}} a(\# W) s(x)^{-(1+\# W)} \prod_{w\in W} s^{(\# w)}(x)
			\eqcm
		\end{align}
		where $l(B) := \# \setByEle{b\in B}{\# b = 1}$ and $a(\# W) = \br{-1}^{\# W} \# W!$.
		Define
		\begin{equation}
			R_k(x) := \sum_{\substack{B \in \mc P_k\\\# B \neq k}} g^{(1+\# B)} \brOf{g^{-1}(x)} s(x)^{-l(B)} \prod_{\substack{b\in B\\\# b>1}} \sum_{W \in \mc P_{\#b-1}} a(\# W) s(x)^{-(1+\# W)} \prod_{w\in W} s^{(\# w)}(x)
		\end{equation}
		to obtain
		\begin{equation}\label{eq:faadirbruno:deriv}
			s^{(k)}(x)  = g^{(k+1)} \brOf{g^{-1}(x)} s(x)^{-k} + R_k(x)
			\eqfs
		\end{equation}
		As $s(x) \in[g_0^{-1}, g_0]$, we have
		\begin{equation}\label{eq:faadirbruno:rest}
			\abs{R_k(x)} \leq c_{k,g_0} \sum_{\substack{B \in \mc P_k\\\# B \neq k}} \abs{g^{(1+\# B)} \brOf{g^{-1}(x)}} \prod_{\substack{b\in B\\\# b>1}} \sum_{W \in \mc P_{\#b-1}} \prod_{w\in W} \abs{s^{(\# w)}(x)}
			\eqfs
		\end{equation}
		\item \textbf{Induction:}\\
		We now prove by induction that, for $k\in\nnset{\ell+1}$,
		\begin{equation}\label{eq:chain:remainder}
			\abs{R_k(x)} \leq c_{k,g_0,r_0} L r^{\beta-k+1}
			\eqfs
		\end{equation}
		For $k = 1$, the statement \eqref{eq:chain:remainder} is true with $R_1(x) = 0$ and
		\begin{equation}
			s^{(1)}(x) = g^{(2)}\brOf{g^{-1}(x)} \frac{1}{s(x)}
			\eqfs
		\end{equation}
		Set
		\begin{equation}
			v := \frac{g^{-1}(x)-z}{r}
			\eqfs
		\end{equation}
		If the statement \eqref{eq:chain:remainder} is true for $k\in\N$, $k\leq \ell$, it is also true for $k+1$:
		First, we have
		\begin{equation}
			g^{(1+\# B)} \brOf{g^{-1}(x)} = L r^{\beta-\# B} K^{(1+\# B)}\brOf{v}
		\end{equation}
		for $B\in\mc P_{k+1}$ with $\# B\neq k+1$, i.e., $\# B\in\nnset{k}$.
		Second, using \eqref{eq:chain:remainder} on $\# w$, we obtain
		\begin{equation}
			\abs{s^{(\# w)}(x)} \leq \abs{g^{(1+\# w)}(g^{-1}(x))s(x)^{-\# w}} + \abs{R_{\# w}(r)}  \leq c_{\# w} L r^{\beta - \# w}\eqcm
		\end{equation}
		where $\# w \in W$, $W\in \mc P_{\# b - 1}$, $b\in B$, and $B\in\mc P_{k+1}$, i.e., $\# w \in\nnset{k}$.
		Using these two bounds in \eqref{eq:faadirbruno:rest} yields 
		\begin{equation}
			\abs{R_{k+1}(x)} \leq c_{{k},g_0} \sum_{\substack{B \in \mc P_{k+1}\\\# B \neq k+1}} L r^{\beta-\#B} \prod_{\substack{b\in B\\\# b>1}} \sum_{W \in \mc P_{\#b-1}} \prod_{w\in W} L r^{\beta - \# w}
			\leq  c_{{k},g_0, r_0} L r^{\beta-k}
		\end{equation}
		with $r\leq r_0$.
		This implies \eqref{eq:chain:remainder} for $k+1$.
		\item \textbf{Derivative bounds:}\\
		For $k\in\nnset{\ell}$, \eqref{eq:faadirbruno:deriv} and $s(x) \in[g_0^{-1}, g_0]$ yield
		\begin{align}
			\supNormOf{s^{(k)}}
			&\leq
			c_{k, g_0} \supNormOf{g^{(k+1)}} + \supNormOf{R_k}
			\\&\leq
			c_{k, g_0} L r^{\beta-k} + c_{k, g_0, r_0} L r^{\beta-k+1}
			\\&\leq
			c_{k, g_0, r_0} L r^{\beta-k}
			\eqfs
		\end{align}
		It remains to show the bound on $\Lbeta$.
		Equation \eqref{eq:chain:remainder} implies that $R_\ell(x)$ is Lipschitz continuous with constant
		\begin{equation}
			\supNormOf{R_{\ell+1}} \leq c_{\ell, g_0, r_0} L r^{\beta-\ell}\eqfs
		\end{equation}
		Thus,
		\begin{align}
			\abs{s^{(\ell)}(x) - s^{(\ell)}(\tilde x)}
			&\leq
			\abs{g^{(\ell+1)} \brOf{g^{-1}(x)} s(x)^{-\ell} - g^{(\ell+1)} \brOf{g^{-1}(\tilde x)} s(\tilde x)^{-\ell}} + \abs{R_\ell(x) - R_\ell(\tilde x)}
			\\&\leq
			\frac{
				\abs{g^{(\ell+1)} (v) s(\tilde x)^{\ell} - g^{(\ell+1)} \brOf{\tilde v} s(x)^{\ell}}
			}{ s(x)^{\ell}s(\tilde x)^{\ell}}
			+ c_{\ell, g_0, r_0} L r^{\beta-\ell} \abs{x-\tilde x}
		\end{align}
		with $v := g^{-1}(x)$, $\tilde v := g^{-1}(\tilde x)$. As $s(x) \in[g_0^{-1}, g_0]$, only need to find a suitable bound on
		\begin{equation}
			\abs{g^{(\ell+1)} (v) s(\tilde x)^{\ell} - g^{(\ell+1)} \brOf{\tilde v} s(x)^{\ell}}
			\leq
			\abs{s(\tilde x)^{\ell}} \abs{g^{(\ell+1)} (v)  - g^{(\ell+1)} \brOf{\tilde v}}
			+
			\abs{g^{(\ell+1)} \brOf{\tilde v}} \abs{s(\tilde x)^{\ell} - s(x)^{\ell}}
			\eqfs
		\end{equation}
		For the first term,
		\begin{align}
			\abs{s(\tilde x)^{\ell}} \abs{g^{(\ell+1)} (v)  - g^{(\ell+1)} \brOf{\tilde v}}
			&\leq
			c_{\ell, g_0, r_0} L r^{\beta-\ell} \abs{K^{(\ell+1)}\brOf{\frac{v-z}{r}}  - K^{(\ell+1)}\brOf{\frac{\tilde v-z}{r}}}
			\\&\leq
			c_{\ell, g_0, r_0} L r^{\beta-\ell}
				L_K\abs{\frac{v-z}{r}  - \frac{\tilde v-z}{r}}^{\beta-\ell}
			\\&=
			c_{\ell, g_0, r_0} L
				L_K \abs{\tilde v - v}^{\beta-\ell}
			\eqfs
		\end{align}
		For the second term
		\begin{align}
			\abs{g^{(\ell+1)} \brOf{\tilde v}} \abs{s(\tilde x)^{\ell} - s(x)^{\ell}}
			&\leq
			c_{\ell, g_0, r_0} L r^{\beta-\ell} \abs{g\pr(\tilde v)^{\ell} - g\pr(v)^{\ell}}
			\\&\leq
			c_{\ell, g_0, r_0} L r^{\beta-\ell}
			\br{
				c_{\ell, g_0} L r^{\beta-1} \abs{\tilde v - v}
			}
			\\&\leq
			c_{\ell, g_0, r_0} L r^{2\beta-\ell-1} \abs{\tilde v - v}
		\end{align}
		as
		\begin{equation}
			\supNormOf{\frac{\dl }{\dl x}\br{(g\pr)^\ell}} \leq \ell \supNormOf{(g\pr)^{\ell-1}} \supNormOf{g\prr} \leq c_{\ell, g_0} L r^{\beta-1}
			\eqfs
		\end{equation}
		Thus, 
		\begin{align*}
			\abs{s^{(\ell)}(x) - s^{(\ell)}(\tilde x)}
			&\leq
			c_{\ell, g_0, r_0} L \br{L_K \abs{\tilde v - v}^{\beta-\ell} +  r^{\beta-\ell} \abs{x-\tilde x} + r^{2\beta-\ell-1} \abs{\tilde v - v}}
			\\&\leq
			c_{\ell, g_0, r_0} L L_K \abs{\tilde x - x}^{\beta-\ell}
			\eqfs
		\end{align*}
	\end{enumerate}
\end{proof}
\subsection{Stubble Model -- Lower Bound -- Probabilistic}\label{sec:app:stubble:probabilistic}
Let $d\in\N$ and $\beta\in\Rpo$.
Set $\ell := \llfloor\beta\rrfloor$.
Let $L_0, \dots, L_\ell, \Lbeta\in\Rpp$.
Set $\FSmooth := \Sigma^{d\to d}(\beta, \indset{L}{0}{\ell}, \Lbeta)$.
Let $m\in\N$ and $\xany 1, \dots, \xany m \in \R^d$, $n_1, \dots, n_m\in\N$, $T_1, \dots, T_m\in\Rpp$. Set $\nmax := \max_{j\in\nnset m} n_j$ and $\Tmax := \max_{j\in\nnset m} T_j$.
Set $n := \sum_{j\in\nnset{m}} n_j$.
Set $\setIdx := \setByEleInText{(j,i)}{i\in\nnset{n_j}, j\in\nnset{m}}$.
Let $t_{j,i}\in\Rp$ for $(j,i)\in\setIdx$ such that
$0 \leq t_{j,1} \leq \dots \leq t_{j,n_j} = T_j$.
For $\ftrue\in\FSmooth$, consider the ODE model
\begin{equation}
	Y_{j,i} = U(\ftrue, \xany j, t_{j,i}) + \noise_{j,i}\eqcm \ (j,i)\in\setIdx\eqcm
\end{equation}
with independent and identically distributed $\R^d$-valued noise variables $\noise_{j,i}$ such that $\Eof{\noise_{j,i}} = 0$. Denote the noise distribution as $P^\noise$, i.e., $\noise_{j,i}\sim P^\noise$.
We know $\xany{j}$ and $t_{j,i}$, and observe $Y_{j,i}$, but $\ftrue$ is unknown and to be estimated on the domain $[0, 1]^d$.
\begin{theorem}\label{thm:stubble:probabilistic:details}
	Assume \assuRef{Noise} for $P^\noise$ with constant $\const{noise}\in\Rpp$ and \assuRef{Cover} for $(\xany j)_{j\in\nnset m}$ with constant $\const{cvr}\in\Rpp$.
	Assume
	\begin{equation}\label{eq:stubble:noise:rmaxAssumption}
		\min_{k\in\nnzset\ell} L_k \geq c_\beta L_\beta
		\eqfs
	\end{equation}
	\begin{enumerate}[label=(\roman*)]
		\item
			Let $x_0\in\hypercube$. Then
			\begin{equation}
				\inf_{\festi} \sup_{\ftrue\in\FSmooth}
				\Pr_{\ftrue}\brOf{\euclOf{\festi - \ftrue}(x_0) \geq C_1\br{m \nmax \Tmax^2}^{-\frac{\beta}{2\beta+d}}}
				\geq \frac14
			\end{equation}
			where
			\begin{equation}
				C_1 := c_{d, \beta} \Lbeta \br{\const{cvr} \const{noise} \Lbeta^2}^{-\frac{\beta}{2\beta+d}}
			\end{equation}
			if
			\begin{align}
				\br{\const{cvr}m}^{\frac{2\beta}d} &\geq c_\beta \const{noise} \Lbeta^2 \nmax \Tmax^2
				\eqcm\\
				m \nmax \Tmax^2 &\geq c_{d,\beta} \const{noise}^{-1} \const{cvr}^{-1} \Lbeta^{-2}
				\eqfs
			\end{align}
		\item
			We have
			\begin{equation}
				\inf_{\festi} \sup_{\ftrue\in\FSmooth}
				\Pr_{\ftrue}\brOf{\sup_{x\in\hypercube}\euclOf{\festi - \ftrue}(x) \geq C_1 \br{\frac{m \nmax \Tmax^2}{\log\brOf{m \nmax \Tmax^2}}}^{-\frac\beta{2\beta+d}}}
				\geq \frac14
			\end{equation}
			where
			\begin{equation}
				C_1 := c_{d,\beta} \Lbeta \br{\const{noise} \const{cvr} \Lbeta^2}^{-\frac\beta{2\beta+d}}
			\end{equation}
			if
			\begin{align}
				m \nmax \Tmax^2 &\geq  c_{d, \beta} \max\brOf{1\eqcm\const{noise}^4} \const{cvr}^{-1} \Lbeta^{-2}
				\\
				\br{\const{cvr}m}^{\frac{2\beta}d} &\geq c_{d, \beta} \const{noise}  \Lbeta^2 \nmax \Tmax^2
				\\
				m \nmax \Tmax^2 &\geq c_{d, \beta} \const{noise}^{-1} \const{cvr}^{-1} \Lbeta^{-2}
				\\
				m \nmax \Tmax^2 &\geq \const{cvr} \Lbeta^2
				\eqfs
			\end{align}
		\item
			We have
			\begin{equation}
				\inf_{\festi} \sup_{\ftrue\in\FSmooth}
				\Pr_{\ftrue}\brOf{\LpNormOf{p}{\hypercube}{\euclOf{\festi - \ftrue}} \geq C_1\br{m \nmax \Tmax^2}^{-\frac{\beta}{2\beta+d}}}
				\geq \frac14
			\end{equation}
			where
			\begin{equation}
				C_1 := c_{d, \beta, p} \Lbeta \br{\const{cvr} \const{noise} \Lbeta^2}^{-\frac{\beta}{2\beta+d}}
			\end{equation}
			if
			\begin{align}
				m \nmax \Tmax^2 &\geq c_{d,\beta} \const{noise}^{-1} \const{cvr}^{-1}  \Lbeta^{-2}
				\eqcm\\
				\br{\const{cvr}m}^{\frac{2\beta}d} &\geq c_\beta \const{noise} \Lbeta^2 \nmax \Tmax^2
				\eqfs
			\end{align}
	\end{enumerate}
	Here, all infima range over all estimators $\festi$ of $\ftrue$ based on the observations $Y_{\setIdx}$.
\end{theorem}
\begin{proof}[Proof of \cref{thm:stubble:probabilistic:details}]
    We want to apply \cref{thm:lower:master}. For this, we first map the symbols used there with the objects we are dealing with here. Then, we prove that the assumptions made in \cref{thm:lower:master} are fulfilled. Finally, we present the result of the application of \cref{thm:lower:master}.
    
    \textbf{Mapping Symbols.}
    See \cref{tbl:symbolsMaster} for an overview.
    We start with a consecutive indexing of the observation scheme: For $k\in\nnset n$, denote $j_k \in\nnset m$ and $i_k\in\nnset{j_k}$ indices of the observations, such that all observations are enumerated, i.e., $\#\setByEleInText{(j_k, i_k)}{k\in\nnset n} = n$. Now, set the dimension-parameters in this setting as
    \begin{equation}
        \dm q := \dm u := \dm f := d
    \end{equation}
    and the domain of interest as $\xdomain := \hypercube$. The parameter space $\Theta$ is given by $\Sigma^{d\to d}(\beta, \indset{L}{0}{\ell}, \Lbeta)$. For each function $\theta\in\Theta$ we define the following attributes:
    \begin{enumerate}[label = (\roman*)]
        \item (Location): $q_k(\theta) := \xany{j_k}$ --- The initial condition, which does not depend on the choice of $\theta$.
        \item (Observed Objects): $u_k(\theta) := U(\theta, x_{j_k}, t_{j_k, i_k})$ --- The location along the $j_k$-th trajectory at time $ t_{j_k, i_k}$ driven by the model function $\theta$ with fixed initial condition $x_{j_k}$.
        \item (Function to Estimate):  $f_\theta  := \theta|_\xdomain$ --- The model function $\theta$ restricted to the domain of interest $\xdomain$.
    \end{enumerate}
    \textbf{Verifying Assumptions.} Let us check \cref{ass:lower:general}:
	\begin{enumerate}[label=(\roman*)]
		\item[\ref{ass:lower:general:ref}:]
            We define the reference function $g:\R^d\to\R^d$ for any $x\in\R^d$ by
            \begin{equation}
                g(x) := \Lbeta h(x)\mo{e}_1
            \end{equation}
            with $h:=\hbump{d}{\beta}$ as in \cref{lmm:bumpandpulse} and set $\rmax = \rmax\brOf{\beta, \indset{L}{0}{\ell}, \Lbeta, h, 0}$ from the same lemma.  By \eqref{eq:stubble:noise:rmaxAssumption}, we have $\rmax \geq c_{\beta}$. We have $\operatorname{supp}(g) \subseteq \ball^d(0,1)$.
            Set $v_0 := 0$ and
            \begin{equation}
            	f_{1,z,r}(x) := \Lbeta r^\beta h\left(\frac{x-z}{r}\right)
            	\eqfs
            \end{equation}
            Moreover, for $x\in\R^d$, we define hypotheses
            \begin{equation}
                \theta_0(x) := 0
                \quad\text{and}\quad
                \theta_{z,r}(x) := f_{1,z,r}(x)\mo{e}_1 = \Lbeta r^\beta h\left(\frac{x-z}{r}\right)\mo{e}_1\eqcm
            \end{equation}
            for any $z\in\R^d$ and $r\in\Rpp$. \cref{lmm:bumpandpulse} then yields $\theta_0,\theta_{z,r}\in\Sigma^{d\to d}(\beta,\indset{L}{0}{\ell}L_\beta)$ for $r \leq \rmax$. Thus, we define $\zeta:=\beta$. Moreover, the definitions above are consistent, i.e., we have
			\begin{equation}
				f_{\theta_{z, r}}(x)  = r^{\zeta} g\brOf{\frac{x-z}{r}} + v_0
				\eqfs
			\end{equation}
            Furthermore, we define
            \begin{equation}
                \rho_n^- := (\const{cvr} m)^{-\frac1d}
                \quad\text{and}\quad
                \rho_n^+ := \min\brOf{\frac12, \rmax}.
            \end{equation}
            Then, the condition on $\theta_0\in\Theta$ is trivial. We have $\xany{j_k}=q_k(\theta_0)=q_k(\theta_{z,r})$. If $\xany{j_k}\in\R^d\setminus\ball^d(z,r)$, we have $\theta_{z, r}(\xany{j_k}) = 0$. Hence, the solution to the ODE with model function $\theta_{z, r}$ and initial conditions $\xany{j_k}$ is constant. Therefore, $u_k(\theta_{z,r}) = \xany{j_k} =u_k(\theta_0)$, i.e., the condition on $u_k$ is fulfilled.
			By \cref{lmm:bumpandpulse} \ref{lmm:bumpandpulse:smooth}, we have $\theta_{z,r}\in\Theta$ if $r \leq \rmax$, which is true for  $r \leq \rho_n^+$.
        \item[\ref{ass:lower:general:obs}:]
            For $r\in [\rho_n^-, \rho_n^+]$,
            we show
        	\begin{align}
        		\psi_{n}(r) \leq  \supNormof{h} \Lbeta \Tmax r^\beta
        		\quad\text{and}\quad
        		\chi_{n}(r) \leq  \const{cvr} r^d m \nmax
        		\eqfs
        	\end{align}
            We then set
            \begin{equation}
                a_{n} :=  c_\beta \const{cvr} m \nmax \Tmax^2 \Lbeta^2
                \quad\text{and}\quad
                \gamma := 2\beta+d
            \end{equation}
            to obtain
            \begin{equation}
                \psi_{n}(r)^2 \chi_{n}(r) \leq a_n r^\gamma
                \eqfs
            \end{equation}
            Indeed, we have
			\begin{align}
				\psi_{n}(r)
				&=
				\sup_{z\in\R^d} \sup_{k\in\nnset{n}} \euclOf{u_k(\theta_{z,r}) - u_k(\theta_0)}
				\\&=
				\sup_{z\in\R^d} \sup_{k\in\nnset{n}} \euclOf{U(\theta_{z,r}, \xany{j_k}, t_{j_k, i_k}) - U(\theta_0, \xany{j_k}, t_{j_k, i_k})}
				\eqfs
			\end{align}
			We have $U(\theta_0, x, t) = x$ for all $x\in\hypercube$, $t\in\R$. Moreover, $U(\theta_{z,r}, x, t) - x = v(x,t) \mo{e}_1$,
			where
			\begin{equation}
				v(x,t) := \int_0^t f_{1, z, r}(U(\theta_{z,r}, x, s)) \dl s
				\eqfs
			\end{equation}
			Set $b :=  \Lbeta r^\beta \supNormof{h}$.
			We have $0\leq f_{1, z, r}(x)\leq b$. Hence, $0\leq v(x, t) \leq tb$. Therefore,
			\begin{equation}
				\psi_{n}(r) \leq \Tmax b
				\eqfs
			\end{equation}
			Next, using \assuRef{Cover}, we obtain
			\begin{align}
				\sum_{j\in\nnset m} \indOfOf{\ball^d(z, r)}{\xany{j}}
				&\leq
				\max\brOf{1,\,\const{cvr} r^d m}
				\\&=  \const{cvr} r^d m
				\eqcm
			\end{align}
			for $ \const{cvr} r^d m \geq 1$, which is equivalent to $r \geq \rho_n^-$. Thus,
			\begin{align}
				\chi_{n}(r)
				&=
				\sup_{z\in\R^d} \#\setByEle{k\in\nnset n}{q_k(\theta_{z,r}) \in\ball^d(z,r)}
				\\&=
				\sup_{z\in\R^d} \sum_{j\in\nnset{m}} \indOfOf{\ball^d(z, r)}{\xany{j}} n_j
				\\&\leq
				 \const{cvr} r^d m \nmax
				\eqfs
			\end{align}
		\item[\ref{ass:lower:general:pack}:]
			As $r \leq \rho^+$, we have $r \leq 1/2$. Thus,
			\begin{equation}
				\eta(r) \geq \left\lfloor\frac{1}{2r}\right\rfloor^d \geq c_d r^{-d}
				\eqfs
			\end{equation}
			Hence, we can set $\const{pack} := c_d$.
		\item[\ref{ass:lower:general:add}:]
			The definitions above directly imply these conditions.
    \end{enumerate}
	\textbf{Presenting Results.}
	After proving \cref{ass:lower:general}, we apply \cref{thm:lower:master} and evaluate the results. Recall
	\begin{align}
		\supNormOf{g} &= c_{d,\beta} \Lbeta\eqcm \\
		a_{n} &=  c_\beta \const{cvr} m \nmax \Tmax^2 \Lbeta^2 \eqcm \\
		\zeta &= \beta \eqcm\\
		\gamma &= 2\beta+d\eqcm\\
		\rho_n^+ &=  \min\brOf{\frac12, \rmax}\eqcm\\
		\rho_n^- &= (\const{cvr} m)^{-\frac1d}
		\eqfs
	\end{align}
	\begin{enumerate}[label=(\roman*)]
		\item
			By \cref{thm:lower:master} \ref{thm:lower:master:pointwise}, we obtain the following lower bound.
			We have
			\begin{align}
				\frac14
				&\leq
				\inf_{\festi} \sup_{\theta\in\Theta} \Pr_\theta\brOf{\euclOf{\festi - f_\theta}(x_0) \geq \frac12 \supNormOf{g} \br{2 \const{noise} a_{n}}^{-\frac\zeta{\gamma}}}
				\\&\leq
				\inf_{\festi} \sup_{\ftrue\in\FSmooth}
				\Pr_{\ftrue}\brOf{\euclOf{\festi - \ftrue}(x_0) \geq C_1\br{m \nmax \Tmax^2}^{-\frac{\beta}{2\beta+d}}}
			\end{align}
			with
			\begin{equation}
				C_1 := c_{d, \beta} \Lbeta \br{\const{cvr} \const{noise} \Lbeta^2}^{-\frac{\beta}{2\beta+d}}
			\end{equation}
			The conditions for this bound to hold are:
			\begin{enumerate}[label=(\Roman*)]
				\item
					\begin{align}
						&
						(\rho_n^-)^{-\gamma} \geq 2 \const{noise}a_{n} \geq (\rho_n^+)^{-\gamma}
						\\\Leftrightarrow\qquad&
						(\const{cvr} m)^{\frac{2\beta+d}d} \geq c_{\beta} \const{noise} \const{cvr} m \nmax \Tmax^2 \Lbeta^2 \geq \max\brOf{2, \rmax^{-1}}^{2\beta+d}
						\\\Leftrightarrow\qquad&
						\br{\const{cvr}m}^{\frac{2\beta}d} \geq c_\beta \const{noise} \Lbeta^2 \nmax \Tmax^2
						\\\quad \text{and} \quad&
						m \nmax \Tmax^2 \geq c_{d,\beta} \const{noise}^{-1} \const{cvr}^{-1} \Lbeta^{-2}
						\eqfs
					\end{align}
			\end{enumerate}
		\item
			By \cref{thm:lower:master} \ref{thm:lower:master:sup}, we obtain the following lower bound.
			We have
			\begin{align}
				\frac14
				&\leq
				\inf_{\festi} \sup_{\theta\in\Theta}
				\Pr_\theta\brOf{\sup_{x\in\xdomain}\euclOf{\festi - f_\theta}(x) \geq \frac12 \supNormOf{g} \br{\frac{12 \gamma \const{noise}}{\dm q} a_{n} \log(a_{n})^{-1}}^{-\frac\zeta{\gamma}}}
				\\&\leq
				\inf_{\festi} \sup_{\ftrue\in\FSmooth}
				\Pr_{\ftrue}\brOf{\sup_{x\in\hypercube}\euclOf{\festi - \ftrue}(x) \geq C_1 \br{\frac{m \nmax \Tmax^2}{\log\brOf{A_1 m \nmax \Tmax^2}}}^{-\frac\beta{2\beta+d}}}
			\end{align}
			with $A_1 := \const{cvr} \Lbeta^2$ and
			\begin{equation}
				C_1 := c_{d,\beta} \Lbeta \br{\const{noise} \const{cvr} \Lbeta^2}^{-\frac\beta{2\beta+d}}\eqfs
			\end{equation}
			If additionally $m \nmax \Tmax^2 \geq A_1$, then
			\begin{equation}
				\frac14
				\leq
				\inf_{\festi} \sup_{\ftrue\in\FSmooth}
				\Pr_{\ftrue}\brOf{\sup_{x\in\hypercube}\euclOf{\festi - \ftrue}(x) \geq C_1 \br{\frac{m \nmax \Tmax^2}{\log\brOf{m \nmax \Tmax^2}}}^{-\frac\beta{2\beta+d}}}
				\eqfs
			\end{equation}
			The conditions for this bound to hold are:
			\begin{enumerate}[label=(\Roman*)]
				\item
				\begin{align}
					&
					a_{n} \geq \max\brOf{2\eqcm\br{\frac{12 \gamma \const{noise}}{\dm q\const{pack}^{\frac{\gamma}{\dm q}}}}^{4}}
					\\\Leftrightarrow\qquad&
					m \nmax \Tmax^2 \geq  c_{d, \beta} \max\brOf{1\eqcm\const{noise}^4} \const{cvr}^{-1} \Lbeta^{-2}
					\eqfs
				\end{align}
				\item
				\begin{align}
					&
					(\rho_n^-)^{-\gamma} \geq \frac{12 \gamma \const{noise}}{\dm q} a_{n} \geq (\rho_n^+)^{-\gamma}
					\\\Leftrightarrow\qquad&
					(\const{cvr} m)^{\frac{2\beta+d}d} \geq c_{d, \beta} \const{noise} \const{cvr} \Lbeta^2 m \nmax \Tmax^2 \geq \max\brOf{2, \rmax^{-1}}^{2\beta+d}
					\\\Leftrightarrow\qquad&
					\br{\const{cvr}m}^{\frac{2\beta}d} \geq c_{d, \beta} \const{noise}  \Lbeta^2 \nmax \Tmax^2
					\\\quad \text{and} \quad&
					m \nmax \Tmax^2 \geq c_{d, \beta} \const{noise}^{-1} \const{cvr}^{-1} \Lbeta^{-2}
					\eqfs
				\end{align}
				\item
				\begin{align}
					&
					m \nmax \Tmax^2 \geq A_1
					\\\Leftrightarrow\qquad&
					m \nmax \Tmax^2 \geq \const{cvr} \Lbeta^2
				\end{align}
			\end{enumerate}
		\item
		By \cref{thm:lower:master} \ref{thm:lower:master:lp}, we obtain the following lower bound.
		We have
		\begin{align}
			\frac14
			&\leq
			\inf_{\festi} \sup_{\theta\in\Theta}
			\Pr_\theta\brOf{\LpNormOf{p}{\xdomain}{\euclOf{\festi - f_\theta}} \geq  2^{-1-3/p} \const{pack}^{1/p} \LpNormOf{p}{\R^{\dm q}}{g} \br{ 36 \const{noise} a_{n}}^{-\frac\zeta{\gamma}}}
			\\&\leq
			\inf_{\festi} \sup_{\ftrue\in\FSmooth}
			\Pr_{\ftrue}\brOf{\LpNormOf{p}{\hypercube}{\euclOf{\festi - \ftrue}} \geq C_1\br{m \nmax \Tmax^2}^{-\frac{\beta}{2\beta+d}}}
		\end{align}
		with
		\begin{equation}
			C_1 := c_{d, \beta, p} \Lbeta \br{\const{cvr} \const{noise} \Lbeta^2}^{-\frac{\beta}{2\beta+d}}
			\eqfs
		\end{equation}
		The conditions for this bound to hold are:
		\begin{enumerate}[label=(\Roman*)]
			\item
			\begin{align}
				&
				36 \const{noise} a_{n} \geq \br{\frac18 \const{pack}}^{-\frac{\gamma}{\dm q}}
				\\\Leftrightarrow\qquad&
				m \nmax \Tmax^2 \geq c_{d,\beta} \const{noise}^{-1} \const{cvr}^{-1}  \Lbeta^{-2}
			\end{align}
			\item
			\begin{align}
				&
				(\rho_n^-)^{-\gamma} \geq 36 \const{noise}a_{n} \geq (\rho_n^+)^{-\gamma}
				\\\Leftrightarrow\qquad&
				(\const{cvr} m)^{\frac{2\beta+d}d} \geq c_\beta \const{noise} \const{cvr} m \nmax \Tmax^2 \Lbeta^2  \geq \max\brOf{2, \rmax^{-1}}^{2\beta+d}
				\\\Leftrightarrow\qquad&
				\br{\const{cvr}m}^{\frac{2\beta}d} \geq c_\beta \const{noise} \Lbeta^2 \nmax \Tmax^2
				\\\quad \text{and} \quad&
				m \nmax \Tmax^2 \geq c_{d, \beta} \const{noise}^{-1} \const{cvr}^{-1} \Lbeta^{-2}
				\eqfs
			\end{align}
		\end{enumerate}
	\end{enumerate}
\end{proof}

%% file: sec_app_snake.tex
\section{Proofs for Section \ref{sec:snake}}\label{sec:app:snake}
\subsection{Snake Model -- Lower Bound -- Deterministic}\label{sec:app:snake:deterministic}
\begin{theorem}\label{thm:snake:deterministic:details}
	Let $d\in\N_{\geq 2}$. Let $\beta\in\Rpo$, $L_0, \dots, L_\ell, \Lbeta\in\Rpp$. Let $\delta\in\Rpp$. Let $x_0\in\R^d$, $p\in\Rpo$.
	Assume $\delta \leq \delta_0$, where $\delta_0\in\Rpp$ depends only on $d, \beta,\indset{L}{0}{\ell}, \Lbeta$.
	Then, there are $m\in\N$, $x_1, \dots, x_m \in \R^d$, $f_0, f_1\in \FSmooth$ that have following properties:
	\begin{enumerate}[label=(\roman*)]
		\item
		We have
		\begin{equation}
			m = \br{1 + \left\lceil\frac{\sqrt{d}}{2\delta}\right\rceil}^{d-1}
			\eqfs
		\end{equation}
		Set $\theta_0 := (f_0, \indset{x}{1}{m},\indset{T}{1}{m})$ and $\theta_1 := (f_1, \indset{x}{1}{m},\indset{T}{1}{m})$ with $T_j = L_0^{-1}$. Then we have $\theta_0, \theta_1\in \ParamSnake(\indset{L}{0}{\ell}, \Lbeta, \delta, \Tsum)$ with $\Tsum := m/L_0$.
		\item
		For all $t\in\R$, $j \in \nnset m$, we have
		\begin{equation}
			U(f_0, x_j, t) = U(f_1, x_j, t)
			\eqfs
		\end{equation}
		\item
		We have
		\begin{equation}
			\euclOf{f_0 - f_1}(x_0) \geq c_{d,\beta} \Lbeta \delta^\beta
			\qquad\text{and}\qquad
			\LpNormOf{p}{\hypercube}{\euclOf{f_0 - f_1}} \geq c_{d, \beta,p} \Lbeta \delta^\beta
			\eqfs
		\end{equation}
	\end{enumerate}
\end{theorem}
In this subsection we prove \cref{thm:snake:deterministic} in the form of \cref{thm:snake:deterministic:details}. We will construct two different model functions with identical trajectories for appropriately chosen initial conditions.
We need the following definition.
\begin{definition}\label{def:hausdorff}\mbox{}
	For $k\in\N$, define the one-sided Hausdorff distance $d_\subset$ for sets $A, B\subset \R^k$ as
		\begin{equation}
			d_\subset(A, B) = \sup_{a\in A}\inf_{b\in B} \euclOf{a-b}
			\eqfs
		\end{equation}
\end{definition}
\begin{proof}[Proof of \cref{thm:snake:deterministic:details}]\mbox{ }
	\begin{enumerate}[label=(\alph*)]
		\item
			\textbf{Definition of grid:} We define a uniform grid of points $x_1, \dots, x_m$ as illustrated in \cref{fig:grid}:
			Set $r := \delta/\sqrt{d}$.
			Let $z\in\R^d$ to be chosen later.
			Define the offset vector $o_z = z/(2r)-\lfloor z/(2r)\rfloor \in [0,1)^{d}$, where $\lfloor\cdot\rfloor$ is applied component-wise.
			Define $m_0 := \lceil(2r)^{-1}\rceil$.
			Set $m = (m_0+1)^{d-1}$.
			Choose $x_1, \dots, x_m\in\R^d$ such that $\Pi_{1}x_j = 0$ for all $j\in\nnset m$ and
			\begin{equation}
				\{\Pi_{-1}x_1, \dots, \Pi_{-1}x_m\}
				=
				\setByEle{2r\br{\Pi_{-1}o_z - \frac12\mo1 + k}}{k\in\nnzset{m_0}^{d-1}}
				\eqcm
			\end{equation}
			where $\mo 1 = (1, \dots, 1)\tr\in\R^{d-1}$.
		\item
			\textbf{Construction of null-hypothesis:}
			Set $f_0(x) := (L_0, 0, \dots, 0)\tr$ for all $x\in\R^d$. Clearly, $f_0\in\Sigma(\beta, \indset{L}{0}{m}, \Lbeta)$. Moreover,
			\begin{equation}
				U(f_0, x_j, t) = \begin{pmatrix}L_0 t\\\Pi_{-1}x_j\end{pmatrix}
				\qquad\text{and}\qquad
				\dot U(f_0, x_j, t) = L_0 \mo{e}_1
				\eqfs
			\end{equation}
			For the tube distance between $\theta_0 = (f_0,\indset{x}{1}{m},\indset{T}{1}{m})$ and the hypercube, we obtain
			\begin{align}
				d_{\tube}\brOf{\hypercube, \theta_0}
				&=
				d_\subset\brOf{[0,1]^{d-1}, \cb{\Pi_{-1} x_1,\dots, \Pi_{-1} x_m}}
				\leq
				\sqrt{d} r
				=
				\delta
				\eqfs
			\end{align}
			Thus, $\theta_0 \in \ParamSnake( \indset{L}{0}{m}, \Lbeta, \delta, \Tsum)$.
		\item
			\textbf{Construction of alternative:}
			Let $h := \hbump{d}{\beta}\in\Sigma^{d\to1}(\beta, 1)$ be the bump function from \cref{lmm:bumpandpulse} and $\rmax = \rmax\brOf{\beta, \indset{L}{0}{\ell}, \Lbeta, h, 0}$ from the same lemma.
			For $z\in\R^d$, define
			\begin{equation}
				f_{2, z, r}(x) := \Lbeta r^\beta h\brOf{\frac{x-z}{r}}
			\end{equation}
			and set
			\begin{equation}
				f_1 := f_0 + \sum_{k\in\Z^d} f_{2, z+2rk, r}\mo{e}_1
				\eqfs
			\end{equation}
			As the supports $\supp(f_{2,z+2rk,r}) \subset \ball^d(z+2rk, r)$ are disjoint, $f_1$ is well-defined.
			Choose $z$ such that $\supNormOf{f_{1}} = f_{1}(x_0)$.
			By \cref{lmm:bumpandpulse} \ref{lmm:bumpandpulse:smooth}, we have $f_{2,z,r}\in\Sigma^{d\to1}(\beta, \indset{L}{0}{\ell}, \Lbeta)$ for $r \leq r_{\ms{max}}$, which is fulfilled by $\delta\leq\delta_0$. Thus, $f_1\in\Sigma^{d\to d}(\beta, \indset{L}{0}{\ell}, \Lbeta)$.
			As
			\begin{equation}
				\Pi_{-1} \brOf{\bigcup_{k\in\Z^d} \ball^d(z+2rk, r)} \cap \cb{\Pi_{-1}x_1, \dots, \Pi_{-1}x_m} = \emptyset
				\eqcm
			\end{equation}
			we have
			\begin{equation}
				U(f_1, x_j, t) = U(f_0, x_j, t) = \begin{pmatrix}t\\\Pi_{-1}x_j\end{pmatrix}
				\eqfs
			\end{equation}
			Thus, we have $\theta_1 = (f_1, \indset{x}{1}{m}, \indset{T}{1}{m}) \in \ParamSnake(\indset{L}{0}{\ell}, \Lbeta, \delta, \Tsum)$ as before for $f_0$.
		\item
			\textbf{Difference between hypotheses at a point:}
			By the choice of $z$, we have
			\begin{align}
				\euclOf{f_0 - f_1}(x_0)
				&=
				\abs{f_{r,z}(x_0)}
				=
				\supNormOf{f_{2,z,r}}
				=
				\Lbeta r^\beta \supNormOf{h}
				=
				c_{d,\beta} \Lbeta \delta^\beta
				\eqfs
			\end{align}
		\item
			\textbf{Difference between hypotheses in $\lebesgue_p$:} We have, for $r\leq c$,
			\begin{align}
				\int_{\hypercube} \euclOf{f_1(x) - f_0(x)}^p \dl x
				&=
				\sum_{k\in\Z^d} \int_{\hypercube} \abs{f_{2,z+2rk,r}(x)}^p \dl x
				\\&\geq
				\Lbeta^p r^{\beta p + d} \sum_{\substack{k\in\Z^d\\\ball^d(z+2rk,r)\subset \hypercube}} \int_{\R^d} \abs{h(x)}^p \dl x
				\\&\geq
				c_{d, \beta, p} \br{\Lbeta r^\beta}^p
				\\&\geq
				c_{d, \beta, p} \br{\Lbeta \delta^\beta}^p
				\eqfs
			\end{align}
	\end{enumerate}
\end{proof}
\subsection{Snake Model -- Lower Bound -- Probabilistic}\label{sec:app:snake:probabilistic}

Let $d\in\N$ and $\beta\in\Rpo$.
Set $\ell := \llfloor\beta\rrfloor$.
Let $L_0, \dots, L_\ell, \Lbeta\in\Rpp$.
Let $m\in\N$ and $\xany 1, \dots, \xany m \in \R^d$, $n_1, \dots, n_m\in\N$, $T_1, \dots, T_m\in\Rpp$.
Set $\Tsum := \sum_{j\in\nnset m} T_j$.
Set $n := \sum_{j\in\nnset{m}} n_j$.
Set $\setIdx := \setByEleInText{(j,i)}{i\in\nnset{n_j}, j\in\nnset{m}}$.
Let $t_{j,i}\in\Rp$ for $(j,i)\in\setIdx$ such that $0 \leq t_{j,1} \leq \dots \leq t_{j,n_j} = T_j$.
Let $\delta\in\Rpp$.
Let $\ParamSnake = \ParamSnake( \indset{L}{0}{\ell}, \Lbeta, \delta, \Tsum)$ as in \cref{def:snake:smoothnessclass}.
For $\ptrue = (\ftrue,\indset{x}{1}{m},\indset{T}{1}{m}) \in \ParamSnake$ consider the ODE model
\begin{equation}
	Y_{j,i} = U(\ftrue, \xany j, t_{j,i}) + \noise_{j,i}\eqcm \ (j,i)\in\setIdx\eqcm
\end{equation}
with independent and identically distributed $\R^d$-valued noise variables $\noise_{j,i}$ such that $\Eof{\noise_{j,i}} = 0$. Denote the noise distribution as $P^\noise$, i.e., $\noise_{j,i}\sim P^\noise$.
We know $\xany{j}$ and $t_{j,i}$, and observe $Y_{j,i}$, but $f_{\ptrue} = \ftrue$ is unknown and to be estimated on the domain $\hypercube$.
\begin{theorem}\label{thm:snake:probabilistic:details}
	Assume \assuRef{Noise} for $P^\noise$ with constant $\const{noise}\in\Rpp$ and
	\assuRef{CoverTime} with constant $\const{cvrtm}\in\Rpp$.
	Assume
	\begin{equation}
		\delta \leq \sqrt{d} \qquad\text{and}\qquad \Tsum \geq c_{d,\beta} \delta^{-(d-1)} L_0^{-1}
	\end{equation}
	and
	\begin{equation}\label{eq:snake:noise:rmaxAssumption}
		\min_{k\in\nnzset\ell} L_k \geq c_\beta L_\beta
		\eqfs
	\end{equation}
	Furthermore, assume
	\begin{equation}
		\br{\frac{\const{cvrtm}}{L_0}}^{2(\beta+1)+d} \br{\frac n{\Tsum}}^{2\beta+d+1}
		\geq
		c_{d,\beta}  \const{noise} \const{cvrtm} \Lbeta^2  L_0^{-3} \delta^{-(d-1)}
		\geq
		\frac{\Tsum}{n}
		\eqfs
	\end{equation}
	\begin{enumerate}[label=(\roman*)]
		\item
		Let $x_0\in\hypercube$. Then
		\begin{equation}
			\inf_{\festi} \sup_{\theta\in\ParamSnake}
			\Pr_{\ptrue}\brOf{\euclOf{\festi - f_{\ptrue}}(x_0) \geq C_1 \br{\delta^{-(d-1)} \frac{n}{\Tsum}}^{-\frac\beta{2(\beta+1)+d}}}
			\geq \frac14
		\end{equation}
		where
		\begin{equation}
			C_1 :=  c_{d,\beta} \Lbeta \br{\const{noise} \const{cvrtm} \Lbeta^2  L_0^{-3}}^{-\frac\beta{2(\beta+1)+d}}
			\eqfs
		\end{equation}
		\item
		We have
		\begin{equation}
			\inf_{\festi} \sup_{\ptrue\in\ParamSnake}
			\Pr_{\ptrue}\brOf{\sup_{x\in\hypercube}\euclOf{\festi - f_{\ptrue}}(x) \geq C_1 \br{\frac{\Tsum^{-1} \delta^{-(d-1)} n}{\log\brOf{\Tsum^{-1} \delta^{-(d-1)} n}}}^{-\frac\beta{2(\beta+1)+d}}}
			\geq \frac14
		\end{equation}
		where
		\begin{equation}
			C_1 := c_{d,\beta} \Lbeta \br{\const{noise} \const{cvrtm} \Lbeta^2  L_0^{-3}}^{-\frac\beta{2(\beta+1)+d}}
		\end{equation}
		if
		\begin{equation}
			\Tsum^{-1} \delta^{-(d-1)} n \geq c_{d,\beta} \max\brOf{\const{cvrtm}^{-1} \Lbeta^{-2}  L_0^{3},\, \const{cvrtm} \Lbeta^2  L_0^{-3}} \max\brOf{1,\const{noise}^{4}}
			\eqfs
		\end{equation}
		\item
		We have
		\begin{equation}
			\inf_{\festi} \sup_{\ptrue\in\ParamSnake}
			\Pr_{\ptrue}\brOf{\LpNormOf{p}{\xdomain}{\euclOf{\festi - f_{\ptrue}}} \geq  C_1 \br{\Tsum^{-1} \delta^{-(d-1)} n}^{-\frac\beta{2(\beta+1)+d}}}
			\geq \frac14
		\end{equation}
		where
		\begin{equation}
			C_1 :=  c_{d, \beta, p} \br{\const{noise}  \const{cvrtm} \Lbeta^2  L_0^{-3}}^{-\frac\beta{2(\beta+1)+d}}
		\end{equation}
		if
		\begin{equation}
			\Tsum^{-1} \delta^{-(d-1)} n \geq c_{d,\beta} \const{cvrtm}^{-1} \Lbeta^{-2}  L_0^{3}
			\eqfs
		\end{equation}
	\end{enumerate}
	The infima range over all estimators $\festi$ of $f_{\ptrue}$ based on the observations $Y_{\setIdx}$.
\end{theorem}
\begin{proof}[Proof of \cref{thm:snake:probabilistic:details}]
	We want to apply \cref{thm:lower:master}. For this, we first map the symbols used there with the objects we are dealing with here. Then, we prove that the assumptions made in \cref{thm:lower:master} are fulfilled. Finally, we present the result of the application of \cref{thm:lower:master}.

    \textbf{Mapping Symbols.}
    See \cref{tbl:symbolsMaster} for an overview.
    We start with a consecutive indexing of our observations: For $k\in\nnset{n}$, denote by $j_k\in\nnset{m}$ and $i_k\in\nnset{n_{j_k}}$ indices of observations, such that all observations are enumerated, i.e., $\#\{(i_k,j_k):\,k\in\nnset{n}\} = n$. Next, we set the dimensions as
    \begin{equation}
        \dm q := \dm u := \dm f := d
    \end{equation}
    and the domain of interest as $\xdomain := [0,1]^d$. The parameter space $\Theta$ is given by  $\ParamSnake( \indset{L}{0}{\ell}, \Lbeta, \delta, \Tsum)$ from \cref{def:snake:smoothnessclass}. For each $\theta\in\Theta$ with $\theta = (f,\indset{x}{1}{m},\indset{T}{1}{m})$, we define the following attributes:
    \begin{enumerate}[label = (\roman*)]
        \item Location: We set $q_k(\theta):= U(f,x_{j_k},t_{j_k,i_k})$: The location of the $k$-th observation is $j_k$-th trajectory at time $t_{j_k,i_k}$, which is driven by the model function $f$ and has initial conditions $x_{j_k}$.
        \item Observed Objects: We identify the observed objects with the location attribute, i.e., $u_k(\theta):=q_{k}(\theta)$.
        \item Function to Estimate: We set $f_\theta := f|_{\hypercube}$, the model function $f$ restrict to the domain $\hypercube$.
    \end{enumerate}
    \textbf{Verifying Assumptions.}
    Next, we verify \cref{ass:lower:general}.  Set $\rho_n^+ := \min(\frac12, \rmax)$ and $\rho_n^- := \frac12 \const{cvrtm}^{-1} \frac{L_0 \Tsum}n$.
    \begin{enumerate}
        \item[\ref{ass:lower:general:ref}]
		    We define the reference function $g:\R^d\to\R^d$ for any $x\in\R^d$ as
		    \begin{equation}
		        g(x) := \Lbeta \hpulse{d}{\beta}\brOf{x}\mo{e}_2\eqcm
		    \end{equation}
		    with $\hpulse{d}{\beta}$ as in \cref{lmm:bumpandpulse}. By \cref{lmm:bumpandpulse}, we have $\operatorname{supp}(g)\subseteq\ball^d(0,1)$.
		    Set $v_0 := L_0 \mo{e}_1\in\R^d$.
		    We set
		    \begin{align}
		    	f_{0}(x) &= v_0 = L_0 \mo{e}_1\eqcm\\
		    	f_{z,r}(x) &= v_0 + r^\beta g\brOf{\frac{x-z}{r}} = L_0 \mo{e}_1 + \Lbeta r^\beta\hpulse{d}{\beta}\brOf{\frac{x-z}{r}}\mo{e}_2\eqfs
		    \end{align}
		    These functions are smooth in the appropriate sense as shown in \cref{lmm:bumpandpulse}: If $r\in(0, \rmax]$, where $\rmax$ is given in the lemma, we have $f_0, f_{z,r}\in \Sigma^{d\to d}(\beta, \indset{L}{0}{\ell}, \Lbeta)$.
		    The condition on $r$ is fulfilled for $r \leq \rho_n^+$.

		    Let $\xdelta1,\dots,\xdelta m$ be the sequence of initial conditions as constructed in \cref{lem:snake:lower:Covering} with
		    \begin{equation}
		    	m := \br{\left\lceil c_{d,\beta} \delta^{-1} \right\rceil + 1}^{d-1} \leq c_{d,\beta} \delta^{-(d-1)}
		    	\eqcm
		    \end{equation}
		    where the inequality is due to $\delta$ being bounded from above, specifically we assume $\delta \leq \sqrt d$.
		    As we assume $\Tsum \geq c_{d,\beta} L_0^{-1} \delta^{-(d-1)}$, we can choose $T_j := c_{d,\beta} L_0^{-1}$ for $j\geq 2$ and $T_1 = \Tsum - \sum_{j=2}^m T_j \geq c_{d,\beta} L_0^{-1}$.
		    \cref{lem:snake:lower:Covering} now implies $\theta_0 := (f_0,\indset{\xdeltaSymb}{1}{m},\indset{T}{1}{m})  \in \ParamSnake$ and $\theta_{z,r} := (f_{z,r},\indset{\xdeltaSymb}{1}{m},\indset{T}{1}{m}) \in \ParamSnake$ for
		    \begin{equation}
		    	0 < r \leq c_{d,\beta} \min\brOf{1,\,\br{\frac{L_0}\Lbeta}^{\frac1\beta}}
		    	\eqcm
		    \end{equation}
		    which is fulfilled for $r \leq \rho_n^+$.

		    To finalize \cref{ass:lower:general} \ref{ass:lower:general:ref}, we have to show $q_k(\theta_{z,r}) = q_k(\theta_0)$ for all $k\in\nnset n$ whenever $q_k(\theta_0)\in\R^{d}\setminus \ball^d(z, r)$. Then the same condition on the observed objects $u_k$ is also fulfilled as $u_k = q_k$.

		    For the location attribute, we have
		    \begin{equation}
		         q_k(\theta_0) = U(f_0 , \xdelta{j_k}, t_{j_k, i_k}) = \xdelta{j_k} + L_0 t_{j_k, i_k} \mo e_1
		    \end{equation}
		    and
		    \begin{equation}
		         q_k(\theta_{z,r}) = U(f_{z,r}, \xdelta{j_k}, t_{j_k, i_k}) = \xdelta{j_k} + L_0 t_{j_k, i_k} \mo e_1 + w_{z,r}^{j_k}(t_{j_k, i_k})\mo e_2\eqcm
		    \end{equation}
		    where
		    \begin{equation}\label{eq:snake:lower:proba:detail:wzr}
		        w_{z,r}^j(t) := \Lbeta r^\beta\int_0^{t}\hpulse{d}{\beta}\brOf{\frac{U(f_{z,r}, \xdelta{j}, s)-z}{r}}\mathrm{d}s.
		    \end{equation}
		    Now assume that $q_k(\theta_0)\notin \ball^d(z,r)$, which in turn implies
		    \begin{equation}\label{eq:notinball}
		        \xdelta{j} + L_0 t_{j, i} \mo e_1 \notin \ball^d(z,r).
		    \end{equation}
		    for $j = j_k$ and $i = i_k$.
		    Then three different cases can occur:
		    \begin{itemize}
		        \item \eqref{eq:notinball} and $\euclOf{\Pi_{-1} \br{\xdelta{j} - z}}\geq \delta$: In this situation, the resulting trajectory will not be inside of the support of $g$ at any time, such that $ w_{z,r}^j(t) = 0$ for all $t\in\R$.
		        \item \eqref{eq:notinball}, $\euclOf{\Pi_{-1} \br{\xdelta{j} - z}}\leq \delta$, and $L_0 t_{j,i} + \Pi_1( \xdelta{j}) \leq \Pi_1(z)$: In this situation, the trajectory has not yet reach the support of $g$, so that no displacement into the second dimension has taken place. In other words, $w_{z,r}(t)=0$ for all $t \leq t_{j,i}$.
		        \item \eqref{eq:notinball}, $\euclOf{\Pi_{-1} \br{\xdelta{j} - z}}\leq \delta$, and $L_0 t_{j,i} + \Pi_1( \xdelta{j}) \geq \Pi_1(z)$: In this case, the trajectory has run completely through the support of $g$. Due to the symmetry property of the solution of the initial value problem, see \cref{lem:snake:lower:ODEprop}, the solution after passing through the support of $g$ behaves as it did before the pass, i.e., it is zero in the second dimension. In other words, $w_{z,r}(t)=0$ for all $t \geq t_{j,i}$.
		    \end{itemize}
		    All in all, we have $q_k(\theta_{z,r}) = q_k(\theta_{0})$ if $q_k(\theta_0)\notin \ball^d(z,r)$.
        \item[\ref{ass:lower:general:obs}]
        	Observing that $U(f_{z,r}, \xdelta{j_k}, t_{j_k, i_k})$ and $U(f_0, \xdelta{j_k}, t_{j_k, i_k})$ only differ in the second component, we obtain
        	\begin{align}
        		\psi_{n}(r)
        		&=
        		\sup_{z\in\R^d} \sup_{k\in\nnset{n}} \euclOf{u_k(\theta_{z,r}) - u_k(\theta_0)}
        		\\&=
        		\sup_{z\in\R^d} \sup_{k\in\nnset{n}} \euclOf{U(\theta_{z,r}, \xdelta{j_k}, t_{j_k, i_k}) - U(\theta_0, \xdelta{j_k}, t_{j_k, i_k})}
        		\\&= \sup_{z\in\R^d} \sup_{k\in\nnset{n}} |w_{z,r}^{j_k}(t_{j_k,i_k})|
        	\end{align}
        	with $w_{z,r}^j$ from \eqref{eq:snake:lower:proba:detail:wzr}.
            Next, for $a\in\R$ and $b>0$ we have
            \begin{align}
            	\int_0^{t} \tilde K_{\beta}\pr\brOf{a + b s} \dl s
            	&=
            	\frac1b\int_{a}^{a+bt} \tilde K_{\beta}\pr\brOf{s} \dl s
            	\\&\leq
            	\frac1b \supNormOf{\tilde K_{\beta}\pr} \min(2, bt)
            	\\&\leq
            	c_\beta \frac1b
            \end{align}
            as $\supp(\tilde K_{\beta}) \subset [-1, 1]$. Thus,
            \begin{align}
                w_{z,r}^{j}(t)
                &=  \Lbeta r^\beta\int_0^{t}\hpulse{d}{\beta}\brOf{\frac{U(f_{z,r} , \xdelta{j}, s)-z}{r}}\mathrm{d}s\\
                & \leq \Lbeta r^\beta  \tilde K_{\beta}(0)\int_0^{t} \tilde K_{\beta}\pr\brOf{
                	\frac1r\br{\Pi_1 \xdelta{j} + L_0 s - \Pi_1 z}
            	} \dl s
            	\\& \leq
            	c_\beta \Lbeta L_0^{-1} r^{\beta+1}
                \eqcm
            \end{align}
            independent of $z\in\R^d$, $t\in\Rpp$, and $k\in\nnset{n}$, such that
            \begin{equation}
                \psi_{n}(r)
                \leq
                c_{\beta} \Lbeta  L_0^{-1} r^{\beta+1}
                \eqfs
            \end{equation}
            Now, for $\chi_{n}(r)$, we have
            \begin{align}
         		\chi_{n}(r)
         		&= \sup_{z\in\R^d} \#\setByEle{k\in\nnset n}{q_k(\theta_{z,r}) \in\ball^d(z,r)}
         		\\&= \sup_{z\in\R^d} \sum_{j\in\nnset{m}} \sum_{i\in\nnset{n_j}} \indOfOf{\ball^d(z,r)}{\xdelta{j} + L_0 t_{j, i} \mo e_1 + w_{z,r}^{j}(t_{j, i})\mo e_2}
         		\eqfs
            \end{align}
            Initial conditions outside the projected ball, $\Pi_{-1}\xdelta{j} \not\in \ball^{d-1}(\Pi_{-1} z, r)$, lead to linear trajectories that never intersect the ball $\ball^d(z, r)$. For $\Pi_{-1}\xdelta{j} \in \ball^{d-1}(\Pi_{-1} z, r)$, the trajectory $U(f_{z, r}, \xdelta j, t)$ can only intersect the ball $\ball^d(z, r)$ if $\Pi_1 \xdelta j + L_0 t \in [\Pi_1 z - r, \Pi_1 z + r]$.
            Hence, with \assuRef{CoverTime}, we have
            \begin{align}
              	\chi_{n}(r)
              	&\leq
              	\sup_{z\in\R^d} \sum_{j\colon \Pi_{-1} x_j^\delta\in\ball^{d-1}(\Pi_{-1}  z,r)} \sum_{i\in\nnset{n_j}} \indOfOf{[L_0^{-1}(\Pi_1(z - \xdelta j) - r),\, L_0^{-1}(\Pi_1(z - \xdelta j) + r)]}{t_{j, i}}
          		\\&\leq
          		\sup_{z\in\R^d} \sum_{j\colon \Pi_{-1} x_j^\delta\in\ball^{d-1}(\Pi_{-1}  z,r)} \max\brOf{1, \const{cvrtm}\frac{2r}{L_0\Tsum}n}
          		\\&\leq
          		c_d\const{cvrtm} L_0^{-1} \Tsum^{-1} m n r^d
          		\eqcm
            \end{align}
            where we use $1 \leq \const{cvrtm}\frac{2r}{L_0 \Tsum}n$ due to $r \geq \rho_n^- = \frac12 \const{cvrtm}^{-1} \frac{L_0 \Tsum}n$.
          	As $m \leq c_{d,\beta} \delta^{-(d-1)}$, we have
         	\begin{align}
         		\psi_{n}(r)^2 \chi_{n}(r)
    			&\leq
    			c_{d,\beta} \const{cvrtm} \Lbeta^2 L_0^{-3}  \Tsum^{-1} m n r^{2(\beta+1)+d}
    			\\&\leq
    			c_{d,\beta} \const{cvrtm} \Lbeta^2 L_0^{-3} \Tsum^{-1} \delta^{-(d-1)} n r^{2(\beta+1)+d}
    			\\&\leq
    			a_{n} r^{\gamma},
         	\end{align}
           with
           \begin{equation}
               a_n := c_{d,\beta} \const{cvrtm} \Lbeta^2  L_0^{-3}  \Tsum^{-1} \delta^{-(d-1)} n
               \qquad\text{and}\qquad
               \gamma := 2(\beta+1)+d
               \eqfs
           \end{equation}
        \item[\ref{ass:lower:general:pack}]
	        As $r \leq \rho_n^+ \leq \frac12$ and $\xdomain=[0,1]^d$ we can set $\const{pack} = c_d$.
        \item[\ref{ass:lower:general:add}]
	        For any $J\in\N$ and $z_1,\dots,z_J\in\R^d$ fulfilling
	        \begin{equation}
	            \inf_{i,j\in\nnset{J}}\euclOf{z_i-z_j}\geq 2r
	        \end{equation}
	        we have
	        \begin{align}
	            \sum_{j\in\nset{1}{J}}\left(f_{z_j,r}-f_{0}\right) &=  \sum_{j\in\nset{1}{J}} \Lbeta r^\beta\hpulse{d}{\beta}\brOf{\frac{x-z_j}{r}}\mo{e}_2\\
	            &= \left(0,\Lbeta r^\beta\sum_{j\in\nset{1}{J}}\hpulse{d}{\beta}\brOf{\frac{x-z_j}{r}},0,\dots,0\right)^T
	        \end{align}
	        Now set
	        \begin{equation}
	            f_{\Tilde{\theta}} := \Tilde{f}:= \left(1,\Lbeta r^\beta\sum_{j\in\nset{1}{J}}\hpulse{d}{\beta}\brOf{\frac{x-z_j}{r}},0,\dots,0\right)^T.
	        \end{equation}
	        As $\inf_{i,j\in\nnset{J}}\euclOf{z_i-z_j}\geq 2r$ all the arising bumps of $f_\theta $ have disjoint support. Choosing
	        \begin{equation}
	            \Tilde{\theta}:=(\Tilde{f},\indset{\xdeltaSymb}{1}{m},\indset{T}{1}{m}),
	        \end{equation}
	        with the same family of initial values as before we can repeat the proof of \cref{lem:snake:lower:Covering} for all $j\in\nnset{J}$ to obtain the covering argument, which yields $\Tilde{\theta}\in\Theta$.
	        Due to the linearity of the integral, simple computations show
	        \begin{equation}
	            q_k(\Tilde{\theta}) - q_k(\theta_0) = \sum_{j\in\nnset{J}} \br{q_k(\theta_{z_j, r}) - q_k(\theta_0)}
	        \end{equation}
	        for all $k\in\nnset n$. Due to the equality of location and observation, we further have
	        \begin{equation}
	            u_k(\Tilde{\theta})- u_k(\theta_0) = \sum_{j\in\nnset{J}}\br{u_k(\theta_{z_j, r})- u_k(\theta_0)}
	        \end{equation}
	         for all $k\in\nnset n$.
    \end{enumerate}
    \textbf{Presenting Results.}
    After proving \cref{ass:lower:general}, we apply \cref{thm:lower:master} and evaluate the results. Recall
    \begin{align}
    	\supNormOf{g} &= c_{d,\beta} \Lbeta\eqcm \\
    	a_{n}&= c_{d,\beta} \const{cvrtm} \Lbeta^2  L_0^{-3}  \Tsum^{-1} \delta^{-(d-1)} n \eqcm \\
    	\zeta &= \beta \eqcm\\
    	\gamma &= 2(\beta+1)+d\eqcm\\
    	\rho_n^+ &= \min(\frac12, \rmax)\eqcm\\
    	\rho_n^- &= \frac12 \const{cvrtm}^{-1} \frac{L_0 \Tsum}n
    	\eqfs
    \end{align}
    \begin{enumerate}[label=(\roman*)]
	    \item
		    By \cref{thm:lower:master} \ref{thm:lower:master:pointwise}, we obtain the following lower bound.
		    We have
		    \begin{align}
		    	\frac14
		    	&\leq
		    	\inf_{\festi} \sup_{\ptrue\in\Theta} \Pr_\theta\brOf{\euclOf{\festi - f_{\ptrue}}(x_0) \geq \frac12 \supNormOf{g} \br{2 \const{noise} a_{n}}^{-\frac\zeta{\gamma}}}
		    	\\&\leq
		    	\inf_{\festi} \sup_{\ptrue\in\ParamSnake} \Pr_\theta\brOf{\euclOf{\festi - f_{\ptrue}}(x_0) \geq C_1 \br{ \Tsum^{-1} \delta^{-(d-1)} n}^{-\frac\beta{2(\beta+1)+d}}}
		    \end{align}
		    with
		    \begin{equation}
		    	C_1 := c_{d,\beta} \Lbeta \br{\const{noise} \const{cvrtm} \Lbeta^2  L_0^{-3}}^{-\frac\beta{2(\beta+1)+d}}
		    	\eqfs
		    \end{equation}
		    The conditions for this bound to hold are:
		    \begin{enumerate}[label=(\Roman*)]
		    	\item
		    	\begin{align}
		    		&
		    		(\rho_n^-)^{-\gamma} \geq 2 \const{noise}a_{n} \geq (\rho_n^+)^{-\gamma}
		    		\\\Leftrightarrow\qquad&
		    		\br{\const{cvrtm}^{-1} \frac{L_0 \Tsum}n}^{-(2(\beta+1)+d)} 
		    		\geq
		    		c_{d,\beta}  \const{noise} \const{cvrtm} \Lbeta^2  L_0^{-3}  \Tsum^{-1} \delta^{-(d-1)} n  
		    		\\&\phantom{\br{\const{cvrtm}^{-1} \frac{L_0 \Tsum}n}^{-(2(\beta+1)+d)}}\geq
		    		\br{\min(\frac12, \rmax)}^{-(2(\beta+1)+d)}
		    		\\\Leftrightarrow\qquad&
		    		\br{\frac n{\Tsum}}^{2\beta+d+1} \geq c_{d,\beta}  \const{noise} \const{cvrtm} \Lbeta^2  L_0^{-3}   \br{\const{cvrtm}^{-1} L_0}^{(2(\beta+1)+d)} \delta^{-(d-1)}
		    		\\\quad \text{and} \quad&
		    		c_{d,\beta}  \const{noise} \const{cvrtm} \Lbeta^2  L_0^{-3} \geq  \Tsum \delta^{d-1} n^{-1}
		    		\eqfs
		    	\end{align}
		    	where we used \eqref{eq:snake:noise:rmaxAssumption} to bound $\rmax \geq c_\beta$.
		    \end{enumerate}
	    \item
		    By \cref{thm:lower:master} \ref{thm:lower:master:sup}, we obtain the following lower bound.
		    We have
		    \begin{align}
		    	\frac14
		    	&\leq
		    	\inf_{\festi} \sup_{\ptrue\in\Theta}
		    	\Pr_\theta\brOf{\sup_{x\in\xdomain}\euclOf{\festi - f_{\ptrue}}(x) \geq \frac12 \supNormOf{g} \br{\frac{12 \gamma \const{noise}}{\dm q} a_{n} \log(a_{n})^{-1}}^{-\frac\zeta{\gamma}}}
		    	\\&\leq
		    	\inf_{\festi} \sup_{\ptrue\in\ParamSnake}
		    	\Pr_\theta\brOf{\sup_{x\in\xdomain}\euclOf{\festi - f_{\ptrue}}(x) \geq C_1 \br{\frac{\Tsum^{-1} \delta^{-(d-1)} n}{\log(A_1  \Tsum^{-1} \delta^{-(d-1)} n)}}^{-\frac\beta{2(\beta+1)+d}}}
		    \end{align}
		    with $A_1 := c_{d,\beta} C_0$, $C_0:=\const{cvrtm} \Lbeta^2  L_0^{-3}$ and
		    \begin{equation}
		    	C_1 := c_{d,\beta} \Lbeta \br{\const{noise} C_0}^{-\frac\beta{2(\beta+1)+d}}\eqfs
		    \end{equation}
		    If additionally $\Tsum^{-1} \delta^{-(d-1)} n \geq A_1$, then
		    \begin{equation}
		    	\frac14
		    	\leq
		    	\inf_{\festi}\sup_{\ptrue\in\ParamSnake}
		    	\Pr_{f_{\ptrue}}\brOf{\sup_{x\in\hypercube}\euclOf{\festi - f_{\ptrue}}(x) \geq C_1 \br{\frac{\Tsum^{-1} \delta^{-(d-1)} n}{\log\brOf{\Tsum^{-1} \delta^{-(d-1)} n}}}^{-\frac\beta{2(\beta+1)+d}}}
		    	\eqfs
		    \end{equation}
		    The conditions for this bound to hold are:
		    \begin{enumerate}[label=(\Roman*)]
		    	\item
		    	\begin{align}
		    		&
		    		a_{n} \geq \max\brOf{2\eqcm\br{\frac{12 \gamma \const{noise}}{\dm q\const{pack}^{\frac{\gamma}{\dm q}}}}^{4}}
		    		\\\Leftrightarrow\qquad&
		    		c_{d,\beta} C_0 \Tsum^{-1} \delta^{-(d-1)} n \geq \max\brOf{2\eqcm\br{\frac{12 (2(\beta+1)+d) \const{noise}}{d \const{pack}^{\frac{2(\beta+1)+d}{d}}}}^{4}}
		    		\\\Leftrightarrow\qquad&
		    		\Tsum^{-1} \delta^{-(d-1)} n \geq c_{d,\beta} C_0^{-1} \max\brOf{1,\const{noise}^{4}}
		    		\eqfs
		    	\end{align}
		    	\item
		    	\begin{align}
		    		&
		    		(\rho_n^-)^{-\gamma} \geq \frac{12 \gamma \const{noise}}{\dm q} a_{n} \geq (\rho_n^+)^{-\gamma}
		    		\\\Leftrightarrow\qquad&
		    		\br{\frac{\const{cvrtm}}{L_0}}^{2(\beta+1)+d} \br{\frac n{\Tsum}}^{2\beta+d+1} \geq c_{d,\beta}  \const{noise} C_0 \delta^{-(d-1)} \geq \frac{\Tsum}{n}
		    		\eqfs
		    	\end{align}
		    	\item
		    	\begin{align}
		    		&
		    		\Tsum^{-1} \delta^{-(d-1)} n \geq A_1
		    		\\\Leftrightarrow\qquad&
		    		\Tsum^{-1} \delta^{-(d-1)} n \geq c_{d,\beta} C_0
		    	\end{align}
		    \end{enumerate}
	    \item
		    By \cref{thm:lower:master} \ref{thm:lower:master:lp}, we obtain the following lower bound.
		    We have
		    \begin{align}
		    	\frac14
		    	&\leq
		    	\inf_{\festi} \sup_{\ptrue\in\Theta}
		    	\Pr_{\ptrue}\brOf{\LpNormOf{p}{\xdomain}{\euclOf{\festi - f_{\ptrue}}} \geq  2^{-1-3/p} \const{pack}^{1/p} \LpNormOf{p}{\R^{\dm q}}{g} \br{ 36 \const{noise} a_{n}}^{-\frac\zeta{\gamma}}}
		    	\\&\leq
		    	\inf_{\festi} \sup_{\ptrue\in\ParamSnake}
		    	\Pr_{\ptrue}\brOf{\LpNormOf{p}{\xdomain}{\euclOf{\festi - f_{\ptrue}}} \geq  C_1 \br{\Tsum^{-1} \delta^{-(d-1)} n}^{-\frac\beta{2(\beta+1)+d}}}
		    \end{align}
		    with
		    \begin{equation}
		    	C_1 := c_{d, \beta, p} \br{\const{noise}  \const{cvrtm} \Lbeta^2  L_0^{-3}}^{-\frac\beta{2(\beta+1)+d}}
		    	\eqfs
		    \end{equation}
		    The conditions for this bound to hold are:
		    \begin{enumerate}[label=(\Roman*)]
		    	\item
		    	\begin{align}
		    		&
		    		36 \const{noise} a_{n} \geq \br{\frac18 \const{pack}}^{-\frac{\gamma}{\dm q}}
		    		\\\Leftrightarrow\qquad&
		    		\Tsum^{-1} \delta^{-(d-1)} n \geq c_{d,\beta} \const{cvrtm}^{-1} \Lbeta^{-2}  L_0^{3}
		    	\end{align}
		    	\item
		    	\begin{align}
		    		&
		    		(\rho_n^-)^{-\gamma} \geq 36 \const{noise}a_{n} \geq (\rho_n^+)^{-\gamma}
		    		\\\Leftrightarrow\qquad&
		    		\br{\frac{\const{cvrtm}}{L_0}}^{2(\beta+1)+d} \br{\frac n{\Tsum}}^{2\beta+d+1} \geq c_{d,\beta}  \const{noise} C_0 \delta^{-(d-1)} \geq \frac{\Tsum}{n}
		    		\eqfs
		    	\end{align}
		    \end{enumerate}
	\end{enumerate}
\end{proof}

%% file: sec_app_run.tex
\section{Single Trajectory}\label{sec:app:run}
The snake parameter class of \cref{def:snake:smoothnessclass} allows for an arbitrary number of trajectories $m\in\N$. Furthermore, the constructions in the proofs of \cref{thm:snake:probabilistic} and \cref{thm:snake:deterministic} use $m = c_{d, \beta} \delta^{-(d-1)}$ initial conditions. It seems natural to ask whether the lower bound also holds when one is restricted to less trajectories.
We want to argue that the same lower bound is true even for $m = 1$.

First, note that the $m$ initial conditions used in the lower bound proofs form a regular grid in all but the first dimension, where they are equal. Furthermore, the trajectories eventually end up in states that are equal to the initial conditions shifted along the first dimension. In order to show that we only require $m=1$ trajectories, we want to connect the $m$ end-points with the starting points via trajectories outside the domain of interest that are solutions to an ODE with a smooth model function.

We make the construction explicit for $d=2$ and $\beta=1$.
\begin{lemma}\label{lmm:singleTraj}
	Let $\delta \in (0, 1/2]$.
	There is a function $f\colon \R^2 \to \R^2$ with solution $u \colon \R \to \R^2$ of the ODE $\dot u(t) =  f(u(t))$ for initial conditions $u(0) = 0\in\R^2$ with following properties:
	\begin{enumerate}[label=(\roman*)]
		\item $f_{|[0,1]^2} = (1, 0, \dots, 0)\tr$
		\item $f \in \Sigma^{2\to2}(\beta, L_0, L_1)$ with $\beta = 1$, $L_0 = 2$, and $L_1 = 5$.
		\item Set $K = \lceil\frac1\delta\rceil$. There is $T \in \Rpp$ with $T \leq c \delta^{-1}$ such that: There are times $s_1, \dots, s_K \in [0, T]$ such that
		\begin{equation}
			\cb{u(s_j) \colon j\in\{1,\dots, K\}} = \cb{\br{0, \frac{k}K}\tr \colon k\in\{0,\dots, J\}}
			\eqfs
		\end{equation}
	\end{enumerate}
\end{lemma}
\begin{figure}[H]
	\centering
	\includegraphics[width= 0.75\textwidth]{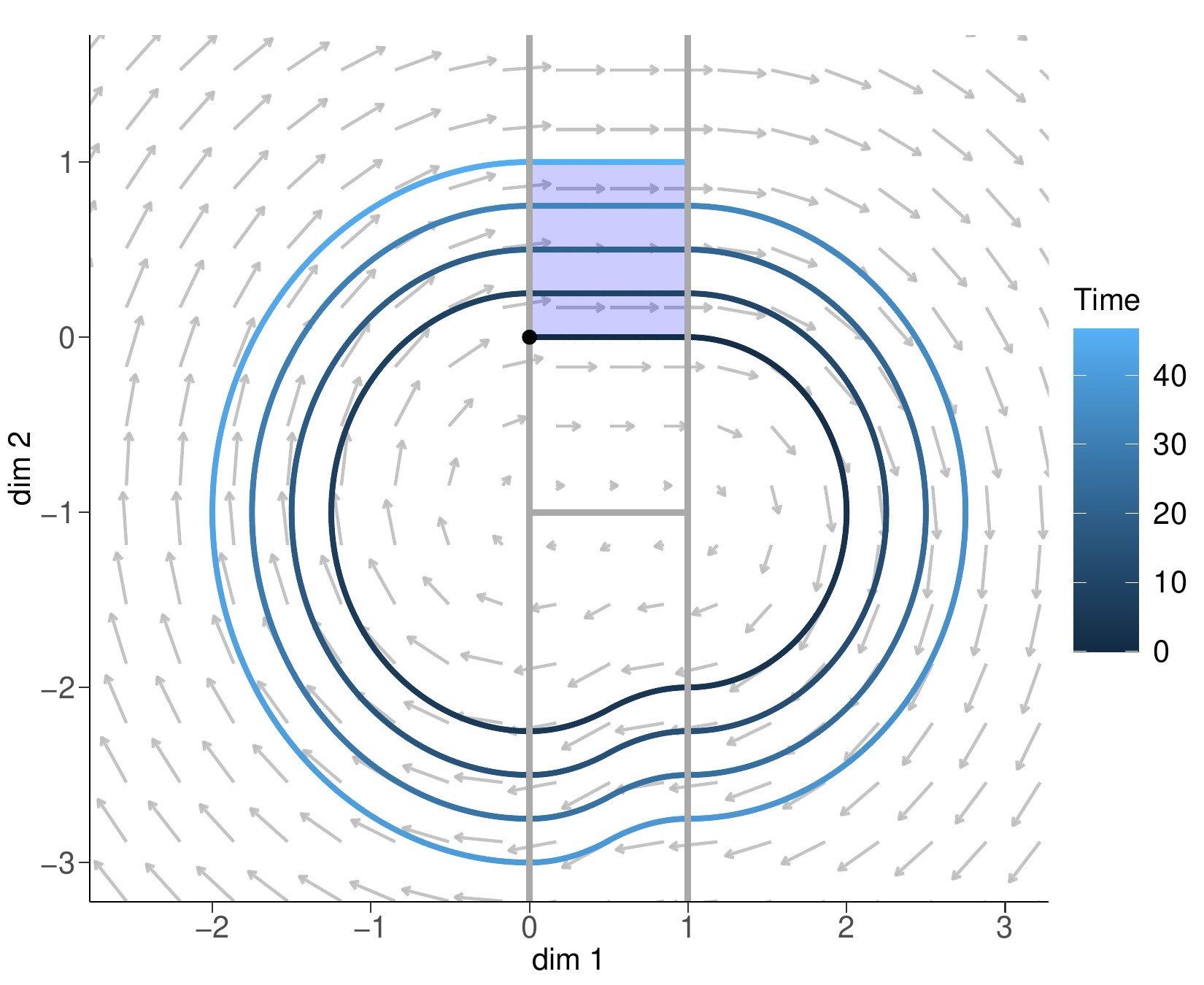}
	\label{fig:run}
	\caption{A grid of states $u(t_{2i-1})$ with $\Pi_1 u(t_{2i-1}) = 1$, $i\in\nnset4$, and another gird of states $u(t_{2i})$ with $\Pi_1 u(t_{2i}) = 0$ are connected via a single trajectory $u$. The model function $f$ is depicted as arrows. The black dot marks $u_0 = (0, 0)$. The domain of interest $[0, 1]^2$ is colored in purple. The image of the solution $u\colon[0,T] \to \R^2$ of $\dot u(t) = f(u(t))$ with initial conditions $u(0) = u_0$ is shown in black to blue. Gray lines mark the partitioning in the construction of $f$, see \eqref{eq:2dlipschtiz}.}
\end{figure}
The proof of the lemma is given below. By adding bumps or pulses to the constructed model function $f$ and scaling it, we can create the alternative model functions required for the proofs of the lower bounds for the snake model.

Recall \cref{def:snake:smoothnessclass}: For $d=2$ and $\beta=1$, we have
\begin{align}
	&\ParamSnake(L_0, L_1, \delta, \Tsum)
	\\&=
	\setByEle{\theta = (f, \indset{x}{1}{m}, \indset{T}{1}{m})\in \ParamSnakeBase{2,1}{m}(L_0, L_1) }{
		\begin{array}{l}
			m\in\N,\cr
			\sum_{j\in\nnset{m}} T_j = \Tsum,\cr
			d_\tube([0,1]^2, \theta) \leq \delta
		\end{array}
	}\eqcm
\end{align}
where $\ParamSnakeBase{2,1}{m}(L_0, L_1) = \Sigma^{2\to 2}(1, L_0, L_1) \times\R^d \times \Rpp$. With \cref{lmm:singleTraj} we can show that the lower bounds so far presented for the parameter class $\ParamSnakeSymb_{2,1}(L_0, L_1, \delta, \Tsum)$ also hold for the smaller class
\begin{equation}
	\setByEle{\theta = (f, x, \Tsum)\in \ParamSnakeBase{2,1}{1}(L_0, L_1) }{
		d_\tube([0,1]^2, \theta) \leq \delta
	}\eqcm
\end{equation}
where the number of trajectories is restricted to $m=1$. Furthermore, we conjecture that similar arguments can be made for arbitrary $d \in\N_{\geq 2}$ and $\beta\in\Rpo$.
\subsection{Proof of \cref{lmm:singleTraj}}\label{sec:app:run:proof}
We start by constructing the model function $f\colon\R^2 \to \R^2$. 
The function is defined piece-wise on following partition:
\begin{align}
	A_{\ms{top}} &:= [0,1]\times [-1,\infty)\eqcm\\
	A_{\ms{bottom}} &:= [0,1]\times (-\infty, -1]\eqcm\\
	A_{\ms{left}} &:= (-\infty,0] \times \R\eqcm\\
	A_{\ms{right}} &:= [1, \infty) \times \R
	\eqfs
\end{align}
To define $f$ on each part, we need some more definitions,
\begin{align}
    z_\ms{left} &:= (0,-1)\eqcm\\
    z_\ms{right} &:= (1,-1)\eqcm\\
    r_\ms{left}(x) &:= \normof{x - z_\ms{left}}\eqcm\\
    r_\ms{right}(x) &:= \normof{x - z_\ms{right}}\eqcm\\
    r_\ms{middle}(x) &:= \abs{x_2 + 1}\eqcm\\
    M_{\ms{rotate90}} &:= \begin{pmatrix}0 & 1\\-1 & 0\end{pmatrix}\eqcm\\
    g_{\ms{speed1}}(x) &:= \frac{x}{\normof{x}}\eqcm\\
    h_\delta(x_1) &:= -4\delta\br{\frac12 - \abs{x_1 - \frac12}}\eqfs
\end{align}
Now we define the model function as
\begin{equation}\label{eq:2dlipschtiz}
    f(x) :=
    \begin{cases}
        \min\brOf{1, r_\ms{middle}(x)}\cdot (1, 0) & \text{if } x\in A_{\ms{top}}\eqcm \\
        \min\brOf{1, r_\ms{middle}(x)}\cdot (-1, h_\delta(x_1)) & \text{if } x\in A_{\ms{bottom}} \eqcm\\
        \min\brOf{1, r_\ms{left}(x)}\cdot g_{\ms{speed1}}(M_{\ms{rotate90}} (x - z_\ms{left})) & \text{if } x\in A_{\ms{left}}\eqcm\\
        \min\brOf{1, r_\ms{right}(x)}\cdot g_{\ms{speed1}}(M_{\ms{rotate90}} (x - z_\ms{right}))  & \text{if } x\in A_{\ms{right}}\eqfs
    \end{cases}
\end{equation}
The remaining part of this section shows further properties of $f$ which, together, prove  \cref{lmm:singleTraj}.
\begin{lemma}\label{lemE2}
    The function $f$ as defined in \eqref{eq:2dlipschtiz} is $\sqrt{1 + 20\delta^2}$-Lipschitz with $\normof{f}_{\infty} = \sqrt{1 + 4\delta^2}$.
\end{lemma}
\begin{proof}[Proof of \cref{lemE2}]
    First, we want to prove $\normof{f}_{\infty}^2 = 1 + (2\delta)^2$.
    As $\min(1, r_{\bullet}(x))\leq 1$ and $\normof{g_{\ms{speed1}}(x)} = 1$ for all $x\in\R^2$, we have $\normof{f(x)}\leq 1$ for $x \in A_{\ms{top}} \cup A_{\ms{left}} \cup A_{\ms{right}}$.
    On $ A_{\ms{bottom}}$, $f$ is maximized for $x_2 \leq -2$ and $x_1\in[0,1]$ such that $\abs{h_\delta(x_1)}$ is maximized. Hence,
    \begin{equation}
        \sup_{x\in\R^2} \normof{f(x)}^2 = \sup_{x\in A_{\ms{bottom}}} \normof{f(x)}^2 = 1 + \sup_{x_1\in[0, 1]}h_\delta(x_1)^2 = 1 + (2\delta)^2
        \eqfs
    \end{equation}

    Next, we consider Lipschitz-continuity of $f$. First, we show Lipschitz-continuity on each of the sets $A_\bullet$. Then we take a look the borders.

    All $r_{\bullet}$ are 1-Lipschitz and bounded by 1. This implies that $f$ is 1-Lipschitz on $A_{\ms{top}}$.

    On $x\in A_{\ms{bottom}}$, we have
    \begin{equation}
        Df(x) =
        \begin{pmatrix}
            0 & \partial_2 \min(1, r_\ms{middle}(x)) \\
            \min(1, r_\ms{middle}(x)) h_\delta\pr(x_1) & h_\delta(x_1) \partial_2 \min(1, r_\ms{middle}(x))
        \end{pmatrix}
    \end{equation}
    with
    \begin{align}
        \br{\partial_2 \min(1, r_\ms{middle}(x))}^2 &\leq 1\eqcm\\
        \br{\min(1, r_\ms{middle}(x)) h_\delta\pr(x_1)}^2 &\leq \br{4\delta}^2\eqcm\\
        \br{h_\delta(x_1) \partial_2 \min(1, r_\ms{middle}(x))}^2 & \leq \br{2\delta}^2\eqfs
    \end{align}
    As the operator norm of a matrix is bounded by the Frobenius norm, we obtain
    \begin{equation}
        \normof{D f(x)}^2 \leq 1 + \br{2\delta}^2 + \br{4\delta}^2
        \eqfs
    \end{equation}

    Because of the symmetry between $f$ on $A_{\ms{left}}$ and on $A_{\ms{right}}$, it is enough to consider only one side.
    On $x\in A_{\ms{left}}$, if $r_\ms{left}(x) \leq 1$, we have $f(x) = M_{\ms{rotate90}} (x - z_\ms{left})$, which is $1$-Lipschitz. If $r_\ms{left}(x) \geq 1$, $f(x + z_\ms{left}) = M_{\ms{rotate90}} \frac{x}{\normof{x}}$, which is also $1$-Lipschitz as the denominator is bounded below by $1$.

    To finish the proof, one can easily check that $f$ does not jump at any point of the borders between $A_{\ms{top}}, A_{\ms{bottom}}, A_{\ms{left}}, A_{\ms{right}}$.
\end{proof}
\begin{lemma}\label{lemE3}
    Let $K\in\N$ and set $\delta := \frac1K$. Set $T := 1 + \br{2 + 3 \pi} K$.
    Let $u\colon [0, T] \to \R^2$ be a solution of $\dot u(t) = f(u(t))$ with initial values $u(0) = (0, 0)$, where $f$ is given in \eqref{eq:2dlipschtiz}. For $k = 0,\dots,K$, define $s_k := \sum_{j=0}^{k-1} \br{2 + \pi \br{2 + \frac jK + \frac{j+1}K}}$.
    Then $u(s_k + t) = (t, \frac kK)$ for $t\in[0,1]$.
\end{lemma}
\begin{proof}[Proof of \cref{lemE3}]
\mbox{}
    \begin{itemize}
        \item If $u(t_0) = (0, -1+a)$ with $a\in[1,\infty)$, then $u(t_0 + 1) = (1, -1+a)$.
        \item If $u(t_0) = (1, -1+a)$ with $a\in[1,\infty)$, then $u(t_0 + \pi a) = (1, -1-a)$.
        \item If $u(t_0) = (1, -1-a)$ with $a\in[1,\infty)$, then $u(t_0 + 1) = (0, -1-a-\delta)$.
        \item If $u(t_0) = (0, -1-a)$ with $a\in[1,\infty)$, then $u(t_0 + \pi a) = (0, -1+a)$.
    \end{itemize}
    When starting at $u(t_0) = (0, -1+a)$, the time to make one round is $q(\delta, a) = 1 + \pi a + 1 + \pi (a + \delta)$ such that $u(t_0 + q(\delta, a)) = (0, -1+a+\delta)$. The total time needed to make all $K-1$ full rounds and to get to  $u(t) = (0, 1)$ plus the additional time to get from there to $u(t) = (1, 1)$ is $T$, where
    \begin{align}
        T
        &=
        \sum_{k=0}^{K-1} \br{1 + \pi (1 + \frac kK) + 1 + \pi (1 + \frac {k+1}K)} + 1
        \\&=
        2 K + 1 + \pi \sum_{k=0}^{K-1} (1 + \frac kK) + \pi \sum_{k=0}^{K-1}  (1 + \frac {k+1}K)
        \\&=
        2 K + 1 + 2 \pi K + 2 \pi \frac1K \sum_{k=0}^{K} k - \pi
        \\&=
        2 K + 1 + 2 \pi K + \pi K
        \\&=
        1 + \br{2 + 3 \pi} K
        \eqfs
    \end{align}
\end{proof}

%% file: sec_acknowledgments.tex
\textbf{Acknowledgments}
This work is funded by Deutsche Forschungsgemeinschaft (DFG, German Re- search Foundation) under Germany’s Excellence Strategy EXC-2181/1-39090098 (the Heidelberg STRUCTURES Cluster of Excellence)